%% file: manuscript_arXiv.tex
\documentclass[a4paper,10pt]{article}
\usepackage{a4wide}

\usepackage{amssymb,amsmath}
\usepackage{amsthm,sectsty}
\usepackage{graphicx}
\usepackage[english]{babel}
\usepackage{relsize}
\usepackage{color}
\usepackage{enumitem}
\usepackage{hyperref}
\usepackage{color}
\usepackage{setspace}

\def\rr{\mathbb{R}} 
\def\sF{\mathcal{F}} 
\def\Pr{\mathbb{P}} 
\def\EX{\mathbb{E}}	
\def\dif{\mathrm{d}} 
\def\landau{\mathcal{O}} 
\def\part{\mathcal{T}_{\mathbf{h}}^N} 
\def\e{\textnormal{e}} 
\def\eps{\epsilon}
\def\ntilde{\widetilde}
\def\nhat{\widehat}
\def\nbar{\overline}

\def\pindex{\mathcal{P}}

\def\vz{x}
\newcommand*{\bigtimes}{\mathop{\raisebox{-.0ex}{\hbox{\Large{$\times$}}}}}

\theoremstyle{definition}
\newtheorem{defn}{Definition}[section]

\newtheorem{rem}{Remark}[section]
\theoremstyle{plain}
\newtheorem{thm}{Theorem}[section]
\newtheorem{lem}{Lemma}[section]
\newtheorem{prop}{Proposition}[section]

\sectionfont{\large}
\subsectionfont{\normalsize}
\subsubsectionfont{\normalsize}

\numberwithin{equation}{section}

\newenvironment{enumerateI}
{\setlength{\leftmargini}{2.0em}
\begin{enumerate}[align=left]
  \setlength{\labelwidth}{0.5em}
  \setlength{\labelsep}{1.5em}
  \setlength{\itemsep}{-0pt}
  \setlength{\parsep}{0pt}  
  
  }
{\end{enumerate}}

\title{\Large\textbf{Limit theorems for infinite-dimensional piecewise deterministic Markov processes. Applications to stochastic excitable membrane models}\footnote{This work has been supported by the Agence Nationale de la Recherche through the ANR Project MANDy ``Mathematical Analysis of Neuronal Dynamics'' ANR-09-BLAN-0008.}}
\author{\normalsize\textsc{Martin G.~Riedler}\\
\small\emph{Johannes Kepler Universit\"at Linz, Institute of Stochastics}\\
\small\texttt{martin.riedler@jku.at}\\
\small\phantom{i}\\
\normalsize\textsc{Mich\`ele Thieullen}\\
\small\emph{Universit\'e Pierre et Marie Curie - Paris 6, Laboratoire de Probabilit\`es et Mod\'eles Al\`eatoires}\\
\small\texttt{michele.thieullen@upmc.fr}\\
\small\phantom{i}\\
\normalsize\textsc{Gilles Wainrib}\\
\small\emph{Universit\'e Paris 13, Laboratoire Analyse G\'eom\'etrie et Applications}\\
\small\texttt{wainrib@math.univ-paris13.fr}\\
}

\date{}

\begin{document}

\maketitle 

\begin{center}
\begin{minipage}{0.8\textwidth}
\textbf{Abstract:} We present limit theorems for a sequence of Piecewise Deterministic Markov Processes (PDMPs) taking values in a separable Hilbert space. This class of processes provides a rigorous framework for stochastic spatial models in which discrete random events are globally coupled with continuous space-dependent variables solving partial differential equations, e.g., stochastic hybrid models of excitable membranes. We derive a law of large numbers which establishes a connection to deterministic macroscopic models and a martingale central limit theorem which connects the stochastic fluctuations to diffusion processes. As a prerequisite we carry out a thorough discussion of Hilbert space valued martingales associated to the PDMPs. Furthermore, these limit theorems provide the basis for a general Langevin approximation to PDMPs, i.e., stochastic partial differential equations that are expected to be similar in their dynamics to PDMPs. We apply these results to compartmental-type models of spatially extended excitable membranes. Ultimately this yields a system of stochastic partial differential equations which models the internal noise of a biological excitable membrane based on a theoretical derivation from exact stochastic hybrid models.\medskip

\medskip

\end{minipage}
\end{center}

\input{paper_introduction}

\input{section_on_pdmps}


\input{associated_martingale}

\input{section_LLN}

\input{section_clt_intro}

\input{section_CLT}

\input{section_langevin}

\section{Proofs of the main results}\label{section_all_proofs}
\input{LLN_proof}

\subsection{Proof of Theorem \ref{clt_by_local_martingale} (Central limit theorem)}\label{sec_clt_proofs}

\input{tightness_theorem_proof}

\input{proof_support_of_limit}

\input{martingale_problem_proof}


\input{section7_main_part.tex}

\input{paper_conclusion}


\medskip\medskip

\textbf{Acknowledgements:} During the time the presented work was accomplished M.~Riedler was a PhD student at Heriot-Watt University supported by the EPSRC grant EP/E03635X/1. M.~Riedler further acknowledges support from a joint UK Mathematical Neuroscience Network (MNN) and the Cell Signalling Network (SIGNET) travel grant.


\begin{appendix}

\input{proof_of_square_integrable}

\section{Proofs for the neuron models}\label{section_9_proof_compart_model}

\input{excitable_media_conditions_proof}

\end{appendix}




\bibliographystyle{plain}
\bibliography{bibliography}

\end{document}

%% file: paper_introduction.tex
\section{Introduction}

In this study we present limit theorems for sequences of Piecewise Deterministic Markov Processes (PDMPs) with values in a separable Hilbert space. PDMPs are a particular class of c\`adl\`ag, strong Markov processes which combine continuous deterministic time evolution and discontinuous, instantaneous, random `jump' events. We note that in view of applications this paper is ultimately motivated by the interest in the derivation of a justifiable Langevin approximation to spatio-temporal stochastic hybrid models of excitable membranes, e.~g., neuronal membranes. This is accomplished by the limit theorems we present in the following.

We start briefly introducing the general idea of our framework and the main results which are made precise in the subsequent sections. We consider a family of fully coupled, Hilbert space-valued PDMPs indexed by $n\in\mathbb{N}$. Here fully coupled means that the PDMPs which split into a continuously moving and a piecewise constant component are such that the jump rates of the processes depend on the state of the full system and the continuous dynamics depend on the state of the jump component. For the limit theorems we rely on two key assumptions. Firstly, jumps possess heights decreasing to zero for $n\to\infty$ but occur at an increasing frequency roughly inversely proportional to the jump size. We are therefore in the {\it fluid limit} setting, cf.~\cite{Kurtz1,Kurtz2}. Secondly, we assume that for each $n$ the continuous dynamics in between jumps depend on the piecewise constant component only via a \emph{finite} set of (Hilbert space-valued) functions thereof, which we call \emph{coordinate functions}. It is the sequence of coordinate functions coupled to the continuous component for which we derive limits. 
The first limit theorem we present is a weak law of large numbers for PDMPs in infinite-dimensional Hilbert spaces where the deterministic limit is given by a solution of an abstract evolution equation. Next we proceed to the presentation of a central limit theorem for the martingales associated with a PDMP. This central limit theorem gives the basis for an approximation of PDMPs by diffusion processes which are solutions of stochastic partial differential equations. Finally, we show how to represent the stochastic process arising as the limit in the central limit theorem as a solution of a stochastic partial differential equation (SPDE) which then yields a Langevin approximation for PDMPs by a system of SPDEs. The new results presented extend previous results for PDMPs and pure jump processes in Euclidean space \cite{Kurtz2,EthierKurtz,Wainrib1}.
%
%
%
%
The difficulties in extending the fluid limit theorems in \cite{Kurtz1,Kurtz2,Wainrib1} to processes taking values in infinite-dimensional Hilbert spaces lie, on the one hand, in the appropriate treatment of Hilbert space-valued martingales. These arise by splitting a PDMP, being a semi-martingale, into a sum of a part with finite variation and a local martingale. As these considerations are essential we have devoted a full section, Section \ref{section_associated_martingale}, to the discussion of the martingales. 
On the other hand, the more intricate existence theory of solutions to abstract evolution equations compared to solutions of ordinary differential equations in Euclidean space demands for additional technical rigour. 


We apply our theoretical findings to spatially extended hybrid models of excitable membranes. A first hybrid formulation of one such model in the context of neuroscience was presented in \cite{Austin} and reformulated and extended as examples for PDMPs taking values in infinite-dimensional Hilbert spaces in \cite{BuckwarRiedler}. 
%
%
%
%
%
%
%
%
%
%
For example, the Hodgkin-Huxley model is a deterministic, macroscopic model for the coupled evolution of the neuronal membrane potential and the averaged gating dynamics of ion channels \cite{HodgkinHuxley52}. 
More realistically, the membrane potential, which is the macroscopically observed variable of interest, arises from the stochastic dynamics of finitely many ion channels. Thus the application of our limit theorems shows that the Hodgkin-Huxley is obtained as the limit of a sequence of stochastic microscopic models taking the form of Hilbert space valued PDMPs in the sense of a law of large numbers. Conceptually, here the fluid limit corresponds to increasing the number of ion channels while simultaneously decreasing the individual influence of an individual channel on the total current. The martingale central limit theorem can then be used to define the Langevin approximation providing a relatively simple stochastic version of the Hodgkin-Huxley model incorporating internal fluctuations.


Concluding this introduction, we comment on related work to fluid limits in the infinite-dimensional setting. Averaging for PDMPs in infinite-dimension, in particular for the neuron model introduced in \cite{Austin}, wherein also a law of large numbers was considered, has been recently considered in \cite{GenadotThieullen}. For a model of linear chemical reactions by jump Markov processes a law of large numbers \cite{ArnoldTheo} and a central limit theorem \cite{Kotelenez1} have been proven based on the original work of \cite{Kurtz1,Kurtz2} for finite-dimensional jump-processes. In these cases the deterministic limit is a reaction-diffusion partial differential equation and the central limit theorem yields diffusion processes given by stochastic partial differential equations. Limit theorems for variations of this model have been investigated in two series of studies, cf.~\cite{Kotelenez2,Kotelenez3,Kotelenez4} and \cite{Blount1,Blount2,Blount3,Blount4,Blount5}. A central difference between spatial models of excitable media to models of chemical reactions is that the latter exhibit diffusive motion of the reactants ($\sim$ channels) which is absent in the former. Additionally excitable media equations exhibit non-local interaction of channels as their dynamics are coupled globally via the membrane potential. The limit theorems we establish have to account for these differences. Further, there is also a difference on the technical side. The technique employed in \cite{Kotelenez1} and in all subsequent publications cited above is based on the semigroup approach to stochastic / deterministic evolution equations. In contrast, we pursue in the present paper the approach of a weak formulation. At large, the weak formulation of evolution problems allows to consider more general equations as when dealing with mild, strong or classical solutions, cf.~a discussion of this aspect in \cite[Chap.~23.1]{Zeidler}. 
Finally, we also mention a central limit theorem for Hilbert-valued martingales in \cite{Metivier} and a diffusion approximation of SPDEs on nuclear spaces driven by Poisson random measures in \cite{Kallianpur}. The methods of proof we employ for the theoretical results in this study are motivated by the two last references, but differ as the classes of stochastic processes considered therein and in the present manuscript are different.

\medskip

The remainder of the paper is organised as follows. We first briefly define PDMPs in Section \ref{sec_PDMP} and precisely state the structure for a sequence of such PDMPs to allow for a limit. Then we discuss in detail the associated martingale process in Section \ref{section_associated_martingale}. Limit theorems and the diffusion approximation are presented in Sections \ref{Sec_limitThm} and \ref{section_the_CLTs}. We have deferred the proofs of the main results to Section \ref{section_all_proofs}. Next in Section \ref{section_application} we discuss applications of these limit theorems to compartmental-type models of excitable membranes where the proofs of the conditions are deferred to Appendix \ref{section_9_proof_compart_model}. The paper is concluded in Section \ref{section_conclusions} with a brief discussion and an outlook on further developments and applications. Finally, the Appendix \ref{section_ito_iso_banach_proof} of the paper contains the proof of the technical Theorem \ref{banach_martingale_is_square_int} that guarantees the square-integrability of the associated Hilbert space valued martingales and establishes an appropriate It{\^o}-isometry.

%% file: section_on_pdmps.tex
\section{Piecewise Deterministic Markov Processes}\label{sec_PDMP}

In the first subsection we briefly define PDMPs and, in particular, discuss the specific subclass of PDMPs for which we present limit theorems in this study. For a general discussion of PDMPs we refer to the monographs \cite{Davis2,Jacobsen} and, specifically, for Hilbert space valued PDMPs associated to solutions of partial differential equations we refer to \cite{BuckwarRiedler,RiedlerPhD}. In the second subsection we present the sequence of PDMPs for which the limits are analyzed in this study. Finally, a notational remark: in this paper pairings $(\cdot,\cdot)$ and $\langle\cdot,\cdot\rangle$ denote the inner product or the duality pairing, respectively, with respect to a certain Hilbert space which is usually indicated with a subindex. Further, $\ast$ is used to denote dual spaces.

\subsection{PDMPs on Hilbert spaces}\label{general_PDMP_section}

Let $(\Omega,\sF,(\sF_t)_{t\geq 0},\Pr)$ denote a filtered probability space satisfying the usual conditions, $X\subset H\subset X^\ast$ be an evolution triple of separable real Hilbert spaces and $K$ be a countable set of isolated states. The product $H\times K$ serves as the state space for a PDMP. Then, a PDMP is a c\`adl\`ag strong Markov process $X_t(\omega)=(U_t(\omega),\Theta_t(\omega))\in H\times K$ for all $t\geq 0$ which consists of two components. The first, $U_t$, takes values in $H$, possesses continuous sample paths and is denoted the \emph{continuous component} of the PDMP. The second, $\Theta_t$, taking values in $K$ and possessing right-continuous, piecewise constant sample paths, we call its \emph{jump component}. We say a PDMP is \emph{regular} if the number of jumps of $\Theta_t$ is a.s.~finite in every finite time interval $[0,T]$. In this study PDMPs are always regular.

We next state the mechanisms which govern the time evolution of the paths of such a PDMP. Firstly, there exist for each $\theta\in K$ an abstract evolution equation
\begin{equation}\label{abstract_ODE}
\dot{u}= A(\theta)\,u+B(\theta,u)
\end{equation}
where $A(\theta):X\to X^\ast$ is a linear and $B(\theta,\cdot): X\to X^\ast$ a (possibly nonlinear) operator. We assume that the family of abstract evolution equations \eqref{abstract_ODE} is well-posed, i.e., given any $\theta\in K$ and any initial condition $u\in H$ there exists a unique global weak solution $\phi(\cdot,(u,\theta))\in L^2((0,T),X)\cap H^1((0,T),X^\ast)$ depending continuously on the initial condition. Note, that the regularity implies $\phi(\cdot,(u,\theta))\in C([0,T],H)$, cf.~\cite[Chap.~11]{RenardyRogers}. Then the trajectory of the continuous component $U_t$ follows in between jumps of the jump component $\Theta_t$ the weak solution to \eqref{abstract_ODE} corresponding to the parameter $\theta$ given by the current state of the jump component. That is, for $\tau_k,\,k\in\mathbb{N}$, denoting the jump times of the PDMP we have that
\begin{equation*}
U_t=\phi(t-\tau_k,(U_{\tau_k},\Theta_{\tau_k})) \qquad\forall\,t\in[\tau_k,\tau_{k+1})\,.
\end{equation*}
Secondly, describing the stochastic transition dynamics of the jump component $\Theta_t$ there exist measurable transition rates $\Lambda: H\times K\to\rr_+$ that define the distributions of the random jump time of $\Theta_t$ in the sense that for all $\theta\in K$
\begin{equation}\label{def_pdmp_survivor_function}
\Pr\bigl[\Theta_{t+s}=\Theta_t,\, 0\leq s\leq\Delta t\,\big|\, \Theta_t\bigr]=\textnormal{exp}\Bigl(-\int_0^{\Delta t} \Lambda(U_{t+s},\Theta_t)\,\dif s\Bigr)\,.
\end{equation}
In view of \eqref{def_pdmp_survivor_function} we assume that $\Lambda$ is integrable along the solutions of \eqref{abstract_ODE} on any finite time interval, i.e.,
\begin{equation*}
\int_0^T \Lambda(\phi(t,(u,\theta)),\theta)\,\dif t <\infty\qquad\forall\,T<\infty
\end{equation*}
for all $\theta\in K$ and all initial conditions $u\in H$, but diverging as $T\to\infty$. We note that in applications we usually find that the transition rate $\Lambda$ is bounded which implies the regularity of the PDMP. Finally, there exists a Markov kernel $\mu$ on $H\times K$ into $K$ that gives the distribution of the post jump value, i.e.,
\begin{equation}
\Pr\bigl[\Theta_t=\xi\,|\,\Theta_t\neq \Theta_{t-}\bigr] = \mu\bigl((U_t,\Theta_{t-}),\{\xi\}\bigr)\quad\forall\,\xi\in K\,.
\end{equation}
The elements of the quadruple $(A,B,\Lambda,\mu)$ are called the \emph{characteristics} of the process and under the above conditions define a regular PDMP uniquely (up to versions). 
Furthermore, under these conditions the following result characterising the extended generator of PDMPs is proven in \cite{BuckwarRiedler,RiedlerPhD}.

\begin{thm}\label{PDMP_gen_theorem} A function $f: H\times K\to\rr$ is in the domain of the extended generator of a PDMP if the mapping $t\mapsto f(U_t,\Theta_t)$ is absolutely continuous almost surely and the mapping $(\xi,s,\omega)\mapsto f(U_{s-},\xi)-f(U_{s-}(\omega),\Theta_{s-}(\omega))$ is integrable with respect to the random measure $\Lambda(U_{s-},\Theta_{s-})\mu\bigl((U_{s-},\Theta_{s-}),\dif\xi\bigr)\dif s$. 

Moreover, if in addition $f$ is continuously Fr\'echet-differentiable with respect to its first argument such that the Riesz Representation\footnote{Note that the Fr\'echet derivative at a point $u\in H$ is a linear, bounded functional on $H$ and thus an element of the dual $H^\ast$.} $f_u\in H$ of the Fr\'echet derivative satisfies $f_u(u,\theta)\in X$ for $u\in X$ and is a locally bounded composition operator in $L^2((0,T),X)$,\footnote{An example of such a function $f$ is $(u,\theta)\mapsto\|u\|_H^2$ in which case $f_u(u,\theta)=2 u$.} then the extended generator $\mathcal{A}f$ is given by
\begin{equation}\label{infindimGen}
\mathcal{A}f(u,\theta)=\bigl\langle A(\theta)\,u+B(\theta,u), f_u(u,\theta)\bigr\rangle_X+\Lambda(u,\theta)\int_K \Bigl(f(u,\xi)-f(u,\theta)\Bigr)\,\mu\bigl((u,\theta),\dif \xi\bigr)\,.
\end{equation}
\end{thm}

\subsection{An appropriate sequence of PDMPs}\label{limit_appr_pdmps}

Let $E$ denote another separable real Hilbert space. Further, for a certain $m\in\mathbb{N}$ (its significance is explained in the next paragraph) we denote by $\mathcal{H}=\bigtimes_{i=1}^m H$, $\mathcal{E}=\bigtimes_{i=1}^m E$ the direct products of the Hilbert spaces $H$ and $E$ which are Hilbert spaces themselves. Finally we set $\mathcal{E}^\ast=\bigtimes_{j=1}^m E^\ast$ which is the dual space to $\mathcal{E}$.

We now define the structure of the sequence of processes for which we derive the limit theorems. For all $n\in\mathbb{N}$ let $(\Omega^n,\sF^n,(\sF_t^n)_{t\geq 0},\Pr^n)$ be a a filtered probability space satisfying the usual conditions and the processes $(X_t^n)_{t\geq 0}=(U^n_t,\Theta^n_t)_{t\geq 0}$ defined thereon are regular PDMPs taking values in $H\times K_n$ with path properties as defined in Section \ref{general_PDMP_section}. Correspondingly, the characteristics of the PDMPs are given by $(A^n,B^n,\Lambda^n,\mu^n)$. Note that the state space $K_n$ for the piecewise constant component changes with varying index $n$ whereas the state space $H$ for the continuous component remains fixed. Therefore, in order for such a sequence of processes to allow for a limit we need to impose a special structure on the characteristics referring to the continuous component. To this end we assume there exists an $m\in\mathbb{N}$, introduced above, such that for each PDMP $(U^n_t,\Theta^n_t)_{t\geq 0}$ there exists a family of measurable \emph{coordinate functions} $z_i^n:K_n\to E,\, i=1,\ldots,m$, such that the characteristics $A^n(\theta), B^n(\theta)$ depend on the piecewise constant component and on the index $n$ only via the $\mathcal{E}$--valued \emph{coordinate process} $z^n(\theta)=(z_1^n(\theta)\,,\ldots,\,z^n_m(\theta))$. That is, there exist measurable operators $A,B:\mathcal{E}\times X\to X^\ast$ such that for all $n\in\mathbb{N}$, all $u\in H$ and all $\theta\in K_n$
\begin{equation}\label{uniform_operators}
A^n(\theta)\, u\,=\,A(z^n(\theta))\, u, \quad B^n(\theta,u)\,=\,B(z^n(\theta),u).
\end{equation}
The coordinates $z^n$ can be interpreted as a `sufficient statistic' of the piecewise constant component for the evolution of the continuous component. In statistics a sufficient statistic for a quantity of interest is a function of the observations that is sufficient to estimate this particular quantity. For example, the sample average of independently and identically distributed real random variables is a sufficient statistic for the mean of their distribution. In the present setting, this means that the coordinate functions contain all information about the vector $\theta$ that is needed to determine the continuous dynamics in between jumps. 
Further, the essence of the subsequent limit theorems is that the sequence of coordinate processes on the space $\mathcal{E}$ allows for a limit under certain conditions. Typically, in applications one is interested in the dynamics of the continuous components only, thus a restriction of the attention to the coordinate functions is well justified. 
As $\mathcal{E}$ is a (vector-valued) Hilbert space itself, no generality would be lost if instead of the family of coordinate functions we assumed the existence of Hilbert space-valued functions $z^n$ taking values in the same Hilbert space for each $n$. However, we decided to use this more detailed notation since in examples one usually encounters that it is a set of coordinate functions that encodes the information necessary for defining the dynamics of the continuous component.

In order to illustrate this set-up let us briefly discuss the Hodgkin-Huxley model as an example of the general excitable membrane model considered in Section \ref{section_application}. Here the sequence of abstract evolution equations \eqref{abstract_ODE} arises from parabolic partial differential equations modelling the space-time evolution of the membrane potential of the form
%
\begin{equation}\label{Hodgkin_huxley_system}
\dot u(t,x) = \Delta u(t,x) + \sum_{i=\textnormal{Na,K,L}}\nbar g_i p_i(t,x) \bigl(E_i-u(t,x)\bigr),\quad t\geq 0,\, x\in D \subset\rr^d 
\end{equation}
with constants $\nbar g_i>0$ and $E_i\in\rr$, cf.~\eqref{limit_theorem_apl_stoch_cable}. The indices refer to electrical currents due to the movement of charged Sodium (Na) and Potassium (K) ions across the membrane and ohmic leakage (L) current mainly due to Chloride ions \cite{Koch}. In hybrid versions of the Hodgkin-Huxley system the conductances $p_i(t,x)$ depend on the finite number of open ion channels distributed in the membrane which increases with $n$. Each individual channel is modelled stochastically opening or closing at random times with dynamics depending on $u$, cf.~Section \ref{subsection_compartmental_type_models} for more details. In the case of constant potential $u(t,x)\equiv \nbar u$ each channel were a continuous time Markov chain. The collection of channel states at any time instant $t$ defines the discrete component $\Theta^n_t$. Finally, the coordinate functions $z^n$ relate channels in a specific state to their location in the physical space $D$, cf.~their definition in \eqref{Arnold_coord_fct}. They map the channel configurations into piecewise constant space-time functions stating the local density of channels in the particular states, thus $p_i^n(t,x):=z^n_i(\Theta^n_t)\in L^2(D)=E$. Hence, equipped with suitable boundary conditions equation \eqref{Hodgkin_huxley_system} is an abstract evolution equation of the type \eqref{abstract_ODE} where the Hilbert spaces $H,\,X$ and $E$ are spaces of real functions on $D\subset\rr^d$.

%% file: associated_martingale.tex
\section{The associated martingale process}\label{section_associated_martingale}

For the limit theorems we derive in this paper, the main estimation procedures concern certain martingales associated with the PDMP. As these are of such central importance we discuss them in this separate section. The principle aim is, on the one hand, to derive conditions that imply the convergence in probability of the associated martingales as needed for the law of large numbers (cf.~condition \eqref{uniform_martingale_bound} in Theorem \ref{LLN}) and, on the other hand, we present some necessary structure for the central limit theorems.
Therefore we define for all $j=1,\ldots,m$ the $E$-valued stochastic process $M^n_j$ by
\begin{equation}\label{local_martingale_2}
M_j^n(t):=z_j^n(\Theta_t^n)-z_j^n(\Theta_0^n)-\int_0^t \bigl[\mathcal{A}^n\langle\,\cdot\,,z_j^n(\cdot)\rangle_E\bigr](U^n_s,\Theta^n_s)\,\dif s,
\end{equation}
where the integrand in the right hand side is given by
\begin{eqnarray}
\bigl[\mathcal{A}^n\langle\,\cdot\,,z_j^n(\cdot)\rangle_E\bigr](U^n_s,\Theta^n_s)\!\!&=&\!\!\Lambda^n(U_s^n,\Theta^n_s)\int_{K_n} \,
\Bigl(z_j^n(\xi)-z_j^n(\Theta_s^n)\Bigr)\,\mu^n\bigl((U^n_s,\Theta^n_s),\dif \xi\bigr)\nonumber\\
&=&\!\! \Lambda^n(U_s^n,\Theta_s^n)\sum_{\xi\in K_n}\Bigl(z_j^n(\xi)- z_j^n(\Theta_s^n)\Bigr)\,\mu^n\bigl((U^n_s,\Theta^n_s),\{\xi\}\bigr)\,.
\phantom{xxxx}\label{new_equation_MT_sugg}
\end{eqnarray}
Hence the integrand is a countable convex combination of elements in $E$ with time-dependent coefficients and in between jumps it depends continuously on $s$. Anticipating condition \eqref{LLN_martingale_bound_first_lemma} below, which we generally assume to hold, we find that the integral in the right hand side of \eqref{local_martingale_2} almost surely exists in the sense of Bochner. For a brief discussion of the Bochner integral we refer to \cite[App.~A]{PrevotRoeckner}. 
For an application of a functional $\phi\in E^\ast$ to \eqref{local_martingale_2} we obtain
\begin{equation}\label{local_martingale}
\langle\phi,M_j^n(t)\rangle_E=\langle\phi,z_j^n(\Theta^n_t)\rangle_E-\langle\phi,z_j^n(\Theta^n_0)\rangle_E-\int_0^t \bigl[\mathcal{A}^n\langle\phi,z_j^n(\cdot)\rangle_E\bigl](U_s^n,\Theta^n_s)\, \dif s,
\end{equation}
where the integrand is
\begin{equation*}
\bigl[\mathcal{A}^n\langle\phi, z_j^n(\cdot)\rangle_E\bigr](U^n_s,\Theta^n_s)=\Lambda^n(U_s^n,\Theta^n_s)\!\!\int_{K_n}\!\!
\langle\phi, z_j^n(\xi)\rangle_E-\langle\phi, z_j^n(\Theta_s^n)\rangle_E\,\mu^n\bigl((U^n_s,\Theta^n_s),\dif \xi\bigr).
\end{equation*}
Thus the integral has the form of the extended generator, cf.~Theorem \ref{PDMP_gen_theorem}, applied to the mapping $(u,\theta)\mapsto\langle\phi,z^n_j(\theta)\rangle_E$. This already suggests that the processes \eqref{local_martingale} are martingales under suitable boundedness conditions. 
In fact we are able to establish that the processes $M^n_j$ are $E$--valued c\`adl\`ag martingales. We refer to \cite{DaPratoZabczyk,PrevotRoeckner} for a brief discussion of martingales in infinite-dimensional spaces. The easiest way to validate the martingale property is due to the following result \cite[Sec.~2.3]{PrevotRoeckner}: If $\EX^n\|M^n_j(t)\|_E< \infty$ for all $t\in[0,T]$, the Hilbert space-valued martingale property holds if and only if $\langle\phi,M^n_j(t)\rangle_E$ is a real-valued martingale for all $\phi\in E^\ast$. The following theorem gives a condition that guarantees that the processes \eqref{local_martingale_2} are square-integrable martingales and satisfy an It{\^o}-isometry. The proof is rather technical and thus we have deferred it to the Appendix \ref{section_ito_iso_banach_proof}.

\begin{thm}\label{banach_martingale_is_square_int} Let $n\in\mathbb{N}$ be fixed and assume that for all $t>0$ it holds that
\begin{equation}\label{LLN_martingale_bound_first_lemma}
\EX\,\int_0^t\Bigl[\Lambda^n(U^n_s,\Theta^n_s)\int_{K_n} \|z^n_j(\xi)-z^n_j(\Theta^n_s)\|_E^2\,\mu^n\bigl((U^n_s,\Theta^n_s),\dif\xi\bigr)\Bigr]\,\dif s\,<\infty\,.
\end{equation}
Then the process $M^n_j$ is a square-integrable martingale and satisfies the \emph{It{\^o}-isometry}
\begin{equation}\label{banach_ito_isometry}
\EX^n\|M^n_j(t)\|_E^2 = \int_0^t\EX^n\Bigl[\Lambda^n(U^n_s,\Theta^n_s)\int_{K_n} \|z^n_j(\xi)-z^n_j(\Theta^n_s)\|_E^2\,\mu^n\bigl((U^n_s,\Theta^n_s),\dif\xi\bigr)\Bigr]\,\dif s\,. 
\end{equation}
\end{thm}


We continue the investigation of the processes $M^n_j$ as Hilbert space valued martingales. From now on we always assume that assumption \eqref{LLN_martingale_bound_first_lemma} holds. Note that the finiteness of the second moments of the jump sizes is a standard condition in related fluid limit theorems \cite{Kurtz1,Wainrib1,Metivier}. We introduce a concept akin to the quadratic covariance operator in Euclidean finite dimensional spaces. This concept is important for the central limit theorems in, on the one hand, establishing weak convergence, and, on the other hand, characterising the limit. For further reference we refer to \cite{MetivierSemi}.

\begin{defn} For the square-integrable, $E$--valued, c\`adl\`ag martingale $M^n_j$ we denote by $(\ll\!\! M^n_j\!\!\gg\!\!_t)_{t\geq 0}$ its \emph{predictable quadratic variation process}, i.e., the unique (up to indistinguishability), predictable $L_1(E^\ast,E)$-valued\footnote{$L_1(E^\ast,E)$ denotes the space of trace class operators from the Hilbert space $E^\ast$ into $E$.}  process which satisfies that for all $\phi,\psi\in E^\ast$ the real-valued process
\begin{equation}\label{def_limit_theorems_cross_caracteristcs}
t\mapsto\langle \phi, M^n_j(t)\rangle_E\,\langle\psi,M^n_j(t) \rangle_E- \langle\phi, \ll\!\! M^n_j\!\!\gg\!\!_t\psi \rangle_E
\end{equation}
is a local martingale. 
\end{defn}

The aim now is to obtain an explicit formula for the quadratic variation process of the individual martingales $M^n_j$ as well as of the vector-valued process $M^n$ of all martingales $M^n_j$, i.e., the $\mathcal{E}$--valued process
\begin{equation*}
t\mapsto M^n(t)=\bigl(M^n_1(t),\ldots,M^n_m(t)\bigr).
\end{equation*}
To this end we define for all $i,j=1,\ldots,m$ operators $G_{ij}^n\in L(E^\ast,E)$ by
\begin{eqnarray}\label{martingale_cov_operator_diagonal}
\lefteqn{\psi\mapsto\ G_{ij}^n(u,\theta^n)\psi\,:=}\\[2ex]
&&\phantom{xxxxxxx}:=\,\Lambda^n(u,\theta^n)\int_{K_n} \langle\psi, z^n_i(\xi)-z^n_i(\theta^n) \rangle_E\,\Bigl(z^n_j(\xi)-z^n_j(\theta^n)\Bigr)\,\mu^n\bigl((u,\theta^n),\dif\xi).\nonumber
\end{eqnarray}
Clearly, these are linear, bounded operators mapping $E^\ast\to E$ and depend measurably on $(u,\theta^n)\in H\times K_n$. For $i=j$ each operator is \emph{non-negative}, i.e., $\langle\phi,G^n_{jj}(u,\theta^n)\phi\rangle_E\geq 0$ for all $\phi\in E^\ast$, and \emph{symmetric}, i.e., $\langle\psi,G^n_{jj}(u,\theta^n)\phi\rangle_E=\langle\phi,G^n_{jj}(u,\theta^n)\psi\rangle_E$ for all $\phi,\psi\in E^\ast$. Let $(\varphi_k)_{k\in\mathbb{N}}$ denote an orthonormal basis in $E^\ast$. We find due to the Riesz Representation Theorem and Parseval's identity that the trace of the operators $G_{jj}$ satisfies
\begin{eqnarray}\label{quad_variation_trace_charac}
\textnormal{Tr}\,G_{jj}^n(u,\theta^n) &=& \Lambda^n(u,\theta^n)\int_{K_n} \sum_{k\in\mathbb{N}}\Bigl(\langle\varphi_k, z^n_j(\xi)-z^n_j(\theta^n) \rangle_E\Bigr)^2\,\mu^n\bigl((u,\theta^n),\dif\xi)\nonumber\\
&=& \Lambda^n(u,\theta^n)\int_{K_n} \big\|z^n_j(\xi)-z^n_j(\theta^n)\big\|_E^2\,\mu^n\bigl((u,\theta^n),\dif\xi).
\end{eqnarray}
For arbitrary $i,j$ the trace is bounded in terms of \eqref{quad_variation_trace_charac} as it follows from Young's inequality that
$\textnormal{Tr}\,G_{ij}^n(u,\theta^n)\leq \tfrac{1}{2}\textnormal{Tr}\,G_{ii}^n(u,\theta^n)+\tfrac{1}{2}\textnormal{Tr}\,G_{jj}^n(u,\theta^n)$. 

Let $\Phi=(\phi_1,\ldots,\phi_m)$ and $\Psi=(\psi_1,\ldots,\psi_m)$ be elements of $\mathcal{E}^\ast$. Summing over all operators \eqref{martingale_cov_operator_diagonal}   
applied to the components of $\Phi,\,\Psi$ as indicated by the indices, i.e., 
\begin{equation}\label{quadratic_rep_of_Gn}
\langle\Phi,G^n(u,\theta^n)\,\Psi\rangle_{\mathcal{E}}:=\sum_{i,j=1}^m \langle\phi_i,G^{n}_{ij}(u,\theta^n)\,\psi_j\rangle_E,
\end{equation}
we obtain a linear, bounded operator $G^n(u,\theta^n)$ mapping $\mathcal{E}^\ast$ to $\mathcal{E}$. This operator is symmetric as the operators $G_{ij}^n$ satisfy $\langle\phi,G^n_{ij}(u,\theta^n)\psi\rangle_E=\langle\psi,G^n_{ji}(u,\theta^n)\phi\rangle_E$ for all $i,j$. Moreover, the operator $G^n(u,\theta^n)$ is non-negative as it holds that
\begin{equation*}
\langle\Psi,G^n(u,\theta^n)\,\Psi \rangle_\mathcal{E} \,=\, \Lambda^n(u,\theta^n)\int_{K_n} \Bigl(\sum_{i=1}^m\bigl\langle\psi_i,z^n_i(\xi)-z^n_i(\theta^n)\bigr\rangle_E\Bigr)^2\,\mu^n\bigl((u,\theta^n\bigr),\dif\xi)\,.
\end{equation*}
Finally, the operator $G^n(u,\theta^n)$ is of trace class if the operators $G_{jj}$, $j=1,\ldots,m$, are of trace class and the trace satisfies
\begin{equation}
\textnormal{Tr}\,G^n(u,\theta^n) = \Lambda^n(u,\theta^n)\int_{k_n}\|z^n(\xi)-z^n(\theta^n))\|_\mathcal{E}^2\,\mu^n\bigl((u,\theta^n),\dif\xi\bigr)\,. 
\end{equation}


We next prove that the operators \eqref{quad_variation_trace_charac} give the quadratic variations of the martingales \eqref{local_martingale_2}.

\begin{prop}\label{prop_char_quad_var} The quadratic variation of the martingale $M^n_j$ satisfies for all $t\geq 0$
\begin{equation}\label{explicit_def_cross_1}
\ll\!\! M^n_j\!\!\gg\!\!_t\,=\,\int_0^t G_{jj}^n(U^n_s,\Theta^n_s)\,\dif s\,. 
\end{equation}
\end{prop}

\begin{rem} It is an immediate consequence of Proposition \ref{prop_char_quad_var} that the quadratic variation of the $\mathcal{E}$--valued martingale $M^n$ is given analogously to \eqref{explicit_def_cross_1} by integrating the operator $G^n$.
\end{rem}

\begin{proof} First of all note that due to the characterisation of the trace \eqref{quad_variation_trace_charac} and condition \eqref{LLN_martingale_bound_first_lemma} it holds that the process in the right hand side of \eqref{explicit_def_cross_1} takes values in $L_1(E^\ast,E)$ almost surely. Further, it holds that $\ll\!\! M^n_j\!\!\gg\!\!_t$ satisfies for all $\phi,\psi\in E$ that
\begin{eqnarray*}\label{explicit_def_prop_cross_char}
\lefteqn{\langle\phi, \ll\!\! M^n_j\!\!\gg\!\!_t\psi \rangle_E\ =}\\[2ex]
&&\hspace{-20pt}=\ \int_0^t\Lambda^n(U^n_s,\Theta^n_s)\int_{K_n} \langle\psi, z^n_j(\xi)-z^n_j(\Theta^n_{s}) \rangle_E\,\langle\phi, z^n_j(\xi)-z^n_j(\Theta^n_{s}) \rangle_E\,\mu^n\bigl((U^n_s,\Theta^n_s),\dif\xi)\,\dif s\nonumber 
\end{eqnarray*}
as this right hand side is, due to \cite[Prop.~4.6.2]{Jacobsen}, the unique real-valued process such that $\langle \phi, M^n_j(t)\rangle_E\,\langle\psi,M^n_j(t) \rangle_E-\langle\phi, \ll\!\! M^n_j\!\!\gg\!\!_t\psi \rangle_E$ is a local martingale. Here $\langle \phi, M^n_j(t)\rangle_E$ and $\langle\psi,M^n_j(t) \rangle_E$ are understood as real-valued stochastic integrals with respect to the associated martingale measure of the PDMP. Thus we infer that for all $\phi,\psi\in E$ it holds
\begin{equation*}
\langle\phi,\ll\!\! M^n_j\!\!\gg\!\!_t\psi \rangle_E\, =\, \int_0^t\langle\phi,G^n_{jj}(U_s^n,\Theta^n_s)\psi \rangle_E\,\dif s\,.
\end{equation*}
Finally, the linearity of the Bochner integral (note that $L_1(E^\ast,E)$ is a Banach space) implies \eqref{explicit_def_cross_1}.
\end{proof}

A further second property of the quadratic variation is that the process
\begin{equation*}
t\mapsto\|M^n_j(t)\|^2_E-\textnormal{Tr}\ll\!\! M^n_j\!\!\gg\!\!_t
\end{equation*}
is a local martingale. We note that the \emph{trace process} $t\mapsto \textnormal{Tr}\ll\!\! M^n_j\!\!\gg\!\!_t$ is the unique, predictable increasing process exhibiting this property. Using the characterisation \eqref{explicit_def_cross_1} of the quadratic variation we thus obtain that the process
\begin{equation}\label{trace_process}
t\mapsto\|M^n_j(t)\|_E^2-\textnormal{Tr}\,\Bigl(\int_0^t G_{jj}^n(U^n_s,\Theta^n_s)\,\dif s\Bigr)\, =\, \|M^n_j(t)\|_E^2-\int_0^t \textnormal{Tr}\,G_{jj}^n(U^n_s,\Theta^n_s)\,\dif s 
\end{equation}
is a local martingale vanishing almost surely at $t=0$ and analogously in the case of the $\mathcal{E}$--valued martingale $M^n$.


\medskip

We are now in a position to state a lemma which establishes the convergence in probability \eqref{martingale_lemma_result_2} of the processes $(M^n_j)_{t\geq 0}$ necessary for the law of large numbers, cf.~condition \eqref{uniform_martingale_bound} in Theorem \ref{LLN}.

\begin{lem}\label{martingale_bound_1} Assume that for all $T>0$
\begin{equation}\label{LLN_martingale_bound_first_lemma_2}
\lim_{n\to\infty}\, \EX\,\int_0^T\Bigl[\Lambda^n(U^n_s,\Theta^n_s)\int_{K_n} \|z^n_i(\xi)-z^n_i(\Theta^n_s)\|_E^2\,\mu^n\bigl((U^n_s,\Theta^n_s),\dif\xi\bigr)\Bigr]\,\dif s\,=\,0\,.
\end{equation}
Then the process \eqref{trace_process} is a martingale and for all $T,\,\eps>0$, it holds that
\begin{equation}\label{martingale_lemma_result_2}
\lim_{n\to\infty}\Pr^n\bigl[\sup\nolimits_{\,t\in[0,T]}\|M_j^n(t)\|_{E}>\eps\bigr]=0\,.
\end{equation}
\end{lem}


\begin{proof} As the process $M^n_j$ is an $E$-valued c\`adl\`ag martingale, it holds that $\|M^n_j\|_E^2$ is a c\`adl\`ag submartingale. Thus an application of Markov's and Doob's inequalities yield the estimates
\begin{equation*}
\Pr^n\Bigl[\sup\nolimits_{\,t\in[0,T]}\,\|M^n_j(t)\|^2_E>\eps\Bigr]\,\leq\, \frac{1}{\eps}\,\EX^n\bigl[\sup\nolimits_{\,t\in[0,T]}\,\|M^n_j(t)\|^2_E\bigr]
\,\leq\,\frac{4}{\eps}\, \EX^n\|M_j^n(T)\|^2_E\,.
\end{equation*}
Now, the It{\^o}-isometry \eqref{banach_ito_isometry} and condition \eqref{LLN_martingale_bound_first_lemma_2} imply the convergence in probability \eqref{martingale_lemma_result_2}. 
It remains to show that the process \eqref{trace_process} is a martingale. A sufficient condition, see, e.g., \cite[Prop.~B.0.13]{Jacobsen}, 
is that for all $T>0$ it holds
\begin{equation}\label{proof_martingale_bound_2_suff_cond}
\EX^n\Bigl[\sup\nolimits_{t\in[0,T]}\Big|\|M^n_j(t)\|_E^2-\int_0^t \textnormal{Tr}\,G_{jj}^n(U^n_s,\Theta^n_s)\,\dif s\Big|\Bigr]<\infty\,. 
\end{equation}
Estimating the term inside the expectation we obtain
\begin{eqnarray*}
\lefteqn{\sup\nolimits_{t\in[0,T]}\Big|\|M^n_j(t)\|_E^2-\int_0^t \textnormal{Tr}\,G_{jj}^n(U^n_s,\Theta^n_s)\,\dif s\Big|}\\
&&\hspace{-20pt}\leq\ \sup\nolimits_{t\in[0,T]}\|M^n_j(t)\|_E^2+\sup\nolimits_{t\in[0,T]}\int_0^t \Lambda^n(u,\Theta^n)\int_{K_n} \big\|z^n_j(\xi)-z^n_j(\Theta^n)\big\|_E^2\,\mu^n\bigl((u,\Theta^n),\dif\xi)\,\dif s.
\end{eqnarray*}
The expectation of the first supremum term in the right hand side is bounded due to Doob's inequality and the square-integrability of the martingale. The term inside the second supremum is increasing, thus its expectation is finite due to condition \eqref{LLN_martingale_bound_first_lemma_2}.

\end{proof}

%% file: section_LLN.tex
\section{A weak law of large numbers}\label{Sec_limitThm}

In order to propose a deterministic limit of the sequence of PDMPs we consider functions $F_j:\mathcal{E}\times H\to E$, $j=1,\ldots,m$. In combination with the operators $A,\,B$ these functions are used to define a coupled system of deterministic abstract evolution equations
\begin{equation}\label{ODE_system}
\left.\begin{array}{rcll}
\dot u&=&A(p)\,u+B(p,u), &\\[1ex]
\dot p_j&=& F_j(p,u), & j=1,\ldots,m\,.
\end{array}\right.
\end{equation}
We assume that to suitable initial condition $(u_0,p_0)\in H\times\mathcal{E}$ there exists a unique weak solution $(u(t),p(t))_{t\geq 0}$ in $C(\rr_+,H\times\mathcal{E})$ of \eqref{ODE_system}. Additionally, we assume that for all $i=1,\ldots,m$ the components $p_i$ satisfy
\begin{equation}\label{det_limit_int_equation}
\langle\phi,p_i(t)\rangle_E=\langle\phi,p_i(0)\rangle_E+\int_0^t\langle\phi,F_i(p(s),u(s))\rangle_E\,\dif s\qquad\forall\,t\in[0,T],\,\phi\in E^\ast\,. 
\end{equation}
That is, the components $p_j$ satisfy the equation \eqref{ODE_system} in the sense of an Hilbert space valued integral equation. We note that in application one usually encounters deterministic limit systems that possess strong or classical solutions and hence the current weak framework is satisfied. Finally, we assume that the operators $A$, $B$ and $F_j$, $j=1,\ldots,m$, satisfy Lipschitz-type conditions on $L^2((0,T),\mathcal{E}\times X)$ in the sense that for every $T>0$ there exist constants $L_1$ and $L_2$ such that for all $u,v \in L^2((0,T),X)$ and all $p,q\in L^2((0,T),\mathcal{E})$ it holds that
\begin{eqnarray}
\lefteqn{\int_0^T\langle A(q)\,v-A(p)\,u,v-u\rangle_X
+\langle B(q,v)-B(p,u),v-u\rangle_X\,\dif t}\nonumber\\ &&\phantom{xxxxxxxxxxxxxxxxxxxxxxxx}\leq\ L_1\int_0^T\|v-u\|^2_H+\sum_{i=1}^m\|q_i-p_i\|_{E}^2\,\dif t\,.\phantom{xxxxxx}\label{onesided_lip}
\end{eqnarray}
and
\begin{equation}\label{det_sys_lipschitz}
\Bigl(\int_0^T\|F_j(q,v)-F_j(p,u)\|_{E}\,\dif t\Bigr)^2\leq L_2\int_0^T\|v-u\|_H^2+\sum_{i=1}^m\| q_i-p_i\|_{E}^2\,\dif t,
\end{equation}
where we have omitted the arguments $t$ of the functions $u,v,p$ and $q$.



\begin{rem}\label{weaker_conditions} In the proof of the law of large numbers, see Section \ref{section_proof_LLN}, these Lipschitz conditions are applied such that one pairing $(v,q)$ refers to a path segment of the continuous component of a PDMP and the coordinate process and the second $(u,p)$ to the deterministic limit functions. 
%
Thus for the applications of \eqref{onesided_lip} and \eqref{det_sys_lipschitz} in the proof it is sufficient that these hold only for pairings $(v,q)$ out of a set containing almost all paths of the sequence of PDMPs and $(u,p)$ being the deterministic limit, i.e., \mbox{one (!)} distinguished pairing. This restriction of \eqref{onesided_lip} and \eqref{det_sys_lipschitz} to be satisfied only for particular pairings $(v,q)$ and $(u,p)$ out of the whole path space has a decisive advantage: In order to establish these conditions we are able to incorporate additional qualitative results on the trajectories of the PDMPs and the deterministic limit and the constants $L_1,\,L_2$ may depend on $(u,p)$. For example, in the application to excitable membrane models such an additional qualitative 
%
is that the components corresponding to $u,v,p,q$ are pointwise bounded.
\end{rem}

We now present a weak law of large numbers in Theorem \ref{LLN} below. The proof of the theorem follows the lines of previously published limit theorems considering processes in finite dimensions \cite{Kurtz1,Wainrib1}. The main difficulties arising in infinite-dimensional phase space
concerns the bounds on the martingale part, cf.~condition (C1), which is rarely a problem in finite dimensions. 
However, using the appropriate martingale theory in Hilbert spaces 
these can be kept to a minimum. Then the difficulties are mainly of a technical nature as martingale theory in connection with PDMPs in infinite-dimensional spaces gets more involved and is not covered by previous results in \cite{Jacobsen}.
%
%
We have established the necessary theory in the preceding Section \ref{section_associated_martingale} and addressed the question of the convergence of the martingale part (C1) within this framework. Most importantly, in Lemma \ref{martingale_bound_1} we have proven a sufficient condition for (C1) to be satisfied. 
In particular, this sufficient condition \eqref{LLN_martingale_bound_first_lemma_2} is a natural extension of the condition employed in finite dimensions, cf.~\cite{Kurtz1,Wainrib1}. 

A different approach to establishing condition (C1) which avoids using martingale theory in Hilbert spaces is exemplified in the law of large numbers proved in \cite{Austin}. In infinite-dimensional space this approach encounters the problem of simultaneously controlling countably many real martingales compared to only finitely many in the case of its finite-dimensional counterpart. This problem can be overcome with an intricate compactness argument which relies on the assumption that the dual space $E^\ast$ is compactly embedded in some additional normed space and all estimates -- \mbox{especially} an estimate which also implies condition \eqref{LLN_martingale_bound_first_lemma_2} -- have to be derived in the norm of this additional space. Furthermore, the condition, that all martingales $(\langle\phi,M^n_j(t)\rangle_E)_{t\geq 0}$, $j=1,\ldots,m$ and $\phi\in E^\ast$, possess almost surely uniformly bounded paths, has to be introduced. We are of the opinion that our approach is more elegant, but, more importantly, it avoids the introduction of additional conditions.
\medskip

Finally, consistently with the notation in Section \ref{section_associated_martingale} we use   in the subsequent theorem and its proof the notation $\bigl[\mathcal{A}^n\langle\cdot, z_j^n(\cdot)\rangle_E\bigr]$ as defined in \eqref{new_equation_MT_sugg}. Then, for given $(u,\theta^n)\in H\times K_n$ functionals $\bigl[\mathcal{A}^n\langle\,\cdot\,, z_j^n(\cdot)\rangle_E\bigr](u,\theta^n)$ on $E^\ast$ are defined by the mappings $\phi\mapsto\bigl[\mathcal{A}^n\langle\phi, z_j^n(\cdot)\rangle_E\bigr](u,\theta^n)$. As usual we identify the bidual $E^{\ast\ast}$ with $E$ and thus $\bigl[\mathcal{A}^n\langle\,\cdot\,, z_j^n(\cdot)\rangle_E\bigr](u,\theta^n)\in E$.

\begin{thm}\label{LLN}
We assume that the following conditions hold:
\begin{enumerateI}
\item[\textnormal{(C1)}]\ For all $j=1,\ldots,m$ it holds that for fixed $T,\eps>0$ 
\begin{equation}\label{uniform_martingale_bound}
\lim_{n\to\infty}\Pr^n\bigl[\sup\nolimits_{\,t\in[0,T]}\|M^n_j(t)\|_{E}> \eps\bigr]=0\,.
\end{equation}
\item[\textnormal{(C2)}]\ The functions $F_j$, $j=1,\ldots,m$, satisfy for all $\eps>0$ that
\begin{equation}\label{finite_var_bound}
\lim_{n\to\infty}\Pr^n\Bigl[\int_0^T\big\|\bigl[\mathcal{A}^n\langle\,\cdot\,,z^n_j(\cdot)\rangle_E\bigr](U^n_t,\Theta^n_t) -  F_j(z^n(\Theta^n_t),U^n_t)\big\|_{E}\,\dif t >\eps\Bigr]=0\,,
\end{equation}
where we have omitted the argument $t$ of the functions $u$ and $\theta$.
\item[\textnormal{(C3)}]\ The initial conditions $(U_0^n,\Theta_0^n)$ of the sequence of PDMPs converge in probability to the initial conditions of the deterministic limit in the sense that for all $\eps>0$
\begin{equation*}
\lim_{n\to\infty}\Pr^n\Bigl[\|U^n_0-u_0\|_H+\sum_{i=1}^m\|z_i^n(\Theta_0^n)-p_i(0)\|_{E}>\eps\,\Bigr]=0\,.
\end{equation*} 
\end{enumerateI}
Then, for every $\eps>0$ and every fixed $T>0$ it holds that
\begin{equation}\label{conv_estimate}
\lim_{n\to\infty}\Pr^n\Bigl[\sup\nolimits_{\,t\in[0,T]}\Bigl( \|U_t^n-u(t)\|^2_H+\sum_{j=1}^m\| z_j^n(\Theta_t^n)-p_j(t) \|_{E}^2\Bigr) \ > \ \eps\Bigr]=0\,.
\end{equation}
\end{thm}

\begin{rem} The result \eqref{conv_estimate} implies convergence in probability of the processes $(U^n_t,z^n(\Theta^n_t))_{t\geq 0}$ to the deterministic function $(u(t),p(t))_{t\in[0,T]}$ in the Hilbert space $L^2((0,T),H\times\mathcal{E})$. If the differences of the components are almost surely bounded independent of $n$ the convergence even holds in the mean, cf.~the application of the law of large numbers in Theorem \ref{LLN_compartmental_models}. 
Further, the conditions (C1)--(C3) are generalisations from Euclidean space to infinite-dimensional Hilbert spaces of those employed in the corresponding theorems for PDMPs in Euclidean space \cite{Wainrib1} and, in particular, of the original formulation in case of pure jump processes in Euclidean space \cite{Kurtz1}. In these cases the conditions above reduce to the corresponding assumptions.
\end{rem}

%% file: section_clt_intro.tex
\section{The central limit theorem and the Langevin approximation}\label{section_the_CLTs}

We proceed to the presentation of the central limit theorem for associated martingales $(M_t^n)_{t\geq 0}$ defined in \eqref{local_martingale_2}. The central limit theorem provides the theoretical basis for an approximation of spatio-temporal PDMPs by Hilbert-space valued diffusion processes where the latter can be represented by solutions of stochastic partial differential equations, see Section \ref{section_langevin_approximation}.

Proving central limit theorems usually involves two tasks: On the one hand, one has to show the existence of a limit and, on the other hand, one has to provide a characterisation of the limit as a certain stochastic process. The former is equivalent to the problem of tightness of the stochastic processes. In the case of martingales sufficient conditions for tightness depending on the quadratic variation process are stated in \cite{Metivier}. In order to characterise the limit there exist different approaches,  
showing either that the limit solves a given (local) martingale problem which is known to have a unique solution (cf.~\cite{Kallianpur,Metivier}) or proving weak convergence of the finite dimensional distributions (cf.~\cite{Kotelenez1,Wainrib1}). We present two central limit theorems, Theorems \ref{clt_by_local_martingale} and \ref{clt_via_char_functions}, employing the two methods, respectively, however, to avoid repetition we state only the proof of the first in the present study and refer to the PhD thesis of one of the authors \cite{RiedlerPhD} for the proof of the second. The two theorems differ in a technical assumption which in each case arises in addition to the central condition of the convergence of the quadratic variations. We believe that for applications of the limit theorems it is advantageous to know both versions of the martingale central limit theorem, as it is easily conceivable that only one of these technical assumptions is satisfied. Hence the theorems are applicable in different situations.

Finally, we emphasise that in the following the space $\mathcal{E}$ need not necessarily be the same space for which the law of large numbers is satisfied. However, clearly, the space $\mathcal{E}$ in the present section contains the space in the law of large numbers as subspace. In applications, usually, the law of large numbers holds in a space with a stronger norm, for example, for the excitable membrane model considered in Section \ref{section_application} the law of large numbers holds in $L^2(D)$ whereas the central limit theorem holds in the space $H^{-2s}(D)$.\footnote{Here and everywhere else $H^{-2s}(D)$ is the dual space to the Sobolev space $H^{2s}(D)$ where $D\subset\rr^d$ and $s>d/2$.} This is a major difference to the corresponding results in finite-dimensional space where both limit theorems hold in the same space.\footnote{Note also that in finite-dimensional spaces all norms, and hence also all norms on subspaces, are equivalent which does not hold in the case of an infinite-dimensional Hilbert space.}

%% file: section_CLT.tex
\subsection{A martingale central limit theorem}\label{section_martinagle_clt}

In this section we present central limit theorems for the scaled $\mathcal{E}$--valued martingales $(\sqrt{\alpha_n}\,M^n_t)_{t\geq 0}$ associated with a sequence of PDMPs where $\alpha_n\in\rr_+,\,n\in\mathbb{N}$, is a suitable rescaling sequence such that $\lim_{n\to\infty}\alpha_n=\infty$.
Clearly, the rescaling is necessary in order to be able to obtain a limit different from the trivial limit as \eqref{uniform_martingale_bound} implies that $(M^n_t)_{t\geq 0}$ converges to zero in distribution. We note that the sequence $\alpha_n$ can also be interpreted as characterising the speed of convergence of the martingales $(M^n_t)_{t\geq 0}$.

In the following let $t\mapsto G(u(t),p(t))\in L\bigl(\mathcal{E}^\ast,\mathcal{E}\bigr)$ be a Bochner-integrable operator-valued map such that each $G(u(t),p(t))$ is a symmetric, positive trace class operator. Particularly this implies for all $\Phi\in\mathcal{E}^\ast$ and all $t>0$, that it holds that
\begin{equation}\label{Kurtz_uniform_boundedness_of_operator}
\int_0^t\bigl\langle\Phi,G(u(s),p(s))\,\Phi\bigr\rangle_\mathcal{E}\,\dif s<\infty\,.
\end{equation}
Here $(u(t),p(t))_{t\geq 0}$ is the deterministic limit obtained in Theorem \ref{LLN} and the use of this notation for the -- at this point -- arbitrary time-dependent operator $G$ only illustrates that in applications the time-dependence is due to a dependence on the deterministic limit system. These operator-valued functions are used to define a unique centred diffusion process on $\mathcal{E}$, i.e., an $\mathcal{E}$--valued centred Gaussian process with independent increments, continuous sample paths. Such a process is uniquely defined by its covariance operator and due to a theorem of It{\^o} stated in \cite{Kotelenez1} every family of trace class operators $C^\ast(t)\in L_1(\mathcal{E},\mathcal{E})$ which are increasing and continuous in $t$ define a centred diffusion process. In the present situation we define $C^\ast$ in the following way. We denote by $\iota: \mathcal{E}\to \mathcal{E}^\ast$ the canonical identification of a Hilbert space with its dual, hence we can define for $x,y\in \mathcal{E}$,
\begin{equation*}
\bigl(x,C^\ast(t)\,y\bigr)_{\mathcal{E}}= \int_0^t \bigl\langle\iota(x),G(u(s),p(s))\,\iota(y)\bigr\rangle_{\mathcal{E}}\,\dif s
\end{equation*}
which is continuous and increasing for all $x\in\mathcal{E}$ and $C^\ast(t)$ is a trace class operator on $\mathcal{E}$. Moreover, for operators $C(t)\in L_1(\mathcal{E}^\ast,\mathcal{E})$, defined by
\begin{equation}\label{diffusion_limit_covariance}
\bigl\langle\Phi,C(t)\,\Psi\bigr\rangle_\mathcal{E}= \int_0^t \bigl\langle\Phi,(G(u(s),p(s))\,\Psi\bigr\rangle_{\mathcal{E}}\,\dif s ,
\end{equation}
there is obviously a one-to-one relationship between $C^\ast$ and $C$. Hence, we may say that also the latter defines a diffusion process on the space $\mathcal{E}$.

We proceed to the statement of the central limit theorem. The proof of the theorem employs a characterisation of the limit via the local martingale problem. The essential condition characterising the limit is the convergence of the quadratic variation processes \eqref{clt_martingale_cond_2}. The second condition \eqref{clt_martingale_cond_1} is a technical condition on the jump heights which arises due to the method of proof and is usually satisfied in applications. The remaining conditions are such that (D1) guarantees tightness of the sequence of processes and in combination with (D2) that any limit is a continuous stochastic process. The proof of the following theorem is deferred to Section \ref{sec_clt_proofs}. 

%
%

\begin{thm}\label{clt_by_local_martingale} We assume that the following conditions hold:
\begin{enumerateI}
\item[\textnormal{(D1)}]\ For all $t>0$ it holds that
\begin{equation}\label{first_clt_section_theorem_2nd_mom_cond}
\sup_{n\in\mathbb{N}}\alpha_n\,\EX^n\int_0^t\Bigl[\Lambda^n(U^n_s,\Theta^n_s)\int_{K_n}\|z^n(\xi)-z^n(\Theta^n_s)\|_\mathcal{E}^2\,\mu^n\bigl((U^n_s,\Theta^n_s),\dif\xi\bigr)\,\dif s\Bigr]\,<\infty,
\end{equation}
and there exists an orthonormal basis $(\varphi_k)_{k\in\mathbb{N}}$ of $\mathcal{E}^\ast$ such that for all $k\in\mathbb{N}$ and all $(u,\theta^n)\in H\times K_n$ except on a set of potential zero\footnote{A set of potential zero is a subset of the state space of the process which the process almost surely never reaches.} 
\begin{equation}\label{alternative_tightness_cond}
\alpha_n\,\EX^n\Bigl[\int_0^t\langle \varphi_k,G^n(U^n_s,\Theta^n_s)\varphi_k\rangle_\mathcal{E}\,\dif s\,\Big|\,(U^n_0,\Theta^n_0)=(u,\theta^n)\Bigr]\ \leq\  \gamma_k\,C(t),
\end{equation}
where the constants $\gamma_k>0$, independent of $n,\,t$ and $(u,\theta^n)$, satisfy $\sum_{k\in\mathbb{N}}\gamma_k<\infty$, and the constant $C(t)>0$, independent of $n,\,k$ and $(u,\theta^n)$, satisfies $\lim_{t\to 0}C(t)=0$.

\item[\textnormal{(D2)}]\ For all $\beta>0$ and every $\Phi\in\mathcal{E}^\ast$ it holds that
\begin{equation}\label{conv_gen_support_thm}
\lim_{n\to\infty}\,\EX^n\Bigl[\int_0^t\Lambda^n(U^n_s,\Theta^n_s)\int_{\sqrt{\alpha_n}\,|\langle\Phi, z^n(\xi)-z^n(\Theta_s^n)\rangle_\mathcal{E}|>\beta}\,\mu^n\bigl((U^n_s,\Theta^n_s),\dif\xi\bigr)\,\dif s\Bigr]\,=\, 0\,.
\end{equation}

\item[\textnormal{(D3)}]\ Further, for all $\Phi\in\mathcal{E}^\ast$ and all $t>0$ it holds that
\begin{equation}\label{clt_martingale_cond_2}
\lim_{n\to\infty} \int_0^t\EX^n\big|\bigl\langle\Phi,G(u(s),p(s))\,\Phi\bigr\rangle_\mathcal{E}-\alpha_n\bigl\langle\Phi,G^n(U^n_s,\Theta^n_s)\,\Phi\bigr\rangle_\mathcal{E}\big|\,\dif s\, =\, 0\,.
\end{equation}
Finally, we assume that the jump heights of the rescaled martingales are almost surely uniformly bounded, i.e., there exists a constant $C<\infty$ such that it holds almost surely for all $n\in\mathbb{N}$ that
\begin{equation}\label{clt_martingale_cond_1}
\sup_{t\geq 0}\,\sqrt{\alpha_n}\,\|z^n(\Theta^n_t)-z^n(\Theta^n_{t-})\|_\mathcal{E}<C\,.
\end{equation}
\end{enumerateI}
Then it follows that the process $(\sqrt{\alpha_n}\,M_t^n)_{t\geq 0}$
converges weakly to an $\mathcal{E}$--valued centred diffusion process characterised by the covariance operator \eqref{diffusion_limit_covariance}.
\end{thm}

%
%

We now state a second version of the martingale central limit theorem wherein the limiting process is characterised by the convergence of the characteristic functions.

\begin{thm}\label{clt_via_char_functions} Assume that the laws of the martingales $(\sqrt{\alpha_n}\,M^n_t)_{t\geq 0}$ form a tight sequence, e.g., condition \textnormal{(D1)} is satisfied.
\begin{enumerateI}
\item[\textnormal{(D3')}]\ \, The convergence \eqref{clt_martingale_cond_2} holds and there exists a sequence $\beta_n>0$ decreasing to zero such that for all $\Phi\in\mathcal{E}^\ast$
\begin{eqnarray}\label{conv_2ndmoments}
\lefteqn{\lim_{n\to\infty}\,\alpha_n\,\EX^n\Bigl[\int_0^t\Lambda^n(U^n_s,\Theta^n_s)\int_{\sqrt{\alpha_n}\,|\langle\Phi, z^n(\xi)-z^n(\Theta^n_s)\rangle_\mathcal{E}|>\beta_n}}\\[1ex] &&\phantom{xxxxxxxxxxxxxxxxxxxx}\big|\langle\Phi, z^n(\xi)-z^n(\Theta^n_s)\rangle_{\mathcal{E}}\big|^2\,\mu^n\bigl((U^n_s,\Theta^n_s),\dif\xi\bigr)\,\dif s\Bigr]\,=\, 0.\nonumber
\end{eqnarray}
\end{enumerateI}
Then it follows that the process $(\sqrt{\alpha_n}\,M_t^n)_{t\geq 0}$
converges weakly to an $\mathcal{E}$--valued centred diffusion process characterised by the covariance operator \eqref{diffusion_limit_covariance}.
\end{thm}

The central condition of the convergence of the quadratic variation processes \eqref{clt_martingale_cond_2} is unchanged, however, the second, technical condition \eqref{clt_martingale_cond_1} in (D3) is changed due to the different method of proof. That is, condition \eqref{conv_2ndmoments} arises instead of \eqref{clt_martingale_cond_1} as an assumption on the distribution of the jump heights employing a characterisation of the limit process using convergence of characteristic functions instead of the local martingale problem. The significance for applications of condition \eqref{conv_2ndmoments} in contrast to \eqref{clt_martingale_cond_1} is that the former avoids the almost sure uniform bound on the jump heights in the latter. That is, arbitrarily large jumps are possible for each martingale in the sequence as long as their probability decreases sufficiently fast. Note that \eqref{conv_2ndmoments} is stronger than the similar condition (D2) in the preceding theorem. We omit the proof of the theorem which is an adaptation of the estimating procedures in \cite{Kurtz2,Wainrib1} to the infinite-dimensional setting. For details we refer to the PhD thesis of one of the present authors \cite{RiedlerPhD}.

\begin{rem} We remark without proof that the assumptions (D1) and \eqref{clt_martingale_cond_2} imply the convergence of the trace processes, i.e., for all $T>0$
\begin{equation*}
\lim_{n\to\infty} \alpha_n\int_0^T\EX^n\textnormal{Tr}\,G^n(U^n_s,\Theta^n_s)\,\dif s \ =\ \int_0^T \textnormal{Tr}\,G(u(s),p(s))\,\dif s\,.
\end{equation*}
\end{rem}

%% file: section_langevin.tex
\subsection{Langevin approximation}\label{section_langevin_approximation}

Usually, e.g., in models of excitable membranes, one is ultimately interested in the dynamics of the continuous component. We have discussed in Section \ref{limit_appr_pdmps} that the coordinate functions $z^n_i$, \mbox{$i=1,\ldots,m$,} carry all the information needed for the \mbox{dynamics} of the continuous component $(U^n_t)_{t\geq 0}$. Therefore, the knowledge of the coordinate process $(z^n(\Theta^n_t))_{t\geq 0}$, or a close approximation thereof, is sufficient for many applications. From this point of view the significance of  the law of large numbers and the martingale central limit theorem is that they provide a justification of an approximation of the processes $(U^n_t,z^n(\Theta^n_t))_{t\geq 0}$ for large enough $n$ by a deterministic evolution equation, on the one hand, and, as we argue in this section, by a stochastic partial differential equation on the other hand.
 
To this end we first discuss representations of the limiting diffusion in Theorems \ref{clt_via_char_functions} and \ref{clt_by_local_martingale} as a stochastic integral. By definition $G(u(s),p(s))\circ\iota$ is a non-negative, self-adjoint trace class operator acting on $\mathcal{E}$, hence there exists a unique non-negative square root, i.e., a non-negative operator $\sqrt{G(u(s),p(s))\circ\iota}$ such that $G(u(s),p(s))\circ\iota=\sqrt{G(u(s),p(s))\circ\iota}\circ\sqrt{G(u(s),p(s))\circ\iota}$. Let $(W_t)_{t\geq 0}$ be a standard cylindrical Wiener process on $\mathcal{E}$ with covariance operator given by the identity (cf.~\cite{DaPratoZabczyk,PrevotRoeckner}). Then, as
\begin{equation*}
\EX\int_0^t \textnormal{Tr}\bigl(\sqrt{G(u(s),p(s))\circ\iota}\sqrt{I}\bigr)\bigl(\sqrt{G(u(s),p(s))\circ\iota}\sqrt{I}\bigr)^\ast\,\dif s = \int_0^t\textnormal{Tr}\, G(u(s),p(s))\,\dif s<\infty \,,
\end{equation*}
the mapping $t\mapsto\sqrt{G(u(s),p(s))\circ\iota}$ is a valid integrand process for a stochastic integral with respect to $(W_t)_{t\geq 0}$. That is, the process $(Z_t)_{t\geq 0}$ defined for all $t\geq 0$ by
\begin{equation}\label{rep_langevin_stoch_int}
Z_t:=\int_0^t \sqrt{G(u(s),p(s))\circ\iota}\,\dif W_s
\end{equation}
is an $\mathcal{E}$--valued Gaussian process with continuous sample paths and independent increments which, in addition, is also a square-integrable martingale. Moreover, the process has the covariance given by the operator $\int_0^tG(u(s),p(s))\,\dif s$. Therefore, due to unique definition of Gaussian processes via their covariance operators, the process $(Z_t)_{t\geq 0}$ is a version of the limiting diffusion identified  for the sequence of martingales $(\sqrt{\alpha_n}\,M^n_t)_{t\geq 0}$.

Hence, formally inserting the limits into the decomposition of the PDMP we obtain that the \emph{Langevin approximation} $(\ntilde U^n_t,\ntilde P^n_t)_{t\geq 0}$ of $(U^n_t,z^n(\Theta^n_t))_{t\geq 0}$ is given by the solution of the system of stochastic partial differential equations
\begin{equation}\label{general_langevin_approx}
\left.\begin{array}{rcl}
\dif\ntilde U_t^n&=&\bigl(A(\ntilde P^n_t)\,\ntilde U^n_t+B(\ntilde P^n_t,\ntilde U^n_t)\bigr)\,\dif t\\[2ex]
\dif \ntilde P^n_t&=& F(\ntilde P^n_t,\ntilde U^n_t)\,\dif t + \frac{1}{\sqrt{\alpha_n}}\,\sqrt{G(\ntilde U^n_t,\ntilde P^n_t)}\,\dif W_t\,.
\end{array}\right. 
\end{equation}
The sequence of Langevin approximations $(\ntilde U^n_t,\ntilde P^n_t)_{t\geq 0}$ possesses the same asymptotic behaviour as the sequence of processes $(U^n_t,z^n(\Theta_n))_{t\geq 0}$. It is obvious that for $n\to\infty$ and thus $\alpha_n\to\infty$ the noise term in \eqref{general_langevin_approx} vanishes and the system approximates the deterministic solution $(u(t),p(t))_{t\geq 0}$ of the system \eqref{ODE_system}, just as was proven in the law of large numbers Theorem \ref{LLN} for the sequence of PDMPs. It poses no difficulties to make this statement precise in the form of a weak law of large numbers similar to Theorem \ref{LLN}. Thus for large enough $\alpha_n$ one might expect that equation \eqref{general_langevin_approx} produces a similar behaviour than the PDMP with the major advantage of being analytically (and numerically) to a great extent less complex.

In order to analyse properties of the Langevin approximation, clearly, well-posedness of the system \eqref{general_langevin_approx} has to be addressed first. This is suitably done within the \emph{variational approach} to stochastic partial differential equations. That is, equation \eqref{general_langevin_approx} is assumed to hold as an integral equation in $X^\ast\times\mathcal{E}^\ast$ in contrast to the \emph{semigroup approach} which defines the solution via the semigroup generated by the linear part of \eqref{general_langevin_approx} and the variation of constants formula. (Note that in its generic form \eqref{general_langevin_approx} does not neceassarily posses a fully linear part.) The variational approach reflects the approach of using weak solution to abstract evolution equations  defining the deterministic inter-jump motion of PDMPs taken in this paper. We refer to \cite[Sec.~1.3.1]{KaiLiu} for a concise introduction to the variational approach to SPDEs containing an \mbox{existence} and uniqueness theorem as well as further references. We do not pursue the issue of well-posedness of the Langevin approximation any further at this point, as we are of the opinion that this question is best addressed when analysing the Langevin approximation for particular models.

\begin{rem} The process \eqref{rep_langevin_stoch_int} is not necessarily the only stochastic integral process which is a version of the limiting diffusion. Let $U$ be another separable, real Hilbert space, where $U=\mathcal{E}$ is possible, and assume there exists an operator $Q\in L_1(U,U)$ (or $Q$ cylindrical) and a function\footnote{Here, $L^2((0,T),L_2(U,\mathcal{E}))$ denotes the space of square-integrable functions on $(0,T)$ taking values in the Hilbert-space of Hilbert-Schmidt operators from $U$ to $\mathcal{E}$.} $g\in L^2((0,T),L_2(U,\mathcal{E}))$ for all $T>0$ such that $G(u(t),p(t))\circ\iota=g(u(s),p(s))\circ Q\circ g^\ast(u(s),p(s))$ for all $t\geq 0$. Then, the process $(Z_t^Q)_{t\geq 0}$ defined by the stochastic integral
\begin{equation}\label{rep_langevin_stoch_int_alternative}
Z_t^Q:=\int_0^t g(u(s),p(s))\,\dif W^Q_s\,,
\end{equation}
where $(W_t^Q)_{t\geq 0}$ is an $\mathcal{E}$-valued $Q$--Wiener process, has the same quadratic variation as $(Z_t)_{t\geq 0}$ hence the processes coincide in distribution. Then starting from the representation \eqref{rep_langevin_stoch_int_alternative} the Langevin approximation is given by \eqref{general_langevin_approx} with the obvious changes in the diffusion term. We note that in finite dimensions the non-uniqueness, see, e.g., \cite[Chap.~8]{Arnold2}, of a stochastic integral associated to a given covariance matrix can be exploited to improve the speed of numerical approximations in Monte-Carlo simulations of diffusion approximations by choosing an optimal diffusion coefficient structure, see \cite{Melykuti}. In infinite-dimensions the question of a practical implication of choosing a diffusion approximation based on \eqref{rep_langevin_stoch_int_alternative} over \eqref{general_langevin_approx} needs, to the best of our knowledge, still to be addressed.
\end{rem}

%% file: LLN_proof.tex
\subsection{Proof of Theorem \ref{LLN} (Law of large numbers)}\label{section_proof_LLN}

The central argument of the subsequent proof is an appropriate application of Gronwall's Lemma such that the upper bound satisfies the convergence in probability. Here the estimating procedure yielding the estimate to which Gronwall's Lemma is applied necessitates careful attention due to more intricate regularity \mbox{aspects} of solutions to abstract evolution equations in contrast to solutions of ODEs in Euclidean space.\medskip




The continuous component $U^n_t$ of each PDMP is in between successive jump times the weak solution of an abstract evolution equation. Similarly $u(t)$ is the weak solution of the abstract evolution equation \eqref{ODE_system}. Therefore also the difference of the two paths is in between jump times the weak solution of an abstract evolution equation. It thus holds due to \cite[Sec.~5.9, Thm.~3]{Evans} for almost all $t$ that
\begin{eqnarray*}
\lefteqn{\frac{\dif}{\dif t}\|U^n_t-u(t)\|_H^2=}\\[1ex]
&& 2\bigl\langle A(z^n(\Theta^n_t))\, U^n_t+B(U^n_t,z^n(\Theta^n_t))-A(p(t))\,u(t)-B(u(t),p(t))\,,\,U^n_t-u(t)\bigr\rangle_X\,.
\end{eqnarray*}
Integrating this equation we obtain the integral equation
\begin{eqnarray}\label{int_eq_1}
\lefteqn{\|U^n_{t_1}-u(t_1)\|_H^2\ =\ \|U_{t_0}^n-u_{t_0}\|^2_H}\\[2ex]&&\hspace{-20pt}\mbox{}+2\!\int_{t_0}^{t_1}\!\!\bigl\langle A(z^n(\Theta^n_s))\, U^n_s\!+\!B(U^n_s,z^n(\Theta^n_s))\!-\!A(p(s))\,u(s)\!-\!B(u(s),p(s))\,,\,U^n_s\!-\!u(s)\bigr\rangle_X\,\dif s,\nonumber
\end{eqnarray}
which is valid for almost all $t_0,\,t_1$ in between two successive jump times. Since both sides of equation \eqref{int_eq_1} are continuous the equality \eqref{int_eq_1} holds \emph{for all} $t_0,t_1$ between successive jump times. Moreover, as $U^n_t$ is continuous also at jump times it follows that equation \eqref{int_eq_1} holds for all $t\in[0,T]$, i.e., we have
\begin{eqnarray}\label{int_eq}
\lefteqn{\|U^n_{t}-u(t)\|_H^2\ =\ \|U_{0}^n-u_{0}\|^2_H}\\[1ex]&&\hspace{-20pt}\mbox{}+2\!\int_{0}^{t}\!\bigl\langle A(z^n(\Theta^n_s))\, U^n_s\!+\! B(U^n_s,z^n(\Theta^n_s))\!- \! A(p(s))\,u(s)\!-\! B(u(s),p(s))\,,\,U^n_s\!-\! u(s)\bigr\rangle_X\,\dif s\,.\nonumber
\end{eqnarray}
Next we employ the one-sided Lipschitz condition \eqref{onesided_lip} to estimate the integral in the right hand side of equation \eqref{int_eq}. This yields the inequality
\begin{equation}
\|U_t^n-u_t\|^2_H
\ \leq\ \|U_0^n-u_0\|^2_{H}+2L_1\int_0^t \|U_s^n-u(s)\|_{H}^2\, \dif s
+2L_1\sum_{j=1}^m\int_0^t\|z^n_j(\Theta^n_s)-p_j(s) \|^2_{E}\, \dif s.\label{growth_ineq_v}
\end{equation}
The overall aim is to apply Gronwall's inequality to the growth inequality \eqref{growth_ineq_v}. Therefore, in the next step we derive a control on the terms $\|z^n_j(\Theta^n_s)-p_j(s) \|^2_{E}$ in the right hand side of inequality \eqref{growth_ineq_v}. As $p$ is a solution of \eqref{ODE_system} satisfying \eqref{det_limit_int_equation} we obtain for every functional $\phi\in E^\ast$ a decomposition
\begin{eqnarray}\label{decomposition}
\lefteqn{\langle\phi, z_j^n(\Theta_t^n)-p_j(t) \rangle_E\ = \ \langle\phi,z_j^n(\Theta^n_0)-p_j(0)\rangle_E}\\[1ex]
&&\phantom{xxx}\mbox{}+\int_0^t \bigl[\mathcal{A}^n \langle\phi,z_j^n(\cdot)\rangle_E\bigr](U^n_s,\Theta^n_s)\,\dif s-\int_0^t\langle\phi, F_j(p(s),u(s)) \rangle_E\ \dif s +\langle\phi,M_j^n(t)\rangle_E,\nonumber
\end{eqnarray}
where the term $\langle\phi,M_j^n(t)\rangle_E$ has precisely the form \eqref{local_martingale} for all $t\in[0,T]$. Next we expand the decomposition \eqref{decomposition} to obtain
\begin{eqnarray*}
\lefteqn{\langle\phi, z_j^n(\Theta_t^n)-p_j(t) \rangle_E\ = \langle\phi,z_j^n(\Theta^n_0)- p_j(0)\rangle_E}\\[1.5ex]
&&\mbox{}+\langle\phi,M_j^n(t)\rangle_E+\int_0^t \bigl[\mathcal{A}^n \langle\phi,z_j^n(\cdot)\rangle_E\bigr](U^n_s,\Theta^n_s)-\langle\phi, F_j(z^n(\Theta^n_s),U_s^n)\rangle_E\,\dif s\\
&&\mbox{}+\int_0^t\bigl\langle\phi, F_j(z^n(\Theta^n_s),U_s^n)-F_j(p(s),u(s))\bigr\rangle_E\,\dif s\,.
\end{eqnarray*}
We take the supremum over all $\phi\in E^\ast$ with $\|\phi\|_{E^\ast}\leq 1$ on both sides of this equation, square both sides and apply to the right hand side the inequality $|a_1+\ldots+a_k|^2\leq k(|a_1|^2+\ldots+|a_k|^2)$ and the Cauchy-Schwarz inequality which yields
\begin{eqnarray*}
\lefteqn{\|z_j^n(\Theta_t^n)-p_j(t) \|_E^2}\nonumber\\[2ex]
&\hspace{-5pt}\!\!\leq\! & \!\!\!\!4\,\|z_j^n(\Theta^n_0)\!-\! p_j(0)\|^2_E+4\,\|M_j^n(t)\|^2_E+4\Bigl(\int_0^t\!\!\big\|F_j(z^n(\Theta^n_s),U_s^n\!-\!F_j(p(s),u(s))\big\|_E\,\dif s\Bigr)^2\\
&&\mbox{}+4\Bigl(\int_0^t \big\|\bigl[\mathcal{A}^n \langle\,\cdot\,,z_j^n(\cdot)\rangle_E\bigr](U^n_s,\Theta^n_s)- F_j(z^n(\Theta^n_s),U_s^n)\bigl\|_E\,\dif s\Bigr)^2\,.\nonumber
\end{eqnarray*}
We next apply the Lipschitz condition \eqref{det_sys_lipschitz} on $F$ and obtain the estimate
\begin{eqnarray}\label{proof_ineq1}
\lefteqn{\|z_j^n(\Theta_t^n)-p_j(t) \|_{E}^2} \nonumber\\[1ex]
&\hspace{-5pt}\!\!\leq\!\!\!& \!\!\! 4\,\|z_j^n(\Theta^n_0)- p_j(0)\|_{E}^2+4L_2\int_0^t
\! \|U^n_s-u(s)\|_H^2\,\dif s+4L_2\sum_{i=1}^m\int_0^t\!\|z_i^n(\Theta^n_s)-p_i(s)\|_{E}^2\,\dif s\nonumber\\
&&\mbox{}+4\Bigl(\int_0^t\big\|\bigl[\mathcal{A}^n \langle\,\cdot\,,z_j^n(\cdot)\rangle_E\bigr](U^n_s,\Theta^n_s)-F_j(z^n(\Theta^n_s),U_s^n)\big\|_{E}\,\dif s\Bigr)^2
+\|M_j^n(t)\|_{E}^2\,.
\end{eqnarray}
To further estimate this term we employ the convergence \eqref{uniform_martingale_bound} of the term $\|M^n_j\|_E$ and the convergence \eqref{finite_var_bound} of the generator. It follows by the definition of these limits that for every $\eps_1>0$ and every $\delta>0$ we can find an $N_{\eps_1,\delta}$ such that for all $n\geq N_{\eps_1,\delta}$ it holds due to \eqref{uniform_martingale_bound} for all $j=1,\ldots,m$ and all $t\in [0,T]$ that
\begin{equation*}
\|M_j^n(t)\|_{E}\leq \sqrt{\frac{\eps_1}{m}},
\end{equation*}
and due to \eqref{finite_var_bound} and the Continuous mapping Theorem that 
\begin{equation*}
\Bigl(\int_0^T\big\|\bigl[\mathcal{A}^n \langle\,\cdot\,,z_j^n(\cdot)\rangle_E\bigr](U^n_s,\Theta^n_s)-F_j(z^n(\Theta^n_s),U_s^n)\big\|_{E}\,\dif s\Bigr)^2\,\leq\, \frac{\eps_1}{m}
\end{equation*}
on a set $\Omega_1\subset\Omega$ satisfying $\Pr^n(\Omega\backslash\Omega_1)\leq\delta$ for all $n\geq N_{\eps_1,\delta}$.
Thus continuing to estimate only for paths on the set $\Omega_1$ we obtain from \eqref{proof_ineq1} the inequality
\begin{eqnarray}\label{proof_growth_est_1}
\lefteqn{\|z_j^n(\Theta_t^n)-p_j(t) \|_{E}^2\ \leq\ 4\,\|z_j^n(\Theta^n_0)- p_j(0)\|_{E}^2+5\,\frac{\eps_1}{m}}\\[1ex]
&&\phantom{xxxxxxxxxxxxx}\mbox{}+4L_2\int_0^t
 \|U^n_s-u(s)\|_H^2\,\dif s+4L_2\sum_{i=1}^m\int_0^t\|z_i^n(\Theta^n_s)-p_i(s)\|_{E}^2\,\dif s.\nonumber
\end{eqnarray}
In order to finally obtain the growth estimate suitable for an application of Gronwall's inequality we add inequality \eqref{growth_ineq_v} and inequalities \eqref{proof_growth_est_1} for all $j=1,\ldots,m$ which yields
\begin{eqnarray}\label{last_long_ineq}
\lefteqn{\|U_t^n-u(t)\|^2_H+\sum_{j=1}^m\| z_j^n(\Theta_t^n)-p_j(t) \|_{E}^2\ \leq\ \|U_0^n-u_0\|^2_{H}+4\sum_{j=1}^m\|z_j^n(\Theta^n_0)- p_j(0)\|_{E}^2}\\
&&\qquad\qquad\phantom{xxxxxxxx}\mbox{}+5\eps_1+C\int_0^t \|U_s^n-u(s)\|_{H}^2\ \dif s
+C\sum_{j=1}^m\int_0^t\|z^n_j(\Theta^n_s)-p_j(s) \|^2_{E}\,\dif s\nonumber
\end{eqnarray}
with constant $C=2L_1+4L_2m$.
An application of Gronwall's inequality to \eqref{last_long_ineq} yields
\begin{equation}\label{the_result_LLN_after_Gronwall}
\sup_{t\in[0,T]}\Bigl( \|U_t^n-u(t)\|^2_H+\sum_{j=1}^m\| z_j^n(\Theta_t^n)-p_j(t) \|_{E}^2\Bigr) \ \leq \ K_1\,\e^{C\,T}\,
\end{equation}
where
\begin{equation*}
K_1=\|U_0^n-u_0\|^2_{H}+4\sum_{j=1}^m\|z_j^n(\Theta^n_0)- p_j(0)\|_{E}^2+5\eps_1\,.
\end{equation*}
Finally, due to (C3), i.e., the convergence in probability of the initial conditions, it holds that for every $\eps_2>0$ we can find to every $\delta>0$ an $N_{\eps_2,\delta}$ such that on a set $\Omega_2\subset\Omega$ with $\Pr^n(\Omega\backslash\Omega_2)<\delta$ it holds for all $n\geq N_{\eps_2,\delta}$ that
\begin{equation}
\|U_0^n-u_0\|_H^2\leq \frac{\eps_2}{m+2},\quad \|z_j^n(\Theta^n_0)- p_j(0)\|_{E}^2\leq\frac{\eps_2}{4(m+2)}\quad \forall\, j=1,\ldots,m\,.
\end{equation}
Let $\eps,\,\delta>0$ be arbitrary. Then we obtain choosing $\eps_2=\eps\,\e^{-CT}$ and $\eps_1=\frac{\eps_2}{5(m+2)}$, thus $K_1=\eps_2$, that for all $n\geq N_{\eps,\delta}:= N_{\eps_1,\delta}\vee N_{\eps_2,\delta}$ it holds that
\begin{equation*}
\sup\nolimits_{t\in[0,T]}\Bigl( \|U_t^n-u(t)\|^2_H+\sum_{j=1}^m\| z_j^n(\Theta_t^n)-p_j(t) \|_{E}^2\Bigr)\, \leq\, \eps 
\end{equation*}
on the set $\Omega_1\cap\Omega_2$. Therefore it holds for all $n\geq N_{\eps,\delta}$ that
\begin{eqnarray*}
\Pr^n\Bigl[\sup\nolimits_{t\in[0,T]}\Bigl( \|U_t^n-u(t)\|^2_H+\sum_{j=1}^m\| z_j^n(\Theta_t^n)-p_j(t) \|_{E}^2\Bigr) \ > \ \eps\Bigr]&\leq& 2\delta\,.
\end{eqnarray*}
As $\delta$ and $\eps$ are arbitrary the statement \eqref{conv_estimate} follows.

%% file: tightness_theorem_proof.tex
The proof of Theorem \ref{clt_by_local_martingale} is split into three successive steps. In the first step we prove tightness of the sequence of martingales which guarantees the existence of a limit. Secondly, we show that any limit is a continuous process. Finally,  in the last step we prove that the limit is the specific diffusion process as stated in the theorem. The conditions (D1)--(D3) in Theorem \ref{clt_by_local_martingale} are such that each, in addition, to the preceding is needed in the successive steps of the proof. 



\medskip

\emph{Tightness}\medskip

In order to prove tightness of the sequence of $\mathcal{E}$--valued martingales $(\sqrt{\alpha_n}\,M^n_t)_{t\geq 0}$ it suffices to show that the following conditions are satisfied, cf.~\cite{Metivier} wherein general conditions for tightness of sequences of Hilbert space valued processes and, in particular, martingales are considered:
\begin{enumerate}
\item[(T1)] The sequence of initial conditions $(\sqrt{\alpha_n}\, M^n_0)_{n\geq 0}$ is tight.
\item[(T2)] For all $t\geq 0$ it holds that
\begin{equation}\label{metivier_corol_2}
\lim_{\delta\to\infty}\sup_{n\in\mathbb{N}}\Pr^n\bigl[\textnormal{Tr}\ll\!\!\sqrt{\alpha_n}\,M^n\!\!\gg\!\!_t>\delta\bigr]=0\,,
\end{equation}
and there exists an orthonormal basis $(\varphi_k)_{k\in\mathbb{N}}$ of $E^\ast$ such that for each $\eps>0$
\begin{equation}\label{metivier_corol_1}
\lim_{m\to\infty}\sup_{n\in\mathbb{N}}\,\Pr^n\Bigl[\sum\nolimits_{k>m} \langle \varphi_k, \ll\!\! \sqrt{\alpha_n}\,M^n\!\!\gg\!\!_t \,\varphi_k\rangle_E>\eps\Bigr]=0. 
\end{equation}
\item[(A)]  The sequence of the real-valued trace processes $(\textnormal{Tr}\ll\!\!\sqrt{\alpha_n}\,M^n\!\!\gg_t)_{t\geq 0}$, $n\geq \mathbb{N}$, satisfies the \emph{Aldous condition}: For every $T,\eps,\delta>0$ there exists a $h>0$ and an $N>0$ such that for any sequence of stopping times\footnote{Here every $\sigma^n$ is a stopping time on the respective probability space $(\Omega^n,\sF^n,(\sF^n_t)_{t\geq 0},\Pr^n)$.} $(\sigma^n)_{n\geq 0}$ with $\sigma^n\leq T$ it is valid that
\begin{equation}\label{aldous_condition}
\sup_{n\geq N}\sup_{0\leq s\leq h}\Pr^n\bigl[\,|\,\textnormal{Tr}\ll\!\!\sqrt{\alpha_n}\,M^n\!\!\gg_{\sigma^n+s}-\textnormal{Tr}\ll\!\!\sqrt{\alpha_n}\,M^n\!\!\gg_{\sigma^n}|\geq \delta\,\bigr]\leq\eps\,. 
\end{equation}
\end{enumerate}

We next establish the above conditions. First note that condition (T1) is trivially satisfied as $M^n_0=0$ for all $n>0$. Hence we proceed to condition (T2). In order to establish the first condition \eqref{metivier_corol_2} we use Markov's inequality to obtain the estimate
\begin{eqnarray*}
\Pr^n\bigl[\textnormal{Tr}\ll\!\! \sqrt{\alpha_n}\,M^n\!\!\gg_t\,>\delta\bigr]&\leq& \frac{\alpha_n}{\delta}\,\EX^n\Bigl[\int_0^t \textnormal{Tr}\,G^n(Y^n_s,\theta^n_s)\,\dif s\Bigr]\,,
\end{eqnarray*}
where the right hand side is finite due to assumption \eqref{first_clt_section_theorem_2nd_mom_cond}. Taking the supremum on both sides the same assumption implies \eqref{metivier_corol_2}.

Next, in order to show the second condition \eqref{metivier_corol_1} we employ Markov's inequality, the monotone convergence theorem (in order to change the order of expectation and the countable summation over all $k>m$), the form of the quadratic variation \eqref{explicit_def_cross_1} and inequality \eqref{alternative_tightness_cond} to obtain for the term in the left hand side the estimates
\begin{eqnarray*}
\Pr^n\bigl[\sum\nolimits_{k>m} \langle \varphi_k,\ll\!\!\sqrt{\alpha_n}\, M^n\!\!\gg_t \varphi_k\rangle_\mathcal{E}>\delta\bigr]&\leq& \frac{\alpha_n}{\delta}\,\EX^n\bigl[\sum\nolimits_{k> m} \langle \varphi_k,\ll\!\! M^n\!\!\gg_t \varphi_k\rangle_\mathcal{E}\bigr]\\
&\leq& \frac{1}{\delta}\Bigl(\sum\nolimits_{k>m} \gamma_k\Bigr)\,C(t)\,,
\end{eqnarray*}
where the upper bound is independent of $n\in\mathbb{N}$. Moreover, the property $\sum_{k\in\mathbb{N}}\gamma_k<\infty$ implies that $\lim_{m\to\infty} \sum_{k>m}\gamma_k=0$ and hence \eqref{metivier_corol_1} holds for all $t\geq 0$. 

Finally, it remains to show (A). Let $T,\delta>0$ and $\sigma^n<T$ be an arbitrary sequence of stopping times, then for all for all $h>0$ it holds that for $s\leq h$
\begin{eqnarray*}
\lefteqn{\Pr^n\Bigl[\,\big|\,\alpha_n\textnormal{Tr}\ll\!\! M^n\!\!\gg_{\sigma_n+s}-\,\alpha_n\textnormal{Tr}\ll\!\! M^n\!\!\gg_{\sigma_n}\big|\geq \delta\Bigr]}\\
&&\phantom{xxxxxxxxxxxxxxx}\leq\ \frac{\alpha_n}{\delta}\,\EX^n\Bigl[\int_{\sigma^n}^{\sigma^n+h}\textnormal{Tr}\,G^n(Y^n_r,\theta^n_r)\,\dif r\Bigr]\\ 
&&\phantom{xxxxxxxxxxxxxxx}=\ \frac{1}{\delta} \EX^n\Bigl[\sum_{k\in\mathbb{N}}\alpha_n\,\EX^{X_{\sigma^n}^n}\int_0^h\langle\varphi_k,G^n(Y^n_r,\theta^n_r)\varphi_k\rangle_{\mathcal{E}}\,\dif r\Bigr]\\
&&\phantom{xxxxxxxxxxxxxxx}\leq\ \frac{C(h)}{\delta}\sum_{k\in\mathbb{N}} \gamma_k.
\end{eqnarray*}
Here we have used Markov's inequality, the strong Markov property of the PDMP and the assumption \eqref{alternative_tightness_cond}. As the final upper bound is independent of $s$ and $n$ and converges to zero for $h\to0$ condition (A) follows.

%% file: proof_support_of_limit.tex
\medskip

\emph{Any Limit is a continuous process}

\medskip

In the preceding part of the proof we have established that the laws of the sequence of martingales $(\sqrt{\alpha_n}\, M^n_t)_{t\geq 0}$ are tight which is equivalent to there existence of a weakly convergent subsequence. We now prove that under the additional condition (D2) any cluster point of the sequence is a measure supported on $C(\rr_+,\mathcal{E})$. 
%
The method of proof follows the outline of \cite[Lemma 3.2]{Kallianpur} adapted for the stochastic processes being PDMPs on Hilbert spaces, the general setup in this study and the particular conditions (D1) and (D2) in Theorem \ref{clt_by_local_martingale} which differ from \cite{Kallianpur}. Furthermore, we have extended the result in \cite[Lemma 3.2]{Kallianpur}, which only considers convergence on finite time intervals $[0,T]$, to convergence on $D(\rr_+,\mathcal{E})$. In the following we employ the abbreviations $Z^n_t:=\sqrt{\alpha_n}\,M^n_t$ and $\Delta_t Z^n:=Z^n_t-Z^n_{t-}$, i.e., $(\Delta_t Z^n)_{t\geq 0}$ denotes the process of jump heights. Note that $\Delta_t Z^n=\sqrt{\alpha_n}\,\Delta_tz^n(\theta^n)$.

Further, let $\Pr^\ast$ denote an accumulation point of the sequence $(\Pr^n)_{n\in\mathbb{N}}$. Without loss of generality we use $\Pr^n,\,n\geq 1$, to also denote the subsequence converging weakly to $\Pr^\ast$. 
Furthermore, here $\Pr^n$ is understood as a law on the Skorokhod space $D(\rr_+,\mathcal{E})$ given by the pushforward measure of the process $(\sqrt{\alpha_n}\,M^n_t)_{t\geq 0}$. Then due to the Skorokhod Representation Theorem, e.g., \cite[Chap.~3,Thm.~1.8]{EthierKurtz}, there exists a probability space $(\Omega^o,\sF^o,\Pr^o)$ supporting $D(\rr_+,\mathcal{E})$--valued random variables $\zeta^n,\,n\geq 1$, and $\zeta^\ast$ with distributions $\Pr^n$ and $\Pr^\ast$, respectively, such that $\zeta^n$ converges to $\zeta^\ast$ almost surely with respect to $\Pr^o$. Further, it clearly holds that $\EX^n f(Z^n)=\EX^of(\zeta^n)$ for \mbox{suitable functionals $f$.}\medskip

We begin the proof with preliminary estimates on functions evaluated along the path of the PDMPs. These ultimately allow to infer that the process of jumps vanishes in the limit. Let $g$ be a measurable, bounded, non-negative function $g:\rr\to\rr$, that vanishes in a neighbourhood of $0$ and of $\infty$, that is, there exists a finite constant $C_g:=\sup_{x\in\rr} g(x)/x^2<\infty$. For such a function $g$ and any $\Phi\in\mathcal{E}^\ast$ we define the process
\begin{eqnarray*}
G^n_t(\langle\Phi,Z^n\rangle_\mathcal{E})&:=&\sum_{s\in(0,t]} g\bigl(\langle\Phi,\Delta_sZ^n\rangle_\mathcal{E}\bigr)\\
&&\hspace{-15pt}\mbox{}-\int_0^t\Lambda^n(U^n_s,\Theta^n_s)\int_{K_n} g\bigl(\sqrt{\alpha_n}\langle \Phi,z^n(\xi) -z^n(\Theta_{s}^n) \rangle_\mathcal{E}\bigr)\,\mu^n\bigl((U^n_s,\Theta^n_s)\bigr)\,\dif s\\
&=& \int_0^t\int_{K_n} g\bigl(\sqrt{\alpha_n}\langle \Phi,z^n(\xi)-z^n(\Theta_{s-}^n) \rangle_\mathcal{E}\bigr)\, M^n(\dif\xi,\dif s)\,,
\end{eqnarray*}
where $M^n$ is the martingale measure associated with the PDMP, and we infer that $G^n_t(\langle\Phi,Z^n\rangle_\mathcal{E})$ is a martingale. Note that the above summation over all $s\in (0,t]$ is well-defined as the PDMPs are regular and thus $g\bigl(\langle\Phi,\Delta_sZ^n\rangle_\mathcal{E}\bigr)$ is non-zero for only finitely many $s\leq t$. \medskip

The proof now proceeds as follows. We first show (a) that for all $t\geq 0$ the random variables $G^n_t(\langle\Phi,\zeta^n\rangle_\mathcal{E}),\,n\in\mathbb{N},$ are uniformly integrable for all $t$ and (b) that they converge to $\sum\nolimits_{s\in(0,t]} g\bigl(\langle\Phi,\Delta_s\zeta^n\rangle_\mathcal{E}\bigr)$ in probability. This allows to infer that the convergence result also holds as convergence in mean. 
In part (c) we then use these results to show that the jump heights of the canonical process of the law $\Pr^\ast$ are constantly zero almost surely. This implies that $\Pr^\ast\bigl(C([0,t],\mathcal{E})\bigr)=1$ for every $t>0$  where $C([0,t],\mathcal{E})$ is understood as the subset of $D(\rr_+,\mathcal{E})$ consisting of those c\`adl\`ag functions which are continuous up to and including time $t$. The proof is completed by (d) extending this result to $\Pr^\ast\bigl(C(\rr_+,\mathcal{E})\bigr)=1$.\medskip

(a)\quad To show that the sequence of random variables $G^n_t(\langle\Phi,\zeta^n\rangle_\mathcal{E}),\,n\in\mathbb{N}$, is uniformly integrable in the space $(\Omega^o,\sF^o,\Pr^o)$ it is sufficient that the second moments are uniformly bounded, cf.~\cite[Appendix, Prop.~2.2]{EthierKurtz}. The It{\^o}-isometry for real-valued stochastic integrals with respect to the associated martingale measures, which is implied by taking the expectation of the processes in \cite[Prop.~4.5.3]{Jacobsen}, yields
\begin{eqnarray*}
\lefteqn{\sup_{n\in\mathbb{N}} \EX^o|G^n_t(\langle\Phi,\zeta^n\rangle_\mathcal{E})|^2\ =\ \sup_{n\in\mathbb{N}}\EX^n|G^n_t(\langle\Phi,Z^n\rangle_\mathcal{E})|^2}\\
&=&\sup_{n\in\mathbb{N}}\EX^n\Bigl[\int_0^t\Lambda^n(U^n_s,\Theta^n_s)\int_{K_n} g\bigl(\sqrt{\alpha_n}\langle \Phi,z^n(\xi)-z^n(\Theta_{s}^n) \rangle_\mathcal{E}\bigr)^2\,\mu^n\bigl((U^n_s,\theta^n_s),\dif\xi\bigr)\,\dif s\Bigr].
\end{eqnarray*}
Therefore, employing the special structure of the map $g$  we obtain the estimate
\begin{eqnarray*}
\lefteqn{\sup_{n\in\mathbb{N}} \EX^o|G^n_t(\langle\Phi,\zeta^n\rangle_\mathcal{E})|^2}\\
&\leq&\! C_g\, \sup_{n\in\mathbb{N}}\,\alpha_n\EX^n\Bigl[\int_0^t\Lambda^n(U^n_s,\Theta^n_s)\int_{K_n} \big|\langle \Phi,z^n(\xi) -z^n(\Theta_{s}^n)\rangle_\mathcal{E}\big|^2\,\mu^n\bigl((U^n_s,\Theta^n_s),\dif\xi\bigr)\,\dif s\Bigr],
\end{eqnarray*}
where the right hand side is finite for every $t>0$ due to condition \eqref{first_clt_section_theorem_2nd_mom_cond} in (D1).
\medskip

(b)\quad In this part of the proof we establish convergence in probability of the random variables $G^n_t(\langle\Phi,\Delta\zeta^n\rangle_\mathcal{E})$. Let $\beta>0$ be such that $g(x)=0$ for $|x|\leq\beta$, i.e., the interval $(-\beta,\beta)$ is contained in the neighbourhood of $0$ whereon $g$ vanishes. Then we obtain using Markov's inequality and due to the boundedness of $g$ the estimates
\begin{eqnarray*}
\lefteqn{\Pr^o\Bigl[\sum\nolimits_{s\in(0,t]} g(\langle\Phi,\Delta_s\zeta^n\rangle_\mathcal{E})-G^n_t(\langle\Phi,\zeta^n\rangle_\mathcal{E})\, > \, \delta\Bigr]}\\
&\!\!=&\!\!\Pr^n\Bigl[\int_0^t\Lambda^n(U^n_s,\Theta^n_s)\int_{K_n}g\bigl(\sqrt{\alpha_n}\,\langle \Phi,z^n(\xi) -z^n(\Theta_{s}^n) \rangle_E\bigr)\,\mu^n\bigl((U^n_s,\Theta^n_s),\dif\xi\bigr)\,\dif s\, >\, \delta\Bigr]\\
&\!\!\leq &\!\! 
\frac{1}{\delta}\,\EX^n\Bigl[\int_0^t\Lambda^n(U^n_s,\Theta^n_s)\int_{K_n}g\bigl(\sqrt{\alpha_n}\langle \Phi,z^n(\xi) -z^n(\Theta_{s}^n) \rangle_E\bigr)\,\mu^n\bigl((U^n_s,\Theta^n_s)\bigr)\,\dif s\Bigr]\\
&\!\!\leq &\!\!
\frac{\sup_{x\in\rr}|g(x)|}{\delta}\, \EX^n\Bigl[\int_0^t\Lambda^n(U^n_s,\Theta^n_s) \int
_{\sqrt{\alpha_n}\,|\langle \Phi,z^n(\xi) -z^n(\Theta^n_s)\rangle_E |> \beta}
\mu^n\bigl((U^n_s,\Theta^n_s),\dif\xi\bigr)\,\dif s\Bigr].
\end{eqnarray*}
Thus due to condition \eqref{conv_gen_support_thm} in (D2) it holds that
\begin{equation*}
\lim_{n\to\infty}\Pr^o\Bigl[\sum\nolimits_{s\in(0,t]} g(\langle\Phi,\Delta_s\zeta^n\rangle_\mathcal{E})-G^n_t(\langle\Phi,\zeta^n\rangle_\mathcal{E})\, > \, \delta\Bigr] =0\,.
\end{equation*}
Moreover, it holds on $(\Omega^o,\sF^o,\Pr^o)$ almost surely that 
\begin{equation*}
\lim_{n\to\infty}\sum_{s\in(0,t]} g(\langle\Phi,\Delta_s\zeta^n \rangle_\mathcal{E}) = \sum_{s\in(0,t]} g(\langle\Phi,\Delta_s\zeta^\ast \rangle_\mathcal{E})\,.
\end{equation*}
Therefore, combining these two convergence results we obtain that 
\begin{equation}\label{support_proof_main_convergence}
G^n_t(\langle\Phi,\zeta^n\rangle_\mathcal{E}) \longrightarrow \sum_{s\in(0,t]} g(\langle\Phi,\Delta_s\zeta^\ast \rangle_\mathcal{E})
\end{equation}
holds as convergence in probability. \medskip

(c)\quad From parts (a) and (b) we infer that \eqref{support_proof_main_convergence} also holds as convergence in mean. Together with Jensen's inequality this implies
\begin{eqnarray*}
\lefteqn{\lim_{n\to\infty} \Big|\EX^o\Bigl(G^n_t(\langle\Phi,\Delta_s \zeta^n\rangle_\mathcal{E})-\sum_{s\in(0,T]} g(\langle\Phi,\Delta_s\zeta^\ast\rangle_\mathcal{E})\Bigr)\Big|}\\
&&\phantom{xxxxxxxxxxxxx}\leq\ \lim_{n\to\infty} \EX^o\Big|G^n_t(\langle\Phi,\Delta_s \zeta^n\rangle_\mathcal{E})-\sum_{s\in(0,t]} g(\langle\Phi,\Delta_s\zeta^\ast\rangle_\mathcal{E})\Big|\ =\ 0\,,
\end{eqnarray*}
and hence we infer that
\begin{equation}\label{the_limit_we_derived_uniform_integrability_for}
\EX^o \sum_{s\in(0,t]} g(\langle\Phi,\Delta_s \zeta^\ast\rangle_\mathcal{E})
\,=\, \lim_{n\to\infty} \EX^o G^n_t(\langle\Phi,\zeta^n\rangle_\mathcal{E})\,. 
\end{equation}
Furthermore, $G_t^n(\langle\Phi,Z^n\rangle_\mathcal{E})$ is a martingale which satisfies $G_0^n(\langle\Phi,Z^n\rangle_\mathcal{E})=0$. This, in turn, implies that $\EX^nG_t^n(\langle\Phi,Z^n\rangle_\mathcal{E})=0$ for every $n\in\mathbb{N}$. Therefore we obtain due to \eqref{the_limit_we_derived_uniform_integrability_for}
\begin{eqnarray}\label{some_add_property_suppoRT_of_limiz}
\EX^\ast \sum_{s\in(0,t]} g(\langle\Phi,\Delta_s Z\rangle_\mathcal{E}) &=& \EX^o \sum_{s\in(0,t]} g(\langle\Phi,\Delta_s \zeta^\ast\rangle_\mathcal{E})\\
&=& \lim_{n\to\infty} \EX^o G^n_t(\langle\Phi,\zeta^n\rangle_\mathcal{E})
\ =\ \lim_{n\to\infty} \EX^n G^n_t(\langle\Phi,Z^n\rangle_\mathcal{E})
\ =\ 0\,.\nonumber
\end{eqnarray}
In a next step, let $g_m$ be a sequence of functions satisfying the properties for functions $g$ proposed above. Further we assume that the functions $g_m(x)$ increase pointwise to $x^2$ for $m\to\infty$ (for an example of such functions we refer to \cite{Kallianpur}). Then due to the monotone convergence theorem it holds that
\begin{equation*}
\lim_{m\to\infty} \EX^\ast \sum_{s\in(0,t]} g_m(\langle\Phi,\Delta_sZ\rangle_\mathcal{E}) \,=\, \EX^\ast \sum_{s\in(0,t]} |\langle\Phi,\Delta_sZ\rangle_\mathcal{E}|^2\,.
\end{equation*}
Furthermore, the limiting expectation in the right hand side is zero as each element of the sequence of expectations in the left hand side is zero due to \eqref{some_add_property_suppoRT_of_limiz}. Next we choose $\Phi$ to be an element of an orthonormal basis $(\varphi_k)_{k\in\mathbb{N}}$ of $\mathcal{E}$ and sum the expectations over all elements of the basis yielding
\begin{equation*}
\sum_{k\in\mathbb{N}}\EX^\ast \sum\nolimits_{s\in(0,t]} |\langle\varphi_k,\Delta_sZ\rangle_\mathcal{E}|^2\,. 
\end{equation*}
Due to the dominated convergence theorem we can interchange the countable summation and the expectation and, as the PDMP is regular, we afterwards interchange the resulting two summation inside the expectation. Then Parseval's identity yields
\begin{equation*}
\EX^\ast \sum_{s\in(0,t]} \|\Delta_s Z\|_\mathcal{E}^2 \,=\,0\,. 
\end{equation*}
As the non-negative random variable inside the expectation is zero only for continuous paths of the process $(Z_s)_{s\in[0,t]}$ we infer that almost all paths are continuous, i.e., $\Pr^\ast\bigl(C([0,t],\mathcal{E})\bigr)=1$.\medskip

(d)\quad To conclude the proof let $t_k,\,k\in\mathbb{N}$, be a sequence of times increasing to infinity then
\begin{equation*}
C(\rr_+,\mathcal{E})=\bigcap_{k\in\mathbb{N}} C([0,t_k],\mathcal{E})\,,
\end{equation*}
and the events in the right hand side satisfy $C([0,t_{k+1}],\mathcal{E})\subseteq C([0,t_k],\mathcal{E})$. The properties of a probability measure thus yield
\begin{equation*}
\Pr^\ast\bigl(C(\rr_+,\mathcal{E})\bigr)=\lim_{k\to\infty}\,\Pr^\ast\bigl(C([0,t_k],\mathcal{E})\bigr)=1\,,
\end{equation*}
that is a process with distribution given by the limit $\Pr^\ast$ possesses almost surely continuous paths.

%% file: martingale_problem_proof.tex
\medskip

\emph{Limit is a diffusion process}\medskip

In the final part of the proof we uniquely characterise the limit of the sequence of martingales $(\sqrt{\alpha_n}\,M^n_t)_{t\geq 0}$ under the additional assumptions (D3). The method of proof is via the local martingale problem motivated by a proof presented in \cite{Metivier}, i.e., the limiting probability measure is the unique solution to a particular martingale problem. The author in \cite{Metivier} considers Hilbert space valued stochastic integral equations driven by Hilbert space valued martingales with state dependent quadratic variation. A central limit theorem for the martingales is presented. The arguments of the subsequent proof are closely related to \cite{Metivier}. This is as the general result on martingales associated with PDMPs, which we have proven in Section \ref{section_associated_martingale}, result in the problem in this part of the proof to be of the same underlying structure as in \cite{Metivier}. One difference, however, is that the present conditions (D1)--(D3) are more general than the conditions in \cite{Metivier} and adapted to the PDMP setup, hence some estimates differ.

As in the preceding part of the proof we interpret the sequence of martingales $(\sqrt{\alpha_n}\,M^n_t)_{t\geq 0}$ defined on the probability spaces $(\Omega^n,\sF^n,(\sF^n_t)_{t\geq 0},\Pr^n)$ as random variables on the space $D(\rr_+,\mathcal{E})$ equipped with its natural $\sigma$-field $\mathcal{D}$. Further, laws on the canonical space are given by the pushforward measure. In order to simplify the notation we denote the laws on the canonical space also by $\Pr^n$. Due to results in the preceding two parts of the proof we know the sequence $\Pr^n$, $n\in\mathbb{N}$, admits a limit $\Pr^\ast$ supported on $C(\rr_+,\mathcal{E})$. We use $(\zeta_t)_{t\geq 0}$ to denote the canonical process %
on $D(\rr_+,\mathcal{E})$ which is a version of the martingale $(\sqrt{\alpha_n}\,M^n_t)_{t\geq 0}$ under the push-forward maesure $\Pr^n$ for all $n\in\mathbb{N}$ or of the weak limit under the measure $\Pr^\ast$.

In the following we prove that the limit $\Pr^\ast$ is a solution to a local martingale problem the unique solution of which is an $\mathcal{E}$--valued centered diffusion process with covariance operator $C(t)\in L_1(\mathcal{E}^\ast,\mathcal{E})$ as given in \eqref{diffusion_limit_covariance}. For any twice continuously differentiable function $f:\mathcal{E}\to\rr$ the extended generator $\mathcal{A}f$ of such a diffusion is given by
\begin{equation*}
\mathcal{A}f(x,t)=\frac{1}{2}\textnormal{Tr}\,(D^2f(x)\circ G(t))\,. 
\end{equation*}
Then, in order to uniquely characterise the solution to the local martingale problem connected with this generator and supported on the space $C(\rr,\mathcal{E})$ it suffices to consider mappings $f$ of the form $\langle\Phi,\cdot\rangle_\mathcal{E}$ and $\langle\Phi,\cdot\rangle_\mathcal{E}^2$ for all $\Phi\in\mathcal{E}^\ast$, cf.~\cite{Metivier}. %
That is, we have to show that the canonical process $\zeta_t$ is such that for all $\Phi\in \mathcal{E}^\ast$ the processes $\langle\Phi,\zeta_t \rangle_{\mathcal{E}}$ and
\begin{equation}\label{martingale_problem_proof_2} 
\langle\Phi,\zeta_t \rangle_{\mathcal{E}}^2-\int_0^t \langle\Phi,G(u_s,p_s)\Phi \rangle_{\mathcal{E}}\,\dif s
\end{equation}
are $\Pr^\ast$-local martingales. We start introducing some notation and then show in parts (a) and (b) the local martingale properties of the two indicated processes on the canonical space $D(\rr_+,\mathcal{E})$. 
\medskip

As before we use $Z^n_t:=\sqrt{\alpha_n}\,M^n_t$ and $\Delta_t Z^n:=Z^n_t-Z^n_{t-}$. Further, as indicated above the notation is such that we use $\Pr^n$ and $\EX^n$ to denote probabilities and expectations on the original given measurable spaces $(\Omega^n,\sF^n)$ as well as on the canonical space $(D(\rr_+,\mathcal{E}),\mathcal{D})$. That is, e.g., $\EX^n f(Z^n_t)=\EX^n f(\zeta_t)$ for any bounded function $f$, where the former is the expectation taken on the original space $(\Omega^n,\sF^n,\Pr^n)$ and the latter the expectation on the canonical space of c\`adl\`ag processes with respect to the pushforward measure. 
Furthermore, we employ the It{\^o}-formula \cite[Thm.~25.7]{MetivierSemi} for smooth 
functions $f\in C^\infty_c(\rr)$ applied to semi-martingales. For the particular choice of the semi-martingales being the real martingales $\langle\Phi,Z^n_t\rangle_\mathcal{E}$ the It{\^o}-formula reads
\begin{eqnarray}\label{clt_applied_oto_formula}
f\bigl(\langle\Phi,Z^n_t\rangle_{\mathcal{E}}\bigr)&=&\frac{1}{2}\int_0^t f''\bigl(\langle\Phi,Z^n_{s-}\rangle_{\mathcal{E}}\bigr)\,\bigl(\langle\Phi,\alpha_n\!\!\ll\!\!M^n\!\!\gg\!\!_t\Phi\rangle_{\mathcal{E}}\bigr) \dif s \nonumber\\
&&+\ \sum_{s\leq t}\Bigl[f(\langle\Phi,Z^n_s\rangle_\mathcal{E})-f(\langle\Phi,Z^n_{s-}\rangle_\mathcal{E})-\langle\Phi,\Delta_sZ^n\rangle_\mathcal{E}\, f'(\langle\Phi,Z^n_{s-}\rangle_\mathcal{E})\Bigr]\nonumber\\
&& -\frac{1}{2}\sum_{s\leq t}\Bigr[\langle\Phi,\Delta_sZ^n\rangle^s_{\mathcal{E}}\,f''(\langle\Phi,Z^n_t\rangle_\mathcal{E})\Bigl]+ M^{f,n}_t
\end{eqnarray}
where $(M^{f,n}_t)_{t\geq 0}$ is some martingale on $(\Omega^n,\sF^n,(\sF^n_t)_{t\geq 0},\Pr^n)$ depending on $Z^n$ and $f$.

Next, we introduce on the canonical space for all positive $\rho$ the stopping times \mbox{$\tau_\rho:=\inf\{t\in\rr_+\,|\,\|\zeta_t\|_\mathcal{E}>\rho\}$} and note that due to the bound \eqref{clt_martingale_cond_1} in (D3) on the jump heights we have that for any law $\Pr^n,\,n\geq 1$, it holds almost surely
\begin{equation}
\|\zeta_{\tau_\rho}\|_\mathcal{E}\leq \rho+C\,.
\end{equation}
Analogously we define the stopping times $\tau^n_\rho:=\inf\{t\in\rr_+\,|\,\|Z^n_t\|_\mathcal{E}>\rho\}$ on the spaces $(\Omega^n,\sF^n,(\sF^n_t)_{t\geq 0},\Pr^n)$.

Finally, as already mentioned $(\mathcal{D}_t)_{t\geq 0}$ denotes the natural filtration on the canonical space. Then for $A\in\mathcal{D}_t$ we define $A^n:=(Z^n)^{-1} F\in\sF^n_t$ its preimage with respect to the random variable $Z^n$. We now proceed to show that the two processes $\langle\Phi,\zeta_t\rangle_\mathcal{E}$ and \eqref{martingale_problem_proof_2} are indeed local martingales with respect to the limit measure $\Pr^\ast$.\medskip

(a)\quad Let $\Phi\in\mathcal{E}^\ast$ be fixed and we choose for every $\rho$ a smooth function $f_\rho\in C^\infty_c(\rr)$ which satisfies $f_\rho(x)=x$ if $|x|\leq \|\Phi\|_{\mathcal{E}^\ast}(\rho+C)$ and thus $f'(x)=1$ and $f''(x)=0$ for $|x|\leq \|\Phi\|_{\mathcal{E}^\ast}(\rho+C)$. Therefore it holds for $t< \tau^n_\rho$, which implies the estimate $|\langle\Phi,Z^n_{t-}\rangle_\mathcal{E}|\leq \|\Phi\|_{\mathcal{E}^\ast}(\rho+C)$, that
\begin{equation*}
f_\rho''(\langle\Phi,Z^n_{t-}\rangle_\mathcal{E})=0 
\end{equation*}
and
\begin{equation*}
f_\rho(\langle\Phi,Z^n_t\rangle_\mathcal{E})-f_\rho(\langle\Phi,Z^n_{t-}\rangle_\mathcal{E})-\langle\Phi,\Delta_tZ^n\rangle_\mathcal{E}\, f_\rho'(\langle\Phi,Z^n_{s-}\rangle_\mathcal{E}) =0.
\end{equation*}
It follows that applying the It{\^o}-formula \eqref{clt_applied_oto_formula} to the function $f_\rho$ and the martingale $Z^n_{t\wedge \tau^n_\rho}$ all terms besides the martingale $M^{n,f_\rho}$ vanish in the the right hand side. 
Therefore we obtain for $t_2\geq t_1$ and all $A\in\mathcal{D}_{t_1}$ that
\begin{eqnarray}
\EX^n\Bigl[\mathbb{I}_{A}\,\Bigl(\langle\Phi,\zeta_{t_2\wedge\tau_\rho}\rangle_\mathcal{E}-\langle\Phi,\zeta_{t_1\wedge\tau_\rho}\rangle_\mathcal{E}\Bigr)\Bigr]\nonumber \!&\!=\!&\! \EX^n\Bigl[\mathbb{I}_{A}\,\Bigl(f_\rho\bigl(\langle\Phi,\zeta_{t_2\wedge\tau_\rho}\rangle_\mathcal{E}\bigr)-f_\rho\bigl(\langle\Phi,\zeta_{t_1\wedge\tau_\rho}\rangle_\mathcal{E}\bigr)\Bigr)\Bigr]\nonumber\\[1ex]
&&\hspace{-35pt}=\ \, \EX^n\Bigl[\mathbb{I}_{A^n}\,\Bigl(f_\rho\bigl(\langle\Phi,Z^n_{t_2\wedge\tau_\rho^n}\rangle_\mathcal{E}\bigr)-f_\rho\bigl(\langle\Phi,Z^n_{t_1\wedge\tau_\rho^n}\rangle_\mathcal{E}\bigr)\Bigr)\Bigr]\nonumber\\[1ex]
&&\hspace{-35pt} =\ \,0\,.\label{martingale_problem_proof_zero_exp}
\end{eqnarray}
The proof of the first martingale property is concluded as in \cite{Metivier}: The mapping $\zeta\to f_\rho(\langle\Phi,\zeta_{t_2\wedge\tau_\rho}\rangle_\mathcal{E})$ is almost surely (with respect to the probability $\Pr^\ast$) continuous and as $\Pr^n$ converges weakly to $\Pr^\ast$ it holds due to \eqref{martingale_problem_proof_zero_exp} that
\begin{equation*}
\EX^\ast\Bigl[\mathbb{I}_{A}\,\Bigl(\langle\Phi,\zeta_{t_2\wedge\tau_\rho}\rangle_\mathcal{E}-\langle\Phi,\zeta_{t_1\wedge\tau_\rho}\rangle_\mathcal{E}\Bigr)\Bigr]=0\,. 
\end{equation*}
Here we have employed a weaker version of the continuous mapping theorem, see, e.g., \cite[Thm.~2.7]{Billingsley}\,.

We infer from the definition of the conditional expectation that the stopped processes are martingales. Furthermore, as $\zeta_t$ possesses continuous paths almost surely under the measure $\Pr^\ast$ it holds that $\tau_\rho$ diverges to $\infty$ almost surely for $\rho\to\infty$. Hence, we can find a sequence of stopping times $\tau_{\rho_k}$, $k\in\mathbb{N}$, such that $\tau_{\rho_k}\to\infty$ almost surely for $k\to\infty$. 
Thus, the process $\langle\Phi,\zeta_t\rangle$ is a local martingale with respect to $\Pr^\ast$.\medskip

(b)\quad For the second class of processes we consider smooth functions $g_\rho\in C^\infty_c(\rr)$ such that $g_\rho(x)=x^2$ for all $|x|\leq \|\Phi\|_{\mathcal{E}^\ast}(\rho+C)$. Starting from the definition of the conditional expectation as in \eqref{martingale_problem_proof_zero_exp} we obtain 
\begin{eqnarray*}
\lefteqn{\EX^n\Bigl[\mathbb{I}_A\,\Bigl(\langle\Phi,\zeta_{t_2\wedge\tau_\rho}\rangle_\mathcal{E}^2-\int_0^{t_2\wedge\tau_\rho} \langle\Phi,G(u(s),p(s))\Phi\rangle_\mathcal{E}\,\dif s-\langle\Phi,\zeta_{t_1\wedge\tau_\rho}\rangle_\mathcal{E}^2}\\
&&\phantom{xxxxxxxxxxxXxxxxxxxxxxxxxxxxxxxxxx}\mbox{} + \int_0^{t_1\wedge\tau_\rho} \langle\Phi,G(u(s),p(s))\Phi\rangle_\mathcal{E}\,\dif s\Bigr)\Bigr]\\
&&\phantom{xx}=\
\EX^n\Bigl[\mathbb{I}_A\,\Bigl(\langle\Phi,\zeta_{t_2\wedge\tau_\rho}\rangle_\mathcal{E}^2-\langle\Phi,\zeta_{t_1\wedge\tau_\rho}\rangle_\mathcal{E}^2-\int_{t_1\wedge \tau_\rho}^{t_2\wedge\tau_\rho} \langle\Phi,G(u(s),p(s))\Phi\rangle_\mathcal{E}\,\dif s\Bigr)\Bigr]\\
&&\phantom{xx}=\ \EX^n\Bigl[\mathbb{I}_{A^n}\,\Bigl(\langle\Phi,Z^n_{t_2\wedge\tau_\rho}\rangle_\mathcal{E}^2-\langle\Phi,Z^n_{t_1\wedge\tau_\rho}\rangle_\mathcal{E}^2-\int_{t_1\wedge \tau_\rho}^{t_2\wedge\tau_\rho} \alpha_n\,\langle\Phi,G^n(Y^n_s,\theta^n_s)\Phi\rangle_\mathcal{E}\,\dif s \Bigr)\Bigr]\\
&&\phantom{xxxxx =\ }\mbox{}+ \EX^n\Bigl[\mathbb{I}_{A^n}\,\Bigl(\int_{t_1\wedge \tau_\rho}^{t_2\wedge\tau_\rho} \alpha_n\,\langle\Phi,G^n(Y^n_s,\theta^n_s)\Phi\rangle_\mathcal{E}- \langle\Phi,G^n(u(s),p(s))\Phi\rangle_\mathcal{E}\,\dif s\Bigr)\Bigr]\,.
\end{eqnarray*}
Here the first expectation in the final right hand side vanishes due to the It{\^o}-formula \eqref{clt_applied_oto_formula}: We apply the It{\^o}-formula for the function $g_\rho$ and the martingales $Z^n_{t\wedge\tau^n_\rho}$ to the terms $\langle\Phi,Z^n_{t_2\wedge\tau_\rho}\rangle_\mathcal{E}^2$ and $\langle\Phi,Z^n_{t_1\wedge\tau_\rho}\rangle_\mathcal{E}^2$. Then we find -- similarly to part (a) -- that the summands in the right hand side of the It{\^o}-formula vanish. Therefore we are left with only the martingale $M^{n,g_\rho}$ and the integral term, wherein $g_\rho''(\langle\phi,Z^n_{t-} \rangle_\mathcal{E})=2$ for all $t<\tau^n_\rho$. The martingale term vanishes due to the martingale property and the remaining integral is cancelled by the integral in the above expectation. Overall this shows that the first expectation vanishes.

Next we take the absolute value on both sides of the above equality and obtain, estimating the second expectation and extending the integration interval to $[0,T]$, the inequality
\begin{eqnarray*}
\lefteqn{\Big|\EX^n\Bigl[\mathbb{I}_A\,\Bigl(\langle\Phi,\zeta_{t_2\wedge\tau_\rho}\rangle_\mathcal{E}^2-\int_0^{t_2\wedge\tau_\rho}\! \langle\Phi,G(u(s),p(s))\Phi\rangle_\mathcal{E}\,\dif s-\langle\Phi,\zeta_{t_1\wedge\tau_\rho}\rangle_\mathcal{E}^2}\\
&&\phantom{xxxxxxxxxxxXxxxxxxxxxxxxxxxxxxxxxx}\mbox{}+ \int_0^{t_1\wedge\tau_\rho}\! \langle\Phi,G(u(s),p(s))\Phi\rangle_\mathcal{E}\,\dif s\Bigr)\Bigr]\Big|\\
&&\phantom{xxxxxxxxxxxxxxx}\leq\ \int_0^T \EX^n\Big|\alpha_n\,\langle\Phi,G^n(Y^n_s,\theta^n_s)\Phi\rangle_\mathcal{E}- \langle\Phi,G(u(s),p(s))\Phi\rangle_\mathcal{E}\Big|\,\dif s\,.
\end{eqnarray*}
The convergence of the upper bound to zero for $n\to\infty$ follows by assumption \eqref{clt_martingale_cond_2}. Hence we have proven an analogous result to \eqref{martingale_problem_proof_zero_exp} in part (a). The same line of argument that concluded part (a) also concludes part (b). The proof is completed.

%% file: section7_main_part.tex

\section{Application to models of excitable membranes}\label{section_application}

The primary motivation for the present work stems from the study of stochastic version of the Hodgkin-Huxley model \cite{HodgkinHuxley52} describing action potential generation and propagation in spatially extended neurons in a PDMP formulation. This model is analogous in structure to hybrid models that are used for the modelling of Calcium dynamics, cf.~\cite{Falcke,Koch,Swillens}, or models of cardiac tissue, cf.~ \cite{CardiacModel}. 
Therefore, we consider as an example of the application of the presented limit theorems 
%
%
%
a general \emph{compartmental-type hybrid stochastic model} for spatially extended excitable membranes introduced in \cite[Sec.~3.2]{RiedlerPhD} which subsumises the above mentioned applications. 
(Another example for the application of Theorem \ref{LLN} is the law of large numbers that is presented in \cite{Austin} for a particular one-dimensional hybrid model.) We refrain from discussing the physiological derivations of this type of model and the implications and interpretations of the limit theorems in this setting. These aspects will be subject to a forthcoming publication. 
%
%
%
%
%
%
%
%
%
%
%
%
We now fix some notation for the remainder of the section. The set $D\subset\rr^d$ denotes bounded spatial domain with the physically reasonable dimensions $d\leq 3$. 
That is, the set $D$ is a bounded interval when $d=1$ and when $d\in\{2,\,3\}$ we assume it possesses a $C^3$--boundary. Further, for a given dimension $d$, let $s$ denote the smallest integer such that $s>d/2$. 
Finally, let $m\in\mathbb{N}$ be the fixed number of states ion channels can be in.

\subsection{Deterministic limit system}

The deterministic limit is the solution to the membrane equation
\begin{equation}\label{limit_theorem_apl_det_cable_eq}
 \dot u=\sum_{i,j=1}^d a_{ij}(x)u_{x_ix_j}+\sum_{i=1}^m g_i(x)\, p_i\,(E_i-u)
\end{equation}
with $p_j,\,j=1\ldots,m$ given by solutions of the coupled equations
\begin{equation}\label{limit_theorem_det_gating_sys}
\dot p_j \,=\, F_j(p,u)\,:=\,\sum_{i\neq j} q_{ij}(u)\, p_i - q_{ij}(u)\, p_j\,.
\end{equation}
%
%
%
%
%
We choose Dirichlet boundary conditions for the component $u$, i.e., $u(t,x)=0$ for all $x\in\partial D$ and all $t\in[0,T]$, which, however, is of no particular importance for the considerations that follow and can be readily changed. 
Here the coefficient functions $a_{ij}$ and $g_i$ are smooth on $\nbar D$, with $g_i$ non-negative, and the differential operator is strongly elliptic. Further, the rate functions $q_{ij}$ are sufficiently smooth.\footnote{In detail the conditions are \cite[Sec.~3.3.1]{RiedlerPhD}: The functions $q_{ij}$ are bounded and bounded away from zero on the interval $[\nbar u_-,\nbar u_+]$. Further, on this interval they satisfy a Lipschitz and polynomial growth condition and are twice continuously differentiable with bounded derivatives.} 
Finally, the initial conditions satisfy $u_0\in H^1_0(D)\cap H^s(D)$ and $p_i(0)\in H^s(D)$ and, in addition, the pointwise bounds $u(0,x)\in [\nbar u_-,\nbar u_+]$ and $p_i(0,x)\in [0,1]$, $i=1,\ldots,m$, such that $\sum_{i=1}^m p_i(0,x)=1$, hold for all $x\in D$. 
Then, the deterministic system \eqref{limit_theorem_apl_det_cable_eq}, \eqref{limit_theorem_det_gating_sys} is well-posed, that is, there exists a unique global solution depending continuously on the initial condition, which also satisfies \eqref{det_limit_int_equation} \cite[Sec.~3.3.1]{RiedlerPhD}. In particular, the solution $(u(t),p(t))_{t\in[0,T]}$ is in $C([0,T], H^s(D))$ componentwise for every $T>0$ and pointwise bounded, i.e., $u(t,x)\in [\nbar u_-,\nbar u_+]$ and $p_i(t,x)\in[0,1]$ for all $(t,x)\in [0,T]\times \nbar D$ and all $i=1,\ldots,m$. 

\subsection{Compartmental-type membrane models}\label{subsection_compartmental_type_models}

We briefly summarise the essential features of PDMPs $(U^n_t,\Theta^n_t)_{t\geq 0}$, $n\in\mathbb{N}$, constituting compartmental-type membrane models. 

Firstly, an integral component of the sequence of models is a sequence of compartmentalisation of the spatial domain $D$. Thus, for each $n\in\mathbb{N}$ let $\mathcal{P}_n$ be a \emph{convex} partition of the domain $D$, i.e., $\mathcal{P}_n$ is a finite collection of mutually disjoint convex\footnote{The convexity of the compartments is a technical assumption which allows to employ Poincar\`e's inequality in the proof of the limit theorems with a known optimal Poincar\'e constant \cite{Acosta,Payne}.} subsets of $D$, called \emph{compartments}, such that their union equals $D$. 

The second fundamental aspect is the channel distribution across the compartments yielding the stochastic jump dynamics and the coordinate functions $z^n$. We assume that each compartment either does not contain channels or a fixed deterministic number. Let $p(n)$ denote the number of non-empty compartments of the $n$th model denoted by $D_{1,n},\ldots, D_{p(n),n}$ and $l(k,n)$ be for $k\leq p(n)$ the total number of channels in the $k$th non-empty compartment of the $n$th model. 
%
%
%
Then the piecewise constant components of the PDMPs are given by $mp(n)$-dimensional vectors $\Theta^n_t=(\Theta_i^{k,n}(t))_{i=1,\ldots,m,\,k=1,\ldots,p(n)}$ with finite state spaces $K_n$. Each component $\Theta_i^{k,n}(t)$ counts the number of channels located in the domain $D_{k,n}$ which are in state $i$ \mbox{at time $t$.} and it holds that
\begin{equation*}
\sum_{i=1}^m\Theta^{k,n}_i(t) = l(k,n)\,.
\end{equation*}
as channels can neither be destroyed nor created. %
We proceed to define the stochastic jump dynamics. As two channel switching do not occur simultaneously, the only jumps in the configuration $\theta^n\in K_n$ with non-zero probability are transitions concerning one single channel. That is, events for which in one particular compartment one particular channel changes its state. Let $q_{ij}:\rr\to\rr_+$ denote the $u$-dependent instantaneous rate of one channel switching from state $i$ to $j$. Then given a specific configuration $\theta^n\in K_n$ the rate that one channel in compartment $D_{k,n}$ switches from state $i$ to $j$ is given by
\begin{equation}\label{limits_compartments_one_subdomain_rate}
\theta^{k,n}_i\,Q_{ij}^{k,n}(u)\in\rr_+\,,
\end{equation}
where $Q^n_{ij}(u)$ is a functional of the membrane variable $u\in L^2(D)$ defined as
\begin{equation*}
Q_{ij}^{k,n}(u):=q_{ij}\Bigl(\frac{1}{|D_{k,n}|}\int_{D_{k,n}} u(x)\,\dif x\Bigr)\,.
\end{equation*}
That is, $Q^{k,n}_{ij}(u)$ is the instantaneous rate $q_{ij}$ evaluated at the average value of the membrane variable over the compartment $D_{k,n}$. Hence the rate \eqref{limits_compartments_one_subdomain_rate} is the number of channels in state $i$ in domain $D_{k,n}$ times the rate of one channel switching from $i$ to $j$. This definition yields by summing over all events the total instantaneous rate
\begin{equation}\label{arnold_approach_total_rate}
\Lambda^n(u,\theta^n)\,:=\,\sum_{i,j=1}^{m} \sum_{k=1}^{p(n)}\theta^{k,n}_i\,Q^{k,n}_{ij}(u)\,.
\end{equation}
Note that for each $n$ the total instantaneous rate is bounded and as expected proportional to the total number of channels which implies that the PDMPs are regular. 
Finally, we define on the set $K_n$ for $i=1,\ldots,m$ the coordinate functions
\begin{equation}\label{Arnold_coord_fct}
z^n_i(\theta^n):=\sum_{k=1}^{p(n)} \frac{\theta_i^{k,n}}{l(k,n)}\,\mathbb{I}_{\,D_{k,n}}\,\in\, L^2(D)\,.
\end{equation}
The coordinate process $z^n(\Theta^n_t)$ is c\`adl\`ag with each component taking values in $L^2(D)$. Clearly, the coordinate process is zero on those compartments which do not contain channels. Moreover, each $z_i^n(\Theta^n_t)$ is for every $t\geq 0$ a piecewise constant function on the spatial domain $D$ which takes values in $[0,1]$. 

Thirdly, the family of abstract evolution equations \eqref{abstract_ODE} defining the dynamics of the PDMP's continuous component $U^n$ are given by the parabolic, linear, inhomogeneous second order partial differential equations
\begin{equation}\label{limit_theorem_apl_stoch_cable}
 \dot u=\sum_{i,j=1}^d a_{ij}(x)u_{x_ix_j} +\sum_{i=1}^m g_i(x)\, z^n_i(\theta^n)\,(E_i-u)\,.
\end{equation}
Consistently with the deterministic limit system we equip equation \eqref{limit_theorem_apl_stoch_cable} with Dirichlet boundary conditions. Finally, we define the operators $A,\,B$ depending on $\theta^n$ only via suitable coordinate functions, cf.~\eqref{uniform_operators}, by
\begin{equation}\label{limit_theorem_apl_pdmp_op}
A(z^n(\theta^n))\,u:=\displaystyle\sum_{i,j=1}^d a_{ij}(x)u_{x_ix_j}\,,\qquad 
B(z^n(\theta^n),u):= \displaystyle\sum_{i=1}^m g_i(x)\,z^n_i(\theta^n)\,(E_i-u)\,.
\end{equation}

To conclude, it is easy to see that the characteristics defined via the individual rates \eqref{limits_compartments_one_subdomain_rate}, the total jump rate \eqref{arnold_approach_total_rate} and the evolution equation \eqref{limit_theorem_apl_pdmp_op} define a sequence of $L^2(D)\times K_n$--valued infinite-dimensional PDMPs $(U^n_t,\Theta^n_t)_{t\geq 0}$. Moreover, the membrane  component $(U^n_t)_{t\geq 0}$ is almost everywhere pointwise bounded, i.e., $U^n_t(x)\in [\nbar u_-,\nbar u_+]$ for almost all $x\in D$ and all $t\geq 0$, where $\nbar u_-:=\min_i E_i\leq 0$ and $\nbar u_+:=\max_i E_i\geq 0$, for initial conditions $U^n_0$ satisfying these bounds, cf.~\cite[Sec.~3.2]{RiedlerPhD}, which we always assume.




\subsection{Limit theorems for compartmental-type models}

Applying the limit theorems derived in Sections \ref{Sec_limitThm} and \ref{section_the_CLTs} to compartmental models we find that the conditions therein translate into assumptions on the behaviour of the sequence of partitions $\pindex_n$ and the number of ion channels in the membrane, see Appendix \ref{section_9_proof_compart_model}.
%
%
Thus, we denote by $\delta(n)$ the maximal diameter of the non-empty compartments in the $n$th model, i.e.,
\begin{equation*}
\delta_+(n):=\max_{k=1,\ldots,p(n)} \textnormal{diam}(D_{k,n})\,,
\end{equation*}
and by $\ell_+(n)$ and $\ell_-$ the maximal and minimal number of channels in non-empty compartments, i.e.,
\begin{equation*}
\ell_+(n):=\max_{k=1,\ldots,p(n)} l(k,n),\qquad \ell_-(n):=\min_{k=1,\ldots,p(n)} l(k,n)\,.
\end{equation*}

Then the law of large numbers takes the following form.

\begin{thm}\label{LLN_compartmental_models} Assume that the sequence of partitions satisfies that
\begin{equation}\label{LLN_compartmental_model_part_conds}
\lim_{n\to\infty}\delta_+(n)=0,\qquad\quad \lim_{n\to\infty} \ell_-(n)=\infty, 
\end{equation}
and that the initial conditions $(U^n_0,z^n(\Theta^n_0))$ converge in probability to $(u_0,p_0)$ in the space $L^2(D)^{m+1}$. Then the compartmental-type models converge in probability to the deterministic solution of the excitable media system \eqref{limit_theorem_apl_det_cable_eq}, \eqref{limit_theorem_det_gating_sys} in the sense that it holds for all $\eps>0$ that
\begin{equation}\label{example_LLN_conv_in_prob}
\lim_{n\to\infty}\Pr\Bigl[\sup\nolimits_{t\in[0,T]}\|U^n_t-u(t)\|_{L^2}+\sum_{i=1}^m\sup\nolimits_{t\in[0,T]}\|z^n_i(\Theta^n_t)-p(t)\|_{L^2}>\eps\,\Bigr]\,=\, 0\,.
\end{equation}
Moreover, the convergence also holds in the mean in the space $L^2((0,T),L^2(D))$, i.e.,
\begin{equation}\label{example_LLN_conv_in_mean}
\lim_{n\to\infty}\EX^n\Bigl[\|U^n_t-u(t)\|_{L^2((0,T),L^2)}+\sum_{i=1}^m\|z^n_i(\Theta^n_t)-p(t)\|_{L^2((0,T),L^2)}\Bigr]\,=\,0\,. 
\end{equation}
\end{thm}




Next we present the appropriate quadratic variation process for the martingale central limit theorem. For the definition of the limiting diffusion we consider for $u,p_i\in C(\nbar D)$ the bilinear form
\begin{equation}\label{def_example_cov_operator}
\hspace{-15pt}\left.\begin{array}{rcl}
(\Psi,\Phi)\mapsto\bigl( G(u,p)\,\Psi,\Phi\bigr)_{L^2}&=&
\displaystyle\sum_{j=1}^m\,\sum_{i\neq j}\int_D p_i(\vz)\,q_{ij}(u(\vz))\, \psi_j(\vz)\,\phi_j(\vz)\,\dif\vz\\
&&\mbox{}+\displaystyle\sum_{j=1}^m\,\sum_{i\neq j}\int_D p_j(\vz)\, q_{ji}(u(\vz))\,\psi_j(\vz)\,\phi_j(\vz)\,\dif\vz\\
&&\mbox{}-\displaystyle\sum_{j=1}^m\,\sum_{i\neq j}\int_D p_j(\vz)\, q_{ji}(u(\vz))\,\psi_i(\vz)\,\phi_j(\vz)\,\dif \vz\\
&&\mbox{}-\displaystyle\sum_{j=1}^m\,\sum_{i\neq j}\int_D p_i(\vz)\, q_{ij}(u(\vz))\,\psi_i(\vz)\,\phi_j(\vz)\,\dif \vz\,.
\end{array}\right.
\end{equation}
Note that the right hand side is finite for all $\phi_i,\psi_i\in L^2(D)$ as $p_i$ and $q_{ij}(u)$ are bounded functions. Hence, for every given $\Psi\in L^2(D)^m$ the mapping $\Phi\mapsto (G(u,p)\Psi,\Phi)$ is a linear, bounded functional on $L^2(D)^m$ and, conversely, for every given $\Phi\in L^2(D)^m$ the mapping $\Psi\mapsto (G(u,p)\Psi,\Phi)$ is a linear, bounded functional on $L^2(D)^m$.\medskip

\begin{prop}\label{compartmental_example_cov_operators} The operator $G(u,p)$ defined via \eqref{def_example_cov_operator} is for $s>d/2$ a trace class operator mapping $H^s(D)$ into its dual $H^{-s}(D)$. Moreover, the operator-valued map $t\mapsto G(u(t),p(t))$ defines a unique centred diffusion process on $H^{-s}(D)$.
\end{prop}

\begin{proof} As stated in \cite{Kotelenez1} it is sufficient for the statement of the proposition that the operator $G(u(t),p(t))$ is self-adjoint, positive and of trace class. These properties are easily verified and for a detailed proof we refer to \cite{RiedlerPhD}.\end{proof}

In order to state the conditions on the partitions in the central limit theorem we define $\nu_+(n)$ and $\nu_-(n)$ to be the maximum and minimum Lebesgue measure of non-empty compartments, i.e.,
\begin{equation*}
\nu_+(n):=\max_{k=1,\ldots,p(n)}\, |D_{k,n}|\,,\qquad \nu_-(n):=\min_{k=1,\ldots,p(n)}\, |D_{k,n}|\,.
\end{equation*}
Finally, note that in the following the coordinate functions $z^n$ are considered as maps from $K_n$ into the space $H^{-2s}(D)$.

\begin{thm}\label{CLT_compartmental_model} Let $s$ be the smallest integer such that $s>d/2$. If in addition to \eqref{LLN_compartmental_model_part_conds} and the convergence of the initial conditions the sequence of partitions satisfies
\begin{equation}\label{example_CLT_add_condition}
\lim_{n\to\infty}\frac{\ell_-(n)\,\nu_-(n)}{\ell_+(n)\,\nu_+(n)} = 1\,,
\end{equation}
then the sequence of $H^{-2s}(D)$--valued martingales $\Bigl(\sqrt{\frac{\ell_-(n)}{\nu_+(n)}}\,M^n_t\Bigr)_{t\geq 0}$ converges weakly to the $(H(D)^{-2s})^m$--valued diffusion defined by \eqref{def_example_cov_operator}.
\end{thm}

\begin{rem} We note that for all reasonable physical domains $D$ and all initial conditions $(u_0,p_0)$ sequences of partitions $\pindex_n$ and initial conditions $(U^n_0,\Theta^n_0)$ for the PDMPs can be found satisfying the conditions of Theorems \ref{LLN_compartmental_models} and \ref{CLT_compartmental_model}. For example, a suitable sequence of partitions is obtained by grids of uniform cubes with decreasing edge length covering the domain $D$ and putting channels only into these cubes which are fully contained in $D$. For a more detailed discussion of these aspects we refer to the PhD thesis of one of the present authors \cite{RiedlerPhD}.
\end{rem}

%% file: paper_conclusion.tex
\section{Conclusions}\label{section_conclusions}

As a general theoretical results for PDMPs we have derived a law of large numbers and martingale central limit theorem in Sections \ref{Sec_limitThm} and \ref {section_the_CLTs} of this study. The former establishes a connection of stochastic hybrid models to deterministic models given, e.g., by systems of partial differential equations. Whereas the latter connects the stochastic fluctuations in the hybrid models to diffusion processes. As a prerequisite to these limit theorems we carried out a thorough discussion of Hilbert space valued martingales associated to the PDMPs. Furthermore, these limit theorems provide the basis for a general Langevin approximation to PDMPs, i.e., certain stochastic partial differential equations that are expected to be similar in their dynamics to PDMPs. We have applied these results to compartmental-type models of spatially extended excitable membranes. Ultimately this yields a system of SPDEs which models the internal noise of a biological excitable membrane based on a theoretical derivation from exact stochastic hybrid models.

Topics for further research are motivated by corresponding results in finite-dimensions \cite{Kurtz2,Wainrib1} and for spatially inhomogeneous chemical reaction systems converging to reaction diffusion equations, cf.~\cite{Kotelenez1}. In these studies limit theorems are derived for the fluctuations around the deterministic limit identified by the law of large numbers. 
%
Using the notation of Section \ref{section_the_CLTs} we conjecture that the sequence of processes, $\bigl(\sqrt{\alpha_n}\,(U^n_t-u(t),z^n(\Theta^n_t)-p(t)\bigr)_{t\geq 0}$, $n\in\mathbb{N}$, converges in distribution to a suitable diffusion process. Moreover, we further conjecture that this limit is closely related to the asymptotic linearisation of the Langevin approximation around the solution of the deterministic limit, cf.~\cite{Wainrib1} wherein this result is proven for finite-dimensional PDMPs.

Further, on the applications side we believe that the Langevin approximation to spatio-temporal PDMP models of excitable membranes poses an important object for further investigation. Its derivation was the initial motivation of the study of the limit theorems in the present study and it is their main application herein which enables to write down the system of SPDEs that constitute a Langevin approximation. This system now demands for further analysis, particularly, first of all the question of existence and uniqueness of the Langevin approximation has to be addressed. Subsequently, as SPDEs are analytically more accessible than PDMPs a theoretical analysis of qualitative and quantitative properties of the models may be possible. 

Finally, we want to mention that the limit theorems presented also find applications beyond excitable membrane models. In current work in progress by one of the present authors the limit theorems derived in Sections \ref{Sec_limitThm} and \ref{section_the_CLTs} are applied to stochastic neural field equations, based on a model presented in \cite{Bressloff1}, cf.~a preliminary account in \cite{RiedlerMFO}. We also plan to investigate the connection to similar limits derived for reaction-diffusion models, cf.~the series of results on variations of the model in \cite{Kotelenez1,Kotelenez2,Kotelenez3,Kotelenez4} and \cite{Blount1,Blount2,Blount3,Blount4,Blount5}. An answer to this question would contribute to a more complete picture of limit-theorems for spatio-temporal stochastic models.

%% file: proof_of_square_integrable.tex
\section{Proof of Theorem \ref{banach_martingale_is_square_int} (It{\^o}-isometry)}\label{section_ito_iso_banach_proof}

In this proof we show that under condition \eqref{LLN_martingale_bound_first_lemma} the processes $M^n_j$, $j=1,\ldots,m$, $n\in\mathbb{N}$, defined in \eqref{local_martingale_2} are square-integrable, c\`adl\`ag martingales which satisfy the It{\^o}-isometry \eqref{banach_ito_isometry}.
Throughout the proof we fix a $j=1,\ldots,m$ and $n\in\mathbb{N}$ and the results holds for any such $j$ and $n$. Therefore, speaking of a PDMP in the following always refers to the PDMP $(U^n_t,\Theta^n_t)_{t\geq 0}$ corresponding to the fixed $n$. Further, for notational simplicity we omit the indices $n$ and $j$ discriminating processes and characteristics of PDMPs, i.e., $M^n_j$ and $z^n_j(\Theta^n)$ are denoted simply by $M$ and $z(\Theta)$. Finally, recall that $\tau_k$, $k=1,2,\ldots$, denotes the sequence of increasing random jump times of the PDMP which are stopping times satisfying $\lim_{k\to\infty} \tau_k=\infty$ almost surely.\medskip

First of all, note that the process $M$ is c\`adl\`ag by definition.
The proof of the remaining open results is split into three parts. In the first, part (a), we prove the martingale property for the real process $(\langle\phi,M(t)\rangle_E)_{t\geq0}$ for every $\phi\in E^\ast$. Then, the first main statement of Theorem \ref{banach_martingale_is_square_int}, the square-integrability of the process $M(t)$, is proved in part (b). Moreover, as square-integrability implies integrability, the Hilbert space martingale property follows. Finally, the second main statement, the It{\^o}-Isometry \eqref{banach_ito_isometry}, is established in part (c). The proof we present in part (b) is motivated by the proof of \cite[Prop.~4.5.3]{Jacobsen} which states the corresponding results for real-valued martingales associated with PDMPs. In extending to the present setup the method of proof employed therein one has to ensure, on the one hand, that the employed results and estimation procedures all have corresponding analoga in the infinite-dimensional setting. On the other hand, one has to carefully make sure that only the weaker regularity results available in infinite-dimensions are used. Finally, the introduction of random initial conditions, not considered in \cite{Jacobsen}, also necessitates some adaptations.

(a)\quad First note that for all $\phi\in E^\ast$ the real-valued processes $\langle\phi,M(t)\rangle_E$ satisfy
\begin{eqnarray}\label{second_gen_mart}
\langle\phi,M(t)\rangle_E&\!=\!&\langle \phi, z(\Theta_t)\rangle_E-\langle \phi,z(\Theta_0)\rangle_E\\[1ex]
&&\mbox{}-\int_0^t\Lambda(U_{s-},\Theta_{s-})\int_{K}\langle\phi,z(\xi) \rangle_E-\langle\phi,z(\Theta_{s-})\rangle_E\,\mu\bigl((U_{s-},\Theta_{s-}),\dif \xi\bigr)\,\dif s.\nonumber
\end{eqnarray}
Equation \eqref{second_gen_mart} is obtained from \eqref{local_martingale} due to the regularity of the PDMP as the set of jump times in $[0,t]$ is almost surely finite for all $t\geq 0$. Therefore the integrands in the right hand sides of \eqref{local_martingale} and \eqref{second_gen_mart} differ only on a set of Lebesgue measure zero almost surely. Moreover, the integrand in the right hand side of \eqref{second_gen_mart} has the form of the extended generator, cf.~Theorem \ref{PDMP_gen_theorem}, applied to the map 
\begin {equation}\label{mapping_to_be_in_ext_gen}
(u,\xi)\mapsto\langle\phi,z(\xi))\rangle_E\,,
\end{equation}
which is independent of $u$. It follows that the process $\langle\phi,M(t)\rangle_E$ is a local martingale if the map \eqref{mapping_to_be_in_ext_gen} is in the domain of the extended generator, cf.~Theorem \ref{PDMP_gen_theorem}. Obviously, path-differentiability almost everywhere is trivially satisfied as the map $t\mapsto \langle\phi,z(\Theta_t)\rangle_E$ is piecewise constant. Hence,  it remains to consider the integrability condition for which it is a sufficient that
\begin{equation}\label{condition_to_be_in_L1}
\EX\int_0^t\Lambda(U_{s-},\Theta_{s-})\int_K \big|\langle\phi,z(\xi)-z(\Theta_{s-})\rangle_E\big|\,\mu\bigl((U_{s-},\Theta_{s-}),\dif\xi)\,\dif s<\infty\quad\forall\,t\geq 0\,,
\end{equation}
cf.~\cite{BuckwarRiedler,Davis2,Jacobsen}. Using Young's inequality we obtain an upper bound to \eqref{condition_to_be_in_L1} by
\begin{equation*}
\frac{1}{2}\EX\!\!\int_0^t\!\!\Lambda(U_{s-},\Theta_{s-})\,\dif s + \frac{1}{2}\EX\!\!\int_0^t\!\!\Lambda(U_{s-},\Theta_{s-})\!\int_K\!\!\big|\langle\phi,z(\xi)-z(\Theta_{s-})\rangle_E\big|^2\,\mu\bigl((U_{s-},\Theta_{s-}),\dif\xi)\,\dif s.
\end{equation*}
Here the first expectation is finite due to the PDMP being regular and the second is finite by an immediate consequence of assumption \eqref{LLN_martingale_bound_first_lemma}.

Next, we show that the process is not only a local martingale but even a martingale. As mentioned above the process $\langle\phi,M(t)\rangle_E$ satisfies
\begin{equation*}
\langle\phi,M(t)\rangle_E=\int_0^t\int_K \langle\phi,z(\xi)-z(\Theta_{s-})\rangle_E\,\ntilde M(\dif s,\dif\xi)
\end{equation*}
where $\ntilde M:=N-\nhat N$ is the random martingale measure associated with the PDMP with counting measure $N$ and compensator $\nhat N(\dif\xi,\dif s)=\Lambda(U_{s-},\Theta_{s-})\, \mu\bigl((U_{s-},\Theta_{s-}),\dif\xi)\,\dif s$. The validity of this formula follows as \eqref{mapping_to_be_in_ext_gen} is in the extended generator. 
%
%
Thus the process $\langle\phi,M(t)\rangle_E$ has the form of a stochastic integral with respect to the martingale measure associated with the PDMP. Furthermore, due to \cite[Thm.~4.6.1]{Jacobsen} it holds that the process is a martingale if \eqref{condition_to_be_in_L1} is finite for all $t\geq 0$. But we have already shown that this holds due to the regularity of the PDMP and assumption \eqref{LLN_martingale_bound_first_lemma}.



 
(b)\quad We now prove the square-integrability of the process $M$. In a first step we prove in (b.1) that $M$ stopped at the first jump $\tau_i$ is square-integrable. Subsequently in part (b.2) this result is extended to $M$ stopped at any jump time $\tau_k$, $k\in\mathbb{N}$. Then we are able to infer square-integrability of the process $M$. As square-integrability implies integrability it follows from part (a) that $M$ is a Hilbert space valued martingale.

(b.1)\quad Note that prior to $\tau_1$ the jump component $\Theta$ of the PDMP remains constant. We introduce the notation
\begin{equation*}
\ntilde N(s):=\int_0^s\Lambda(U_r,\Theta_0)\int_{K}z(\xi)-z(\theta_0)\,\mu\bigl((U_r,\Theta_0),\dif\xi)\,\dif r
\end{equation*}
which implies that $s\mapsto\|\ntilde N(s)\|_E^2$ is almost surely absolutely continuous with derivative
\begin{eqnarray}\label{square_mart_proof_derivative_form}
\frac{\dif}{\dif s}\|\ntilde N(s)\|_E^2&=&2\bigl(\tfrac{\dif}{\dif t}\ntilde N(s),\ntilde N(s)\bigr)_E\nonumber\\
&=&2\,\Lambda(U_s,\Theta_0)\int_{K}\bigl(z(\xi)-z(\Theta_0),\ntilde N(s)\bigr)_E\,\mu\big((U_s,\Theta_s),\dif\xi\bigr)\,.
\end{eqnarray}
Due to the structure of a PDMP we obtain for the conditional expectation with respect to the initial condition
\begin{eqnarray*}
\lefteqn{\EX\bigl[\|M(\tau_1\wedge t)\|_E^2\,|\,\sF_0\bigr]\ =\ \|\ntilde N(t)\|_E^2\,\exp\Bigl(-\int_0^t\Lambda(U_r,\Theta_0)\,\dif r\Bigr)}\\
&&\hspace{-25pt}\mbox{}+\!\int_0^t\!\Bigl[\int_K\!\! \big\|z(\vartheta)\!-\!z(\Theta_0)\!-\!\ntilde N(s)\big\|_E^2\,\mu\bigl((U_s,\Theta_0),\dif\vartheta)\bigr)\Bigr]\,\Lambda(U_s,\Theta_0)\exp\Bigl(-\!\int_0^s\!\Lambda(U_r,\Theta_0)\,\dif r\Bigr)\,\dif s.
\end{eqnarray*}
That is, the first term in the right hand side is the position of the stopped process $\|M(\tau_1\wedge t)\|_E^2$ at time $t$ if $t<\tau_1$ times the conditional probability that the first jump does not occur before $t$. The second term is its position after the jump integrated over the conditional density that a jump occurs in $[0,t]$. We apply integration by parts to the first term (note that $\ntilde N(0)=0$) and find that
\begin{eqnarray*}
\|\ntilde N(t)\|_E^2\,\exp\Bigl(-\int_0^t\Lambda(U_s,\Theta_0)\,\dif s\Bigr)
&=&\int_0^t\Bigl[ 2(\tfrac{\dif}{\dif t}\ntilde N(s),\ntilde N(s))_E\,\exp\Bigl(-\int_0^s\Lambda(U_r,\Theta_0)\,\dif r\Bigr)\\
&&\mbox{}-\|\ntilde N(s)\|_E^2\,\Lambda(U_s,\Theta_0) \,\exp\Bigl(-\int_0^s\Lambda(U_r,\Theta_0)\,\dif r\Bigr)\Bigr]\,\dif s\,.
\end{eqnarray*}
Therefore we obtain
\begin{eqnarray*}
\lefteqn{\EX\bigl[\|M(\tau_1\wedge t)\|_E^2\,|\,\sF_0\bigr]\ =\  \int_0^t\Bigl[ 2(\tfrac{\dif}{\dif t}\ntilde N(s),\ntilde N(s))_E\,\exp\Bigl(-\int_0^s\Lambda(U_r,\Theta_0)\,\dif r\Bigr)\Bigr]\,\dif s}\\[1ex]
&&\phantom{xxxxxxxxx}\mbox{}+\int_0^t\Bigl[\Bigl(\int_K \big\|z(\vartheta)-z(\Theta_0)-\ntilde N(s)\big\|_E^2-\big\|\ntilde N(s)\big\|_E^2\,\mu\bigl((U_s,\Theta_0),\dif\vartheta)\bigr)\Bigr)\\
&&\phantom{xxxxxxxxxxxxxxxxxxxxxxxxxxxxxxx}\Lambda(U_s,\Theta_0) \,\exp\Bigl(-\int_0^s\Lambda(U_r,\Theta_0)\,\dif r\Bigr)\Bigr]\,\dif s\,.
\end{eqnarray*}
Note that $\|z(\vartheta)-z(\Theta_0)-\ntilde N(s)\|_E^2=\|z(\vartheta)-z(\Theta_0)\|_E^2+\|\ntilde N(s)\|_E^2-2(z(\vartheta)-z(\Theta_0),\ntilde N(s))_E$ and thus
\begin{eqnarray*}
\lefteqn{\EX\bigl[\|M(\tau_1\wedge t)\|_E^2\,|\,\sF_0\bigr]\ =}\\[1ex]
&&\hspace{-20pt} \int_0^t \Bigl[2(\tfrac{\dif}{\dif t}\ntilde N(s),\ntilde N(s))_E\,\exp\Bigl(-\int_0^s\Lambda(U_r,\Theta_0)\,\dif r\Bigr)\Bigr]\,\dif s\,\\
&&\hspace{-23pt} \mbox{} -2\! \int_0^t\!\Bigl[\Bigl(\int_K\!\bigr(z(\vartheta)\!-\!z(\Theta_0),\ntilde N(s)\bigr)_E \,\mu\bigl((U_s,\Theta_0),\dif\vartheta)\bigr)\!\Bigr)\Lambda(U_s,\theta_0)\exp\Bigl(\!-\!\!\int_0^s\!\!\Lambda(U_r,\Theta_0)\,\dif r\Bigr)\Bigr]\dif s\\
&&\hspace{-23pt} \mbox{} + \int_0^t\Bigl[\Bigl(\int_K \|z(\vartheta)-z(\Theta_0)\|_E^2\,\mu\bigl((U_s,\Theta_0),\dif\vartheta)\bigr)\Bigr)\,\Lambda(U_s,\Theta_0) \,\exp\Bigl(-\!\int_0^s\!\Lambda(U_r,\Theta_0)\,\dif r\Bigr)\Bigr]\dif s.
\end{eqnarray*}
Due to form of the derivative \eqref{square_mart_proof_derivative_form} the first two terms cancel and we are left with the equality
\begin{eqnarray}\label{banach_proof_square_integ_res_1}
\lefteqn{\EX\bigl[\|M(\tau_1\wedge t)\|_E^2\,|\,\sF_0\bigr]}\\[1ex]
&&\hspace{-5pt}=\ \int_0^t\Lambda(U_s,\Theta_0)\int_K \|z(\vartheta)-z(\Theta_0)\|_E^2\,\mu\bigl((U_s,\Theta_0),\dif\vartheta)\bigr)\, \,\exp\Bigl(-\int_0^s\Lambda(U_r,\Theta_0)\,\dif r\Bigr)\dif s\,.\nonumber
\end{eqnarray}
Next we calculate the expectation of the real-valued process
\begin{equation*}
\ntilde N_2(s):=\int_0^s\Lambda(U_r,\Theta_0)\int_{K}\|z(\vartheta)-z(\Theta_0)\|_E^2\,\mu\bigl((U_r,\Theta_0),\dif\vartheta\bigr)\,\dif r 
\end{equation*}
stopped at $\tau_1$. The process $\ntilde N_2$ is connected to the process $\ntilde N$ defined at the beginning of part (b.1) inasmuch as the integrand of the former is the squared norm of the latter. Furthermore note that $\ntilde N_2$ is the term inside the expectation in the right hand side of the It{\^o}-isometry \eqref{banach_ito_isometry}. Thus the aim is now to show that the conditional expectation of $\ntilde N_2(t\wedge\tau_1)$ equals the conditional expectation of $\|M(t\wedge\tau_1)\|_E^2$. Again due to the particular structure of the PDMP we obtain for the conditional expectation
\begin{eqnarray*}
\lefteqn{\EX\bigl[\ntilde N_2(\tau_1\wedge t)\,|\,\sF_0\bigr]}\\[1ex]
&&\hspace{-15pt}= \ntilde N_2(s)\,\exp\Bigl(-\!\int_0^t\!\Lambda(U_r,\Theta_0)\,\dif r\Bigr) +\int_0^t\Bigl[ \ntilde N_2(s)\,\Lambda(U_s,\Theta_0)\,\exp\Bigl(-\!\int_0^s\!\Lambda(U_r,\Theta_0)\,\dif r\Bigr)\Bigr]\,\dif s\,. 
\end{eqnarray*}
Integration by parts applied to the integral term yields
\begin{eqnarray*}
\lefteqn{\int_0^t\Bigl[ \ntilde N_2(s)\,\Lambda(U_s,\Theta_0)\,\exp\Bigl(-\!\int_0^s\Lambda(U_r,\Theta_0)\,\dif r\Bigr)\Bigr]\,\dif s\ = \ -\ntilde N_2(t)\,\exp\Bigl(-\!\int_0^t\Lambda(U_r,\Theta_0)\,\dif r\Bigr)}\\
&&\mbox{}\hspace{-10pt}+\int_0^t\Bigl[ \Lambda(U_s,\Theta_0)\int_{K}\|z(\vartheta)-z(\Theta_0)\|_E^2\,\mu\bigl((U_s,\Theta_0),\dif\vartheta\bigr) \,\exp\Bigl(-\int_0^s\Lambda(U_r,\Theta_0)\,\dif r\Bigr)\Bigr]\,\dif s.
\end{eqnarray*}
Therefore we obtain that
\begin{eqnarray}\label{banach_proof_square_integ_res_2}
\lefteqn{\EX\bigl[\ntilde N_2(\tau_1\wedge t)\,|\,\sF_0\bigr] }\\[1ex]
&&=\ \int_0^t \Lambda(U_s,\Theta_0)\int_{K}\|z(\vartheta)-z(\Theta_0)\|_E^2\,\mu\bigl((U_s,\Theta_0),\dif\vartheta\bigr) \,\exp\Bigl(-\int_0^s\Lambda(U_r,\Theta_0)\,\dif r\Bigr)\,\dif s\,.\nonumber
\end{eqnarray}
A comparison of the right hand sides in equalities \eqref{banach_proof_square_integ_res_1} and \eqref{banach_proof_square_integ_res_2} shows that they are equal and thus we obtain after taking the expectation of both conditional expectations that
\begin{equation}\label{proof_first_stopped_ito_isometry}
\EX \|M(\tau_1\wedge t)\|_E^2 \, =\, \EX\ntilde N_2(\tau_1\wedge t)\,. 
\end{equation}
As $\ntilde N_2$ is increasing and thus $\ntilde N_2(\tau_1\wedge t)\leq \ntilde N_2(t)$ almost surely, we obtain that the right hand side in this equation is finite due to condition \eqref{LLN_martingale_bound_first_lemma}. Note that \eqref{proof_first_stopped_ito_isometry} is the It{\^o}-isometry \eqref{banach_ito_isometry} for the stopped process $M(t\wedge\tau_1)$.

(b.2)\quad In this part of the proof we show the square-integrability for the process $M$ stopped at an arbitrary jump time $\tau_k$, $k\in\mathbb{N}$, and finally for the non-stopped process $M$. To this end we first note that Analogously to part (b.1) we find that
\begin{equation*}
\EX\Bigl[\big\|M(\tau_{k+1}\wedge t)-M(\tau_k\wedge t)\big\|_E^2\,\big|\,\sF_{\tau_k}\Bigr]\, =\, \EX\Bigr[\ntilde N_2(\tau_{k+1}\wedge t)-\ntilde N_2(\tau_k\wedge t)\,\big|\,\sF_{\tau_k}\Bigr]\,. 
\end{equation*}
Thus taking expectations on both sides of this equality yields
\begin{equation}\label{square_int_middle_res}
\EX\,\|M(\tau_{k+1}\wedge t)-M(\tau_k\wedge t)\big\|_E^2\, =\, \EX\,\ntilde N_2(\tau_{k+1}\wedge t)-\EX\,\ntilde N_2(\tau_k\wedge t)\, <\infty\,,
\end{equation}
where the right hand side is finite as due to\eqref{LLN_martingale_bound_first_lemma} both expectations are finite.

By induction we next show that each $M(\tau_k\wedge t)$ is square-integrable. Assume that $\EX\|M(\tau_k\wedge t)\|_E^2<\infty$, where the induction basis for $k=1$ holds due to part (b.1). Then the reverse triangle inequality yields that
\begin{eqnarray*}
\lefteqn{\EX\|M(\tau_{k+1}\wedge t)\|_E^2+\EX\|M(\tau_k\wedge t)\|_E^2-2\EX\bigl(\|M(\tau_{k+1}\wedge t)\|_E\,\|M(\tau_{k}\wedge t)\|_E\bigr)}\\[1ex]
&&\phantom{xxxxxxxxxxxxxxxxxxxxxxxxxxxxxxxxxxx}\leq\ \EX \|M(\tau_{k+1}\wedge t)-M(\tau_k\wedge t)\|_E^2\,.
\end{eqnarray*}
Here the right hand side is finite due to \eqref{square_int_middle_res} and an application of Young's inequality to the product in the left hand side yields that for all $\eps>0$
\begin{equation*}
(1-2\eps)\,\EX\|M(\tau_{k+1}\wedge t)\|_E^2+\bigr(1-\tfrac{1}{2\eps}\bigl)\,\EX\|M(\tau_k\wedge t)\|_E^2\,<\,\infty
\end{equation*}
Assume that $\EX\|M(\tau_{k+1}\wedge t)\|_E^2=\infty$. Then choosing $\eps<1/2$ we obtain a contradiction due to the induction hypotheses.

In a final step of this part of the proof we show square-integrability for the non-stopped process. Using Fatou's Lemma and monotone convergence for interchanging limits and expectation we obtain the following upper estimate
\begin{eqnarray}\label{ito_iso_proof_upper_est}
\EX\|M(t)\|_E^2 & =& \EX\liminf_{k\to\infty} \|M(\tau_k\wedge t)\|_E^2\nonumber\\[1ex]
&\leq & \liminf_{k\to\infty} \EX\|M(\tau_k\wedge t)\|_E^2
\ =\ \lim_{k\to\infty} \EX \ntilde N_2(\tau_k\wedge t)
\ =\ \EX\ntilde N_2(t)\,,\phantom{xxx}
\end{eqnarray}
where the final term is finite due to condition \eqref{LLN_martingale_bound_first_lemma}. Moreover, as square-integrability implies integrability, the martingale property for the Hilbert space valued process $M$ now follows due to part (a).

(c)\quad Finally, in the last part of the proof we establish the It{\^o}-isometry. To this end we first show that equality \eqref{proof_first_stopped_ito_isometry} holds for all $\tau_k\wedge t$, $k\in \mathbb{N}$. Again we proceed by induction with the induction basis given by \eqref{proof_first_stopped_ito_isometry}. We observe that
\begin{eqnarray}
\|M(\tau_{k+1}\wedge t)-M(\tau_k\wedge t)\|_E^2 & =& \|M(\tau_{k+1}\wedge t)\|_E^2-\|M(\tau_k \wedge t)\|_E^2\nonumber\\[2ex]
&&\mbox{}-2 \bigl(M(\tau_{k}\wedge t),M(\tau_{k+1}\wedge t)-M(\tau_{k}\wedge t)\bigr)_E.\phantom{xxxxx}\label{first_bit_ito_iso_induction}
\end{eqnarray}
Taking the conditional expectation with respect to the stopped $\sigma$-field $\sF_{\tau_k\wedge t}$ we find that the second term in the right hand side of \eqref{first_bit_ito_iso_induction} vanishes as it holds
\begin{eqnarray*}
\lefteqn{\EX\bigl[\bigl(M(\tau_{k}\wedge t),M(\tau_{k+1}\wedge t)-M(\tau_{k}\wedge t)\bigr)_E\,\big|\,\sF_{\tau_k\wedge t}\bigr]\ =\ }\\[2ex] &&\phantom{xxxxxxxxxxxxx} =\ \bigl(M(\tau_k\wedge t),\EX\bigl[M(\tau_{k+1}\wedge t)-M(\tau_{k}\wedge t)\bigr)\,\big|\,\sF_{\tau_k\wedge t}\bigr]\bigr)_E\ =\ 0
\end{eqnarray*}
due to the following properties of the conditional expectation: Firstly, for $E$--valued random variables $X,Y$ such that $\EX\|X\|_E\|Y\|_E<\infty$ 
it holds for $\mathcal{G}$--measurable $X$ that $\EX\bigl[(X,Y)_E|\mathcal{G}\bigr] = \bigl(X,\EX[Y|\mathcal{G}]\bigr)_E$ \cite[Lemma 2.1.2]{Stolze}. Secondly, the Optional Sampling Theorem, i.e., $\EX\bigl[M(\tau_{k+1}\wedge t)\,\big|\,\sF_{\tau_k\wedge t}\bigr]=M(\tau_{k}\wedge t)\bigr)$ in the above application, also holds for Hilbert space-valued martingales\footnote{The Optional Sampling Theorem can be proved similarly to the methods employed for \cite[Lemma 2.1.2]{Stolze} relying on the linearity properties of the Bochner integral and the monotone convergence theorem.}. Thus we obtain
\begin{eqnarray*}
\lefteqn{\EX\bigl[\|M(\tau_{k+1}\wedge t)\|_E^2\big|\sF_{\tau_k\wedge t}\bigr]-\EX\bigl[\|M(\tau_k\wedge t)\|_E^2\big|\sF_{\tau_k\wedge t}\bigr]}\\[1ex]
&&\phantom{xxxxxxxxxxxxxxxxxxx} =\ \EX\bigl[\ntilde N_2(\tau_{k+1}\wedge t)\big|\sF_{\tau_k\wedge t}\bigr] - \EX\bigl[\ntilde N_2(\tau_k\wedge t)\big|\sF_{\tau_k\wedge t}\bigr]\,.
\end{eqnarray*}
Taking the expectation on both sides of this equality and using the induction hypotheses, i.e., the second expectations on both sides of the above equality equate, yields
\begin{equation}\label{prove_sq_integrability_last_eq}
\EX\|M(\tau_{k+1}\wedge t)\|_E^2 = \EX\ntilde N_2(\tau_{k+1}\wedge t)\,.
\end{equation}
We conclude the proof extending the It{\^o}-isometry \eqref{prove_sq_integrability_last_eq} from the stopped processes to the non-stopped process. We have already obtained the upper estimate $\EX\|M(t)\|_E^2\leq \EX\ntilde N_2(t)$, cf.~\eqref{ito_iso_proof_upper_est}. Hence it remains to prove that a lower bound is given by the same term. As $\|M(t)\|_E^2$ is a real-valued submartingale it holds for all $k\geq 1$ due to the standard Optional Sampling Theorem for c\`adl\`ag submartingales, see, e.g., \cite[App.~B]{Jacobsen}, that
\begin{equation*}
\EX\|M(t)\|_E^2\,\geq\,\EX\|M(\tau_k\wedge t)\|_E^2 \,=\, \EX\ntilde N_2(\tau_k\wedge t)\,. 
\end{equation*}
Hence, for $k\to\infty$ we obtain by monotone convergence $\EX\|M(t)\|_E^2\,\geq\, \EX \ntilde N_2(t)$ which, combined with the upper bound \eqref{ito_iso_proof_upper_est}, yields the It{\^o}-isometry \eqref{banach_ito_isometry}. The proof is completed.

%% file: excitable_media_conditions_proof.tex

\subsection{Proof of Theorem \ref{LLN_compartmental_models}  (Conditions for the LLN)}

We apply Theorem \ref{LLN} for the choice of spaces $X=H^1_0(D)$, $H=L^2(D)$ and $E=L^2(D)$. Hence, we have to prove in the following that the assumptions therein are satisfied, i.e., (i) the one-sided Lipschitz condition \eqref{onesided_lip} on the operators $A$ and $B$ defined by \eqref{limit_theorem_apl_pdmp_op}, (ii) the Lipschitz condition on the right hand side of the gating system \eqref{limit_theorem_det_gating_sys}, (iii) the uniform convergence of the generator and (iv) the martingale convergence. Finally, in (v) we extend the convergence in probability due to Theorem \ref{LLN} to convergence in the mean \eqref{example_LLN_conv_in_mean}. In the following we use $\cdot$ to denote the pointwise product of real functions on $D$.\medskip

(i)\quad 
For the non-linear operator $B$ we find that the left hand side in the Lipschitz condition is for almost all $t$ given by a finite sum of terms
\begin{equation}\label{example_summands_in_B}
\langle p_i\cdot (E_i-u)-\nhat p_i\cdot(E_i-v), u-v\rangle_{H^1}\,,\quad i=1,\ldots,m,
\end{equation}
with $u,\,v\in H^1_0(D)$ and $p_i,\,\nhat p_i\in L^2(D)$. Hence, the duality pairing corresponds to the inner product in $L^2(D)$. We estimate each of the summand of the type \eqref{example_summands_in_B} separately. Using the triangle inequality we obtain
\begin{equation*}
\big|\langle p_i\cdot (E_i-u)-\cdot \nhat p_i\cdot(E_i-v), u-v\rangle_{H^1}\big|\,\leq\,|E_i|\,\big|(p_i-\nhat p_i,u-v)_{L^2}\big|+\big|(p_i\cdot u-\nhat p_i\cdot v,u-v)_{L^2}\big|\,.
\end{equation*} 
Here, the first term in this right hand side is further estimated using Cauchy-Schwarz and Young's inequality, which yields
\begin{equation*}
\big|(p_i-\nhat p_i,u-v)_{L^2}\big|\, \leq\, \tfrac{1}{2}\big\|p_i-\nhat p_i\big\|^2_{L^2}+\tfrac{1}{2}\big\|u-v\big\|_{L^2}^2\,.
\end{equation*}
For the second term we obtain, making use of the triangle inequality, Cauchy-Schwarz and Young's inequality and the pointwise bounds on $p_i$ and $v$, the sequence of estimates
\begin{eqnarray*}
\big|(p_i\cdot u-\nhat p_i\cdot v,u-v)_{L^2}\big|&\leq&
\big|(p_i\cdot(u-v),u-v)_{L^2}\big|+\big|(p_i-\nhat p_i,v\cdot(u-v))_{L^2}\big|\\[1ex]
&\leq& \big\|p_i\cdot(u-v)\big\|_{L^2}\big\|u-v\big\|_{L^2}+\big\|p_i-\nhat p_i\big\|_{L^2}\big\|v\cdot(u-v)\big\|_{L^2}\\[1ex]
&\leq&\big\|u-v\big\|_{L^2}^2+\tfrac{\nbar u^2}{2}\,\big\|u-v\big\|_{L^2}^2+\tfrac{1}{2}\big\|p_i-\nhat p_i\big\|_{L^2}^2\,.
\end{eqnarray*}
A summation over all these estimates for $i=1,\ldots m$ yields
\begin{equation*}
\langle B(p,u)-B(\nhat p,v),u-v\rangle_{H^1} \leq m\bigl(1+\tfrac{\nbar u+\nbar u^2}{2}\bigr)\,\|u-v\|_{L^2}^2 + \frac{1+\nbar u}{2} \sum_{i=1}^m \big\|p_i-\nhat p_i\big\|_{L^2}^2\,.
\end{equation*}
Adding the estimate
\begin{equation*}
\langle A(u-v),u-v\rangle_{H^1}\,\leq\, - \gamma_1\|u-v\|_{H^1}^2+\gamma_2\|u-v\|_{L^2}^2\leq \gamma_2\|u-v\|_{L^2}^2
\end{equation*}
for some $\gamma_1,\gamma_2>0$, which holds as the linear operator $A$ is coercive and independent of $p$, we obtain
\begin{equation*}
\langle A(u-v),u-v\rangle_{H^1}+\langle B(p,u)-B(\nhat p,v),u-v\rangle_{H^1} \leq C\,\Bigl(\|u-v\|_{L^2}^2 + \sum_{i=1}^m \big\|p_i-\nhat p_i\big\|_{L^2}^2\Bigr)
\end{equation*}
for a suitable constant $C$. Finally, integrating over $(0,T)$ we find the one-sided Lipschitz condition \eqref{onesided_lip} is satisfied.\medskip

(ii)\quad 
Due to the triangle inequality it suffices to consider differences of the form $\| p_i\cdot q(u)-\nhat p_i\cdot q(v)\|_{L^2}$, where $q$ substitutes for an arbitrary rate function $q_{jk}$. Using the triangle inequality, the pointwise boundedness of $\nhat p_i$ and $q$ by $1$ and $\nbar q$, respectively, and the Lipschitz condition on the rate functions $q$ (with common Lipschitz constant $L$) we obtain
\begin{eqnarray*}
\|p_i\cdot q(u)-\nhat p_i\cdot q(v)\|_{L^2}&\leq& \|p_i\cdot q(u)-\nhat p_i\cdot q(u)\|_{L^2}+\|\nhat p_i\cdot q(u)-\nhat p_i\cdot q(v)\|_{L^2}\\[1ex]
&\leq& \nbar q\,\|p_i-\nhat p_i\|_{L^2}+ L\,\|u-v\|_{L^2}\,.
\end{eqnarray*}
A summation over all such separate estimates, integrating and squaring both resulting sides yield the Lipschitz condition \eqref{det_sys_lipschitz}.\medskip

(iii)\quad In order to prove the convergence of the generators \eqref{finite_var_bound} we employ in the following two technical results which we collect in a separate proposition. Firstly, the purpose of the formula \eqref{example_usable_form_of_gen} is to transform the generator of the PDMP into a form that allows comparison with the deterministic limit system \eqref{limit_theorem_det_gating_sys}. Secondly, the inequality \eqref{unifrom_estimate_ogH_1_norm_example}, which bounds the norm $\|U^n\|_{L^2((0,T),H^1)}$ by a deterministic constant uniformly over $n\in\mathbb{N}$, is used repeatedly in the subsequent estimation procedures. 

\begin{prop}\label{proposition_in_examepl_proof}\begin{enumerateI} \item The generator of the PDMP satisfies
\begin{equation}\label{example_usable_form_of_gen}
\hspace{-10pt}\Lambda^n(u,\theta^n)\int_{K_n} \Bigl(z^n_i(\xi)-z^n_i(\theta^n)\Bigr)\,\mu^n\bigl((u,\theta^n),\dif\xi\bigr)\,=\, \sum_{j\neq i} \Bigl(z^n_j(\theta^n)\cdot q_{ji}^n(u)-z^n_i(\theta^n)\cdot q_{ij}^n(u)\Bigr)
\end{equation}
where
\begin{equation*}
q_{ij}^n(u)=\sum_{k=1}^{p(n)} Q^{k,n}_{ij}(u)\,\mathbb{I}_{\,D_{k,n}}\,\in L^2(D) \,.
\end{equation*}
\item For all $n\in\mathbb{N}$ and all $T>0$ it holds that
\begin{equation}\label{unifrom_estimate_ogH_1_norm_example}
\int_0^T\|U^n_t\|_{H^1}^2\,\dif t\leq
C_1(1+T)\e^{2C_2T},
\end{equation}
where the constants $C_1,C_2$ are deterministic and independent of $n\in\mathbb{N}$.
\end{enumerateI}
\end{prop}

\begin{proof} (a)\quad We denote by $\theta^{n}_{k,i\to j}$ for all $k=1,\ldots,p(n)$ and all $i\neq j$, $i,j=1,\ldots m$ the configuration in $K_n$ that arises from the configuration $\theta^n$ through the event that a channel in state $i$ located in the compartment $D_{k,n}$ switches to state $j$. Then simple reorganisation of finite sums yields
\begin{eqnarray*}
\lefteqn{\Lambda^n(u,\theta^n)\int_{K_n} \Bigl(z^n_i(\xi)-z^n_i(\theta^n)\Bigr)\,\mu^n\bigl((u,\theta^n),\dif\xi\bigr)}\\
&=& \sum_{k=1}^{p(n)}\sum_{j\neq i}\Bigl(z_i^n(\theta^n_{k,j\to i})-z_i^n(\theta^n)\Bigr)\,\theta^{k,n}_j\,Q_{ji}^{k,n}(u) + \sum_{k=1}^{p(n)}\sum_{j\neq i}\Big(z_i^n(\theta^n_{k,i\to j})-z_i^n(\theta^n)\Bigr)\,\theta^{k,n}_i\,Q_{ij}^{k,n}(u)\\
&=& \sum_{k=1}^{p(n)}\sum_{j\neq i}\Bigl(\frac{1}{l(k,n)}\,\mathbb{I}_{\,D_{k,n}}\Bigr)\,\theta^{k,n}_j\,Q_{ji}^{k,n}(u) + \sum_{k=1}^{p(n)}\sum_{j\neq i}\Bigl(-\frac{1}{l(k,n)}\,\mathbb{I}_{\,D_{k,n}}\Bigr)\,\theta^{k,n}_i\,Q_{ij}^{k,n}(u)\\
&=& \sum_{j\neq i}z^n_j(\theta^n)\cdot\Bigl(\,\sum_{k=1}^{p(n)} Q_{ji}^{k,n}(u)\,\mathbb{I}_{\,D_{k,n}}\Bigr)\, - \sum_{j\neq i}z^n_i(\theta^n)\cdot\Bigl(\,\sum_{k=1}^{p(n)} Q_{ij}^{k,n}(u)\,\mathbb{I}_{\,D_{k,n}}\Bigr)\,.
\end{eqnarray*}
Thus we obtain that the generator  satisfies \eqref{example_usable_form_of_gen}. 

(b)\quad By definition of a PDMP it holds that the component $(U^n_t)_{t\geq 0}$ is the weak solution of the evolution equation
\begin{equation*}
\dot U^n_t = AU^n_t +\sum_{i=1}^m g_i\,z^n_i(\Theta^n_t)\,(E_i-U^n_t)
\end{equation*}
with initial condition $U^n_0$. We consider the reaction term in this equation as a given inhomogeneity. Then standard estimation procedures from the theory of linear parabolic partial differential equations, cf.~\cite[Sec.~7]{Evans}, yield, after appropriately estimating the inhomogeneous term,
\begin{equation*}
\int_0^T\|U^n_t\|_{H^1}^2\,\dif t\leq
K_1\e^{2K_2T}\Bigl(\|U^n_0\|_{L^2}^2+2\nbar u^2\sum_i \|g_i\|_{L^\infty}\int_0^T\|z^n_i(\Theta^n_t)\|_{L^1}\,\dif t\Bigr),
\end{equation*}
where the constants $K_1,K_2$ are deterministic and depend only on the domain $D$ and the coefficients of $A$. Further, it holds that $\|z^n_i(\Theta^n_t)\|_{L_1}\leq |D|$ and the sequence of initial conditions is bounded by assumption as $U^n_0(x)\in[\nbar u_-,\nbar u_+]$ for all $x\in \nbar D$ almost surely. The inequality \eqref{unifrom_estimate_ogH_1_norm_example} follows.
\end{proof}
We now proceed to the actual proof of the convergence \eqref{finite_var_bound}. To this end we need to consider for almost every $t$ and all $i=1,\ldots,m$, the convergence in $L^2(D)$ of \eqref{example_usable_form_of_gen} to $F_i(z^n(\Theta^n_t),U^n_t)$ where $F_i$ is as defined in \eqref{limit_theorem_det_gating_sys}. That is, we have to estimate
\begin{equation}\label{arnold_approach_the_diff_to_be_estimated}
\Big\|\sum_{j\neq i} \Bigl(z^n_j(\Theta^n_t)\cdot q_{ji}^n(U^n_t)-z^n_i(\Theta^n_t)\cdot q_{ij}^n(U^n_t)\Bigr)-\sum_{j\neq i} \Bigl(z^n_j(\Theta^n_t)\cdot q_{ji}(U^n_t)-z^n_i(\Theta^n_t)\cdot q_{ij}(U^n_t)\Bigr)\Big\|_{L^2}\,.
\end{equation}
We find that the single summands in the two summations match up and thus it suffices to consider each of them separately. Employing the boundedness of the coordinate functions, i.e., $\|z_j^n(\Theta^n_t)\|_{L^\infty}\leq 1$ we obtain the estimates
\begin{eqnarray}\label{estimate_example_gen_convergence_0}
\big\|z^n_j(\Theta^n_t)\cdot q_{ji}^n(U^n_t)-z^n_j(\Theta^n_t)\cdot q_{ji}(U^n_t)\big\|_{L^2}^2
\!&=&\!\|z^n_j(\Theta^n_t)\big\|_{L^\infty}^2\,\big\|q_{ji}^n(U^n_t)-q_{ji}(U^n_t)\big\|_{L^2}^2\nonumber\\
&\leq&\!\Big\|\sum_{k=1}^{p(n)}\Bigl(\mathbb{I}_{\,D_{k,n}} Q^{k,n}_{ij}(U^n_t)\Bigr)-q_{ij}(U^n_t)\Big\|_{L^2}^2\nonumber\\
&=&\!\sum_{k=1}^{p(n)}\int_{D_{k,n}}\big|Q^{k,n}_{ij}(U^n_t)-q_{ij}(U^n_t(x))\big|^2\,\dif x\,.\phantom{xxxxx}
\end{eqnarray}
For the last equality we have used that the summands are mutually orthogonal in $L^2(D)$. Next we estimate each of the remaining integrals in \eqref{estimate_example_gen_convergence_0} using the Lipschitz continuity of $q_{ij}$ and Poincar\'e's inequality in $L^2(D_{k,n})$, i.e.,
\begin{eqnarray*}
\lefteqn{\int_{D_{k,n}}\!\Big|q_{ij}\Bigl(\frac{1}{|D_{k,n}|}\int_{D_{k,n}}\!\! U^n_t(y)\,\dif y\Bigr)-q_{ij}(U^n_t(x))\Big|^2\dif x}\\
&&\phantom{xxxxxxxxxxxxxxxxxxxx}\leq\ L^2\int_{D_{k,n}}\!\Big|\frac{1}{|D_{k,n}|}\int_{D_{k,n}} \!\!U^n_t(y)\,\dif y-U^n_t(x)\Big|^2\dif x\\[2ex]
&&\phantom{xxxxxxxxxxxxxxxxxxxx}\leq\ L^2\pi^{-2}\textnormal{diam}(D_{k,n})^2\,\|\nabla U^n_t\|_{L^2(D_{k,n})}^2\,,
\end{eqnarray*}
where $\|\nabla U^n_t\|_{L^2(D_{k,n})}$ is the norm in $L^2(D_{k_n})$ of the Euclidean norm of the gradient vector $\nabla U^n_t$. Here we have employed that for convex domains the optimal Poincar\'e constant is given by $\pi^{-1}\textnormal{diam}(D_{k,n})$ \cite{Payne}. Hence, a summation over all $k=1,\ldots, p(n)$ and employing the estimate $\|\nabla U^n_t\|_{L^2}^2\leq \|U^n_t\|_{H^1}^2$ yields
\begin{equation*}
\big\|q_{ji}^n(U^n_t)-q_{ji}(U^n_t)\big\|_{L^2}^2\,\leq \delta_+(n)^2\,L^2\pi^{-2}\|U^n_t\|_{H^1}^2\,.
\end{equation*}
Integrating over $(0,T)$ we therefore obtain for \eqref{arnold_approach_the_diff_to_be_estimated} the estimate
\begin{eqnarray*}
\lefteqn{\int_0^T \Big\|\bigl[\mathcal{A}(\phi,z^n_j(\cdot))_{L^2}\bigr](U^n_t,\Theta^n_t)-F_j(z^n(\Theta^n_t),U^n_t)\Big\|_{L^2}^2\,\dif t }\\
&&\phantom{xxxxxxxxxxxxxxxxxxxxxxxxxx}\leq\ \delta_+(n)^2\,L^2\pi^{-2}\,2(m-1)\int_0^T\|U^n_t\|_{H^1}^2\,\dif t\,.
\end{eqnarray*}
Finally, the norm $\|U^n_t\|_{L^2((0,T),H^1)}$ is bounded independently of $n\in\mathbb{N}$ by a deterministic constant due to Proposition \ref{proposition_in_examepl_proof}(b). This upper bound holds for almost all paths of the PDMPs $(U^n_t,\theta^n_t)_{t\geq 0}$ and thus there exists a constant $C>0$ independent of $n$ such that
\begin{equation}
\int_0^T\Big\|\bigl[\mathcal{A}(\phi,z^n_j(\cdot))_{L^2}\bigr](U^n_t,\Theta^n_t)-F_j(z^n(\Theta^n_t),U^n_t)\Big\|^2_{L^2}\,\dif t\,\leq\,\delta_+(n)^{2}\,C
\end{equation}
almost surely. Due to the assmuption \eqref{LLN_compartmental_model_part_conds} the estimate in the right hand side converges to zero for $n\to\infty$ and the convergence \eqref{finite_var_bound} follows. \medskip

(iv)\quad Next we consider convergence in probability of the martingale part. To this end we employ Lemma \ref{martingale_bound_1}. As before we denote by $\theta^n_{k,i\to j}$ the channel configuration that arises from the configuration $\theta^n$ if a channel in compartment $D_{k,n}$ switches from state $i$ to state $j$. Then it holds that
\begin{eqnarray*}
\lefteqn{\Lambda^n(U^n_s,\Theta^n_s)\int_{K_n} \|z^n_i(\xi)-z^n_i(\Theta^n_s)\|_{L^2}^2\,\mu^n\bigl((U^n_s,\Theta^n_s),\dif\xi\bigr)}\\
&=& \sum_{k=1}^{p(n)}\sum_{j\neq i}\Bigl(\|z^n_i(\theta^n_{k,i\to j}(s))-z^n_i(\Theta^n_s)\|_{L^2}^2\, Q^{k,n}_{ij}(U^n_s)\,\Theta^n_i(s)\\
&&\phantom{xxxxxxxxxxxxxxxxxxxxxxxx} +  \|z^n_i(\theta^n_{k,j\to i}(s))-z^n_i(\Theta^n_s)\|_{L^2}^2\, Q^{k,n}_{ji}(U^n_s)\,\Theta^n_j(s) \Bigr)\\
&\leq& \nbar q \sum_{k=1}^{p(n)}\frac{|D_{k,n}|}{l(k,n)^2}\sum_{j\neq i}\bigl(\Theta^n_i(s) +\Theta^n_j(s)\bigr)\,.
\end{eqnarray*}
This implies that
\begin{equation*}
\EX^n\,\int_0^t\Bigl[\Lambda^n(U^n_s,\Theta^n_s)\int_{K_n} \|z^n_i(\xi)-z^n_i(\Theta^n_s)\|_{L^2}^2\,\mu^n\bigl((U^n_s,\Theta^n_s),\dif\xi\bigr)\Bigr]\,\dif s=\landau\bigl(\ell_-(n)^{-1}\bigr)\,.
\end{equation*}
Hence, under condition \eqref{LLN_compartmental_model_part_conds} the assumption of Lemma \ref{martingale_bound_1} is satisfied.\medskip

(v)\quad Finally, we extend the convergence in probability to convergence in the mean for the individual components being in the space $L^2((0,T),L^2)$, see the remark following Theorem \ref{LLN}. First of all note that the components are bounded, i.e.,
\begin{equation*}
\|U^n_t-u(t)\|_{L^2}\leq 2\nbar u\,|D|,\qquad \|z^n_i(\Theta_t^n)-p_i(t)\|_{L^2}\leq 2\,|D|\,. 
\end{equation*}
Therefore it holds that
\begin{equation*}
\|X^n-X\|:=\|U^n-u\|_{L^2((0,T),L^2)}+\sum_{i=1}^m\|z^n_i(\Theta^n)-p_i\|_{L^2((0,T),L^2)}\,\leq\, C
\end{equation*}
for a suitable deterministic bound $C<\infty$ independent of $n\in\mathbb{N}$. Then for all $\eps_0>0$ it holds that
\begin{eqnarray*}
\EX^n\|X_n-X\|&=&\EX^n\bigl[\|X_n-X\|\,\mathbb{I}_{[\|X_n-X\|\leq\eps_0]}\bigr]+\EX^n\bigl[\|X_n-X\|\,\mathbb{I}_{[\|X_n-X\|>\eps_0]}\bigr]\\[2ex]
&\leq& \eps_0 + M\,\Pr^n\bigl[\|X_n-X\|>\eps_0\bigr]\,.
\end{eqnarray*}
Next choose $\eps_0<\eps/2$ and note that due to the convergence in probability there exists an $N_\eps$ such that $M\,\Pr^n\bigl[\|X_n-X\|>\eps_0\bigr]\leq \eps/2$ for all $n>N_\eps$. Hence, for every $\eps>0$ there exists an $N_\eps$ such that $\EX^n\|X_n-X\|<\eps$ for all $n>N_\eps$. Convergence in the mean is proven.


\subsection{Proof of Theorem \ref{CLT_compartmental_model}  (Conditions for the CLT)}

In order to prove Theorem \ref{CLT_compartmental_model} we employ Theorem \ref{clt_by_local_martingale} for the space $E=H^{-2s}(D)$ where $s$ is the smallest integer such that $s>d/2$. We usually employ the simpler notation $E$ and $\mathcal{E}=E^m$ throughout the proof, however occasionally switch to $H^{-2s}(D)$ if we want to emphasise the specific choice of the Hilbert space. The reason choosing this particular integer $s$ is that it is the smallest integer such that the embedding of $H^{2s}(D)$ into $H^s(D)$ is of Hilbert-Schmidt type\footnote{The embedding of a Hilbert space $X$ into another Hilbert space $H$ is of \emph{Hilbert-Schmidt type} if $\sum_{k\in\mathbb{N}} \|\varphi_k\|_H^2<\infty$ for every orthonormal basis $(\varphi_k)_{k\in\mathbb{N}}$ of $X$.} due to Maurin's Theorem \cite[Thm.~6.61]{Adams} and $H^{s}(D)$ is embedded in $C(\nbar D)$ due to the Sobolev Embedding Theorem. These two properties are essential in order to prove the conditions \eqref{first_clt_section_theorem_2nd_mom_cond} -- \eqref{clt_martingale_cond_1} of Theorem \ref{clt_by_local_martingale}. 
All conditions except \eqref{clt_martingale_cond_2}, which establishes the convergence of the quadratic variation, are straightforward consequences of the assumptions of the theorem. These are shown in part (i) of the subsequent proof. For condition \eqref{clt_martingale_cond_2} more involved estimation procedures are necessary which are presented in part (ii). 
\medskip

(i)\quad We first show condition \eqref{first_clt_section_theorem_2nd_mom_cond}. As in the preceding section $\theta^{n}_{k,i\to j}$ denotes the element of $K_n$ that differs from $\theta^n$ by one channel in the $k$th compartment being in state $i$ instead of state $j$. Then, the Sobolev Embedding Theorem yields the estimate
\begin{equation}\label{proof_comp_clt_jump_bounds}
\|z^n_i(\theta^n_{k,i\to j}(t))-z^n(\Theta^n_t)\|_{E}=\sup_{\|\phi\|_{H^{2s}}=1}\big|l(k,n)^{-1}\langle\phi,\mathbb{I}_{D^{k,n}}\rangle_{H^{2s}}\big| \leq \frac{C}{l(k,n)}\,|D^{k,n}|\,,
\end{equation}
where $C$ is a constant resulting from the continuous embedding of $H^{2s}(D)$ into $C(\nbar D)$. Using this estimate for the jump heights in the space $H^{-2s}(D)$ we find similarly to part (iv) of the proof of Theorem \ref{LLN_compartmental_models} that it holds
\begin{equation*}
\alpha_n\,\EX^n\int_0^T\Bigl[\Lambda^n(U^n_t,\Theta^n_t)\int_{K_n}\|z^n(\xi)-z^n(\Theta^n_t)\|_\mathcal{E}^2\,\mu^n\bigl((U^n_t,\Theta^n_t),\dif\xi\bigr)\,\dif t\Bigr]= \landau(1)\,.
\end{equation*}
Hence, condition \eqref{first_clt_section_theorem_2nd_mom_cond} is satisfied. Moreover, we infer from \eqref{proof_comp_clt_jump_bounds} that the rescaled jump sizes are bounded almost surely uniformly, i.e., condition \eqref{clt_martingale_cond_1} is satisfied. Particularly, it holds that $\sqrt{\alpha_n}\,\|z^n(\theta^n_{k,i\to j})-z^n(\theta^n)\|_{E}=\landau(\ell_-(n)^{-1/2})$. This implies that for arbitrary $\beta>0$ and any $\Phi\in (H^{2s}(D))^m$ there exists $N_\beta$ such that for all $n\geq N_\beta$
\begin{equation*}
\int_{\sqrt{\alpha_n}|\langle\Phi,z^n_i(\xi)-z^n_i(\theta^n)\rangle_\mathcal{E}|>\beta} \mu^n\bigl((u,\theta^n),\dif\xi\bigr)\, =\, 0
\end{equation*}
holds for all values $(u,\theta^n)$ the PDMP attains. Therefore, by dominated convergence we infer that also  condition \eqref{conv_gen_support_thm} is satisfied. It remains to consider condition \eqref{alternative_tightness_cond}. To this end let $(\varphi_k)_{k\in\mathbb{N}}$ be an orthonormal basis in $(H^{2s}(D))^m$, where $\varphi_k=(\varphi_k^1,\ldots,\varphi_k^m)$ and hence $(\varphi_k^{i})_{k\in\mathbb{N}}$ is an orthonormal basis in $H^{2s}(D)$ for all $i=1,\ldots,m$. Then we obtain the estimate
\begin{eqnarray*}
\lefteqn{\langle\varphi_k,G^n(U^n_t,\Theta^n_t)\varphi_k\rangle_\mathcal{E}}\\
&=&\Lambda^n(U^n_t,\Theta^n_t) \int_{K_n} \Bigl(\sum_{i=1}^m \langle\varphi_k^i,z^n_i(\xi)-z^n_i(\Theta^n_t)\rangle_{H^{-2s}}\Bigr)^2\,\mu^n\bigl((U^n_t,\Theta^n_t),\dif \xi\bigr)\\
&\leq &m\sum_{i=1}^m \|\varphi_k^i\|_{H^s}^2\,\Bigl(\Lambda^n(U^n_t,\Theta^n_t) \int_{K_n} \|z^n_i(\xi)-z^n_i(\Theta^n_t)\|_{H^{-s}}^2\,\mu^n\bigl((U^n_t,\Theta^n_t),\dif \xi\bigr)\Bigr)\,.
\end{eqnarray*}
Here we have employed for the the individual summands in the right hand side that for $z^n_i(\xi)-z^n_i(\Theta^n_t)\in L^2(D)$ the duality pairing in $H^{2s}(D)$ equals the duality pairing in $H^s(D)$. Further, note that $\|z^n_i(\xi)-z^n_i(\theta^n_t)\|_{H^{-s}}$ satisfies an estimate analogous to \eqref{proof_comp_clt_jump_bounds} due to the continuous embedding of $H^s(D)$  in $C(\nbar D)$. Therefore we overall obtain that
\begin{equation*}
\alpha_n\,\langle\varphi_k,G^n(U^n_t,\Theta^n_t)\varphi_k\rangle_\mathcal{E}\leq C\,\sum_{i=1}^m\|\varphi_k^i\|_{H^{s}}^2  
\end{equation*}
for a suitable non-random constant $C$ independent of $n$. Finally, set $\gamma_k:=\sum_{i=1}^m\|\varphi_k^i\|_{H^{s}}^2$ then it holds that $\sum_{k\in\mathbb{N}}\gamma_k<\infty$ as the embedding $H^{2s}(D)\hookrightarrow H^s(D)$ is of Hilbert-Schmidt type. We infer that condition \eqref{alternative_tightness_cond} is satisfied. \medskip

(ii)\quad In the second part of the proof we establish the central condition \eqref{clt_martingale_cond_2} of the convergence of the quadratic variation. For simplicity of notation we omit the time argument of the PDMP paths and the deterministic solution as the following estimates hold for almost all $t$. First of all we expand the quadratic variation of the martingales into the finite sum
\begin{eqnarray*}
\lefteqn{\Lambda^n(U^n,\Theta^n)\int_{K_n}\langle\Phi,z^n(\xi)-z^n(\Theta^n)\rangle_\mathcal{E}^2\,\mu^n\bigl((U^n,\Theta^n),\dif\xi\bigr)}\\
&\hspace{-20pt}=&\hspace{-15pt}\sum_{j=1}^m\sum_{\substack{i=1\\ i\neq j}}^m\sum_{k=1}^{p(n)} \frac{\Theta^{k,n}_j}{l(k,n)^2}\,Q^{k,n}_{ji}(U^n) \langle\phi_j,\mathbb{I}_{\,D_{k,n}}\rangle_E^2 +\sum_{j=1}^m\sum_{\substack{i=1\\ i\neq j}}^m\sum_{k=1}^{p(n)} \frac{\Theta^{k,n}_i}{l(k,n)^2}\,Q^{k,n}_{ij}(U^n) \langle\phi_j,\mathbb{I}_{\,D_{k,n}}\rangle_E^2\\
&&\mbox{} -\sum_{\substack{i,j=1\\ i\neq j}}^m\,\sum_{k=1}^{p(n)} \Bigl(\frac{\Theta^{k,n}_i}{l(k,n)^2}\,Q^{k,n}_{ij}(U^n)+\frac{\Theta^{k,n}_j}{l(k,n)^2}\,Q^{k,n}_{ji}(U^n)\Bigr) \langle\phi_i,\mathbb{I}_{\,D_{k,n}}\rangle_E\langle\phi_j,\mathbb{I}_{\,D_{k,n}}\rangle_E\,.
\end{eqnarray*}
We find that the terms in this summation match with the integral terms in the definition of the operator $G(u,p)$ in \eqref{def_example_cov_operator}. Thus, due to the triangle inequality it suffices to consider the convergence of the single summands separately, i.e., we have to consider, on the one hand, for all $j=1,\ldots,m$ and $i\neq j$ the differences
\begin{equation}\label{proof_compartmental_model_var_sum_1} 
\Big|\int_D p_j(x)\,q_{ji}(u(x))\,\phi_j^2(x)\,\dif x\,-\,\alpha_n\sum_{k=1}^{p(n)} \frac{\Theta^{k,n}_j}{l(k,n)^2}\,Q^{k,n}_{ji}(U^n) \langle\phi_j,\mathbb{I}_{\,D_{k,n}}\rangle_E^2\Big|
\end{equation}
and, on the other hand, for all $i,j=1,\ldots,m$ such that $i\neq j$ the differences
\begin{equation}\label{proof_compartmental_model_var_sum_2} 
\Big|\int_D p_i(x)\,q_{ij}(u)\,\phi_i(x)\phi_j(x)\,\dif x - \alpha_n\sum_{k=1}^{p(n)} \frac{\Theta^{k,n}_i}{l(k,n)^2}\,Q^{k,n}_{ij}(U^n) \langle\phi_i,\mathbb{I}_{\,D_{k,n}}\rangle_E\langle\phi_j,\mathbb{I}_{\,D_{k,n}}\rangle_E\Big|\,. 
\end{equation}
We next estimate these terms separately in parts (ii.1) and (ii.2). Finally, in part (ii.3) the estimates are combined to prove the convergence of the quadratic variation.

(ii.1)\quad A further application of the triangle inequality yields\begin{eqnarray}\label{example_cov_op_term_I}
\eqref{proof_compartmental_model_var_sum_1}&=& \Big| \int_D p_j(x)\,q_{ji}(u(x))\,\phi_j^2(x)\,\dif x\, -\, \int_D z^n_j(\Theta^n)(x)\,q_{ji}(U^n(x))\,\phi_j^2(x)\,\dif x\Big|\nonumber\\
&&\hspace{-35pt}\mbox{} +\Big|\sum_{k=1}^{p(n)} \frac{\Theta^{k,n}_j}{l(k,n)}\,\int_{D_{k,n}} q_{ji}(U^n(x))\,\phi_j^2(x) \,\dif x\,-\,\alpha_n\sum_{k=1}^{p(n)} \frac{\Theta^{k,n}_j}{l(k,n)^2}\,Q^{k,n}_{ji}(U^n) \langle\phi_j,\mathbb{I}_{\,D_{k,n}}\rangle_E^2 \Big|\,.\nonumber\\
\end{eqnarray}
We estimate the two resulting differences separately and obtain for the first term in the right hand side of \eqref{example_cov_op_term_I} the estimate
\begin{eqnarray}\label{example_cov_op_term_I_a}
\lefteqn{\Big|\int_D p_j(x)\,q_{ji}(u(x))\,\phi_j^2(x)\,\dif x\, -\, \int_D z^n_j(\Theta^n)(x)\,q_{ji}(U^n(x))\,\phi^2(x)\,\dif x \Big|}\nonumber\\
&&\phantom{xxxx}\leq\ \Big|\int_D p_j(x)\,q_{ji}(u(x))\,\phi_j^2(x)\,\dif x - \int_D z^n_j(\Theta^n)(x)\,q_{ji}(u(x))\,\phi^2(x)\,\dif x\Bigl|\nonumber\\
&&\phantom{xxxx\leq}\ \mbox{} +\Big|\int_D z^n_j(\Theta^n)(x)\,q_{ji}(u(x))\,\phi^2(x)\,\dif x - \int_D z^n_j(\Theta^n)(x)\,q_{ji}(U^n(x))\,\phi^2(x)\,\dif x \Big|\nonumber\\[2ex]
&&\phantom{xxxx}\leq\ \nbar q\,\|\phi_j\|_{L^\infty}^2\,\|p_j-z^n_j(\Theta^n)\|_{L^1}\, +\, L\,\|\phi_j\|_{L^\infty}^2\,\|u-U^n\|_{L^1}\,.
\end{eqnarray}
For the second term in the right hand side of \eqref{example_cov_op_term_I} we obtain by employing $\Theta^{k,n}_j/l(k,n)$ $\leq 1$ the estimate
\begin{equation}\label{example_cov_op_term_II}
\sum_{k=1}^{p(n)}\Big|\int_{D_{k,n}}\! q_{ji}(U^n(x))\,\phi_j^2(x) \,\dif x-\frac{\alpha_n}{l(k,n)}\,q_{ji}\Bigl(\frac{1}{|D_{k,n}|}\int_{D_{k,n}}\! U^n(x)\,\dif x\Bigr)\, \Bigl(\int_{D_{k,n}}\!\phi_j(x)\,\dif x\Bigr)^2\Big|
\end{equation}
and we continue estimating each summand therein separately. We begin employing the Mean Value Theorem to expand the rate function $q_{ji}$ in the integral in the left hand side such that
\begin{equation}\label{mean_value_theorem_expansion_in_the_quar_proof}
q_{ji}(U^n(x))=q_{ij}\Bigl(\frac{1}{|D_{k,n}|}\int_{D_{k,n}}\!\! U^n(y)\,\dif y\Bigr) + q_{ji}'(\vartheta^{k,n}(x))\Bigl(U^n(x)-\frac{1}{|D_{k,n}|}\int_{D_{k,n}}\!\! U^n(y)\,\dif y\Bigr),
\end{equation}
where $\vartheta^{k,n}(x)$ denotes an appropriate mean value. For now we omit the remainder term, i.e., the second term in the right hand side of \eqref{mean_value_theorem_expansion_in_the_quar_proof}, a consideration of which is deferred. Hence, we obtain for the absolute value in each summand in \eqref{example_cov_op_term_II} the estimate
\begin{equation*}
q_{ji}\Bigl(\frac{1}{|D_{k,n}|}\int_{D_{k,n}} U^n(y)\,\dif y\Bigr)\,\Big|\int_{D_{k,n}}\phi^2_j(x)\,\dif x - \frac{\alpha_n}{l(k,n)}\Bigl(\int_{D_{k,n}}\phi_j(x)\,\dif x\Bigr)^2\Big|\,.
\end{equation*}
We note that $q_{ji}$ is bounded by $\nbar q$ and continue estimating which yields 
\begin{eqnarray}
&\leq&\nbar q\, |D_{k,n}|\,\Big|\frac{1}{|D_{k,n}|}\int_{D_{k,n}}\phi_j^2(x)\,\dif x-\frac{\alpha_n |D_{k,n}|^2}{l(k,n)|D_{k,n}|}\Bigl(\frac{1}{|D_{k,n}|}\int_{D_{k,n}}\phi_j(x)\,\dif x\Bigr)^2\Big|\nonumber\\
&\leq& \nbar q\,\int_{D_{k,n}}\Bigl(\phi_j(x)-\frac{1}{|D_{k,n}|}\int_{D_{k,n}}\phi_j(y)\,\dif y\Bigr)^2\,\dif x\label{example_cov_op_term_III_1}\\
&&\mbox{}+\nbar q\,|D_{k,n}|\,\Big|\Bigl(1-\frac{\alpha_n |D_{k,n}|^2}{l(k,n)|D_{k,n}|}\Bigr)\,\Bigl(\frac{1}{|D_{k,n}|}\int_{D_{k,n}}\phi_j(x)\,\dif x\Bigr)^2\Big| \label{example_cov_op_term_III_1_2}
\end{eqnarray}
The term \eqref{example_cov_op_term_III_1} is estimated using Poincar\'e's inequality which yields an upper bound by\linebreak[4] $\nbar q\,\pi^{-2}\,\textnormal{diam}^2(D_{k,n})\,\|\nabla\phi\|^2_{L^2(D_{k,n})}$. For the terms \eqref{example_cov_op_term_III_1_2} a summation over all $k=1,\ldots,p(n)$ yields
\begin{equation}\label{example_cov_op_term_III_1_2_e}
\nbar q\,\sum_{k=1}^{p(n)} |D_{k,n}|\,\Big|1-\frac{\alpha_n |D_{k,n}|^2}{l(k,n)|D_{k,n}|}\Big|\Bigl(\frac{1}{|D_{k,n}|}\int_{D_{k,n}}\phi_j(x)\,\dif x\Bigr)^2
\leq \nbar q\,\Big|1-\frac{\ell_-(n)\,\nu_-(n)}{\ell_+(n)\,\nu_+(n)}\Big|\,\|\phi_j^n\|_{L^2}^2\,,
\end{equation}
where $\phi^n_j$ is a piecewise constant approximation to $\phi_j$ defined by
\begin{equation*}
\phi_j^n:=\sum_{k=1}^{p(n)}\Bigl(\frac{1}{|D_{k,n}|}\int_{D_{k,n}}\phi_j(x)\,\dif x\Bigr)\,\mathbb{I}_{D_{k,n}}\,. 
\end{equation*}
As $\phi_j^n$ converges to $\phi_j$ in $L^2(D)$ it holds that the sequence of norms converge, hence $\|\phi^n_j\|_{L^2}$ is a bounded sequence. Therefore the right hand side in \eqref{example_cov_op_term_III_1_2_e} is a componentwise product of convergent sequences. The sequence $|1-(\ell_-(n)\,\nu_-(n)/(\ell_+(n)\,\nu_+(n))|$ converges to zero, cf.~condition \eqref{example_CLT_add_condition}, thus the right hand side in \eqref{example_cov_op_term_III_1_2_e} converges to zero for $n\to\infty$.

Finally, it remains to consider the term arising from the remainder in the expansion of $q_{ji}$, see \eqref{mean_value_theorem_expansion_in_the_quar_proof}, inserted into \eqref{example_cov_op_term_II}. By assumption $q_{ji}'$ is bounded (by a constant $\nbar q$). Therefore we obtain an upper bound on the respective term by
\begin{eqnarray*}
\nbar q\,\|\phi_j\|_{L^\infty}^2\!\sum_{k=1}^{p(n)} \int_{D_{n,k}}\!\! \Big|U^n(x)-\frac{1}{|D_{k,n}|}\int_{D_{k,n}}\!\! U^n(y)\,\dif y\Bigr|\,\dif x
\!\!&\leq&\!\! \nbar q\,\|\phi_j\|_{L^\infty}^2\!\sum_{k=1}^{p(n)} \frac{\delta_+(n)}{2}\|\nabla U^n\|_{L^1(D_{n,k})}\\
\!\!&\leq&\!\! \nbar q\,\|\phi_j\|_{L^\infty}^2 \frac{\delta_+(n)}{2}\|\nabla U^n\|_{L^1}\,.
\end{eqnarray*}
Here we have employed the Poincar\'e inequality in $L^1$ with optimal Poincar\'e constant given by $\textnormal{diam}(D_{k,n})/2$ \cite{Acosta}.

A combination of these estimates yields an upper bound to \eqref{proof_compartmental_model_var_sum_1} by
\begin{equation}\label{result_ii_1} 
\eqref{proof_compartmental_model_var_sum_1}\ \leq\ C_\Phi\,\Bigl(\,\|p_j-z^n_j(\Theta^n)\|_{L^1}+ \|u-U^n\|_{L^1}+\delta_+(n)^2\|\nabla U^n\|_{L^1}+\delta^2_+(n)+\delta_+(n)+R(n)\,\Bigr)\,,
\end{equation}
where the term $R(n)$ is given by the right hand side of \eqref{example_cov_op_term_III_1_2_e} and converges to zero for $n\to\infty$. The constant $C_\Phi<\infty$ is a suitable deterministic constant independent of $n\in\mathbb{N}$ which depends on $\Phi\in (H^{2s}(D))^m$ via the norm in $H^s(D)$ of the components of $\Phi$.

(ii.2)\quad Next we consider the mixed terms \eqref{proof_compartmental_model_var_sum_2}. Analogously to part (ii.1) we apply the triangle inequality and obtain
\begin{eqnarray*}
\eqref{proof_compartmental_model_var_sum_2}
&\!\!\leq\!&\!\! \Big|\int_D p_j(x)\,q_{ji}(u(x))\,\phi_j(x)\,\phi_i(x)\,\dif x\, -\, \int_D z^n_j(\Theta^n)(x)\,q_{ji}(U^n(x))\,\phi_j(x)\,\phi_i(x)\,\dif x\Big|\\
&& \mbox{} +\Big|\sum_{k=1}^{p(n)} \frac{\Theta^{k,n}_j}{l(k,n)}\,\int_{D_{k,n}} q_{ji}(U^n(x))\,\phi_j(x)\,\phi_i(x) \,\dif x\\
&&\phantom{xxxxxxxxxxxxxxxxx}\mbox{}-\,\alpha_n\sum_{k=1}^{p(n)} \frac{\Theta^{k,n}_j}{l(k,n)^2}\,Q^{k,n}_{ji}(U^n) \langle\phi_j,\mathbb{I}_{\,D_{k,n}}\rangle_{E} \langle\phi_i,\mathbb{I}_{\,D_{k,n}}\rangle_{E}\Big|
\end{eqnarray*}
As in (ii.1) we obtain for the first term in this right hand side an upper bound by
\begin{equation*}
\nbar q\,\|\phi_i\|_{L^\infty}\|\phi_j\|_{L^\infty}\|p_j-z^n_j(\Theta^n)\|_{L^1}+L\, \|\phi_i\|_{L^\infty}\|\phi_j\|_{L^\infty}\|u-U^n\|_{L^1}\,.
\end{equation*}
Also the second term is treated as in (ii.1), i.e., applying the Mean Value Theorem and estimating the resulting terms accordingly. In particular the remainder term is estimated completely analogously. Therefore, the only term we are left to estimate is
\begin{eqnarray}
&&\hspace{-25pt}\mbox{}\phantom{+}\nbar q\,|D_{k,n}|\,\Big|\frac{1}{|D_{k,n}|}\int_{D_{k,n}}\!\phi_i(x)\,\phi_j(x)\,\dif x-\Bigl(\frac{1}{|D_{k,n}|}\int_{D_{k,n}}\!\phi_i(x)\,\dif x\Bigr)\Bigl(\frac{1}{|D_{k,n}|}\int_{D_{k,n}}\!\phi_j(x)\,\dif x\Bigr)\Big|\nonumber
\\ &&
\label{example_cov_op_term_IV_2}\\[1ex]
&&\hspace{-25pt}\mbox{}+\nbar q\,|D_{k,n}|\,\Big|\Bigl(1-\frac{\alpha_n |D_{k,n}|^2}{l(k,n)|D_{k,n}|}\Bigr)\Bigl(\frac{1}{|D_{k,n}|}\int_{D_{k,n}}\phi_i(x)\,\dif x\Bigr)\Bigl(\frac{1}{|D_{k,n}|}\int_{D_{k,n}}\phi_j(x)\,\dif x\Bigr)\Big|\,.\nonumber \\
&&
\label{example_cov_op_term_IV_1}
\end{eqnarray}
First of all, using Young's inequality we obtain for the second term the estimate
\begin{equation}\label{example_cov_op_term_IV_1_e}
\eqref{example_cov_op_term_IV_1}\,\leq\, \frac{\nbar q}{2}\,\Big|1-\frac{\ell_-(n)\,\nu_-(n)}{\ell_+(n)\,\nu_+(n)}\Big|\,\bigl(\|\phi_i^n\|_{L^2}^2+\|\phi_j^n\|_{L^2}^2\bigr)\,,
\end{equation}
which converges to zero for $n\to\infty$.

We next estimate the term \eqref{example_cov_op_term_IV_2}. Firstly, we note that as in part (a) we find using Poincar\'e's inequality an upper bound to the term
\begin{equation}\label{the_term_which_is_upper_bound}
|D_{k,n}|\,\Bigl|\frac{1}{D_{k,n}}\int_{D_{k,n}} \bigl(\phi_i(x)-\phi_j(x)\bigr)^2\,\dif x-\Bigl(\frac{1}{D_{k,n}}\int_{D_{k,n}}\phi_i(x)-\phi_j(x)\,\dif x\Bigr)^2\Big|
\end{equation}
and the upper bound is proportional to $\delta_+(n)^2$. Next, expanding the two squared terms in \eqref{the_term_which_is_upper_bound} we find using the reverse triangle inequality that the term \eqref{the_term_which_is_upper_bound} is an upper bound to
\begin{equation*}
\left.\begin{array}{l}
\displaystyle|D_{k,n}|\,\bigg|\,\Big|\frac{1}{D_{k,n}}\int_{D_{k,n}}\phi_i(x)^2\,\dif x+\frac{1}{D_{k,n}}\int_{D_{k,n}}\phi_j(x)^2\,\dif x\\[3ex]
\displaystyle\phantom{xxxxxxxxxxxxxxxxxxxxxxx}\mbox{}-\Bigl(\frac{1}{D_{k,n}}\int_{D_{k,n}}\phi_i(x)\,\dif x\Bigr)^2-\Bigl(\frac{1}{D_{k,n}}\int_{D_{k,n}}\phi_j(x)\,\dif x\Bigr)^2\Big|\\[4ex]
\displaystyle\phantom{xx}\mbox{}-2\Big|\frac{1}{|D_{k,n}|}\int_{D_{k,n}}\phi_i(x)\phi_j(x)\,\dif x - \Bigl(\frac{1}{D_{k,n}}\int_{D_{k,n}}\phi_i(x)\,\dif x\Bigr)\Bigl(\frac{1}{D_{k,n}}\int_{D_{k,n}}\phi_j(x)\,\dif x\Bigr)\Big|\,\bigg|\,.
\end{array}\right.
\end{equation*}
Thus also this term possesses an upper bound which is proportional to $\delta_+(n)^2$. For $n\to\infty$ the upper bound converges to zero. As for $\delta_+(n)\to 0$ also the term spanning the first and second line converges to zero which was established in (ii.1), necessarily also the term in the third line converges to zero. Therefore we infer that the term \eqref{example_cov_op_term_IV_2} converges to zero proportional to $\delta_+(n)^2$.

Now, a combination of these estimates yields analogously to \eqref{result_ii_1} in (ii.1) that
\begin{equation}\label{result_ii_2} 
\eqref{proof_compartmental_model_var_sum_2}\ \leq\ C_\Phi\,\Bigl(\|p_j-z^n_j(\theta^n)\|_{L^2}+ \|u-U^n\|_{L^2}+\delta(n)^2\|\nabla U^n\|_{L^1}+\delta^2(n)+\delta(n)+R(n)\Bigr)\,.
\end{equation}
Here $R(n)$ is a term converging to zero for $n\to\infty$ arising from \eqref{example_cov_op_term_IV_1_e} and it is of the same type as the term $R(n)$ in (ii.1). The deterministic constant $C_\Phi$ is independent of $n\in\mathbb{N}$ and depends on $\Phi$ via the norm in $H^s(D)$ of the components of $\Phi$.

(ii.3)\quad A combination of the final results \eqref{result_ii_1} and \eqref{result_ii_2} in (ii.1) and (ii.2) yields that there exists a constant $C_\Phi<\infty$ such that for almost all $t$ it holds that
\begin{eqnarray*}
\lefteqn{\Big|\langle\Phi,G(u(t),p(t))\,\Phi\bigr\rangle_\mathcal{E}-\alpha_n\bigl\langle\Phi,G^n(U^n_t,\Theta^n_t)\,\Phi\bigr\rangle_\mathcal{E}\Big|}\\
&\hspace{-15pt}\leq& \hspace{-10pt}C_\Phi\Bigl(\,\sum_{i=1}^m\|p_i(t)\!-\!z^n_i(\Theta^n_t)\|_{L^2}\!+\! \|u(t)\!-\!U^n(t)\|_{L^2}\!+\!\delta(n)^2\|\nabla U^n\|_{L^2}\!+\!\delta^2(n)\!+\!\delta(n)\!+\!R(n)\Bigr) 
\end{eqnarray*}
Here we have also employed the continuous embedding of $L^2(D)\hookrightarrow L^1(D)$. We next square both sides of this inequality and integrate over $(0,T)$. Afterwards we take the square root of the integral terms and further take the expectation of the resulting inequality. Finally, appropriate applications of Jensen's inequality yields that
\begin{eqnarray}\label{compartmental_model_CLT_proof_squared_term_estimate}
\lefteqn{\int_0^T\EX^n\Big|\langle\Phi,G(u(t),p(t))\,\Phi\bigr\rangle_\mathcal{E}-\alpha_n\bigl\langle\Phi,G^n(U^n_t,\Theta^n_t)\,\Phi\bigr\rangle_\mathcal{E}\Big|\,\dif t}\\
&\hspace{-15pt}\leq& \hspace{-10pt} C_{\Phi,T}\,\Bigl(\delta^2(n)\!+\!\delta(n)\!+\!R(n)\!+\!\EX^n\Bigl[\|u\!-\!U^n\|_{L^2((0,T),L^2)}\!+\!\sum_{i=1}^m\|p_i\!-\!z^n_i(\Theta^n)\|_{L^2((0,T),L^2)}\Bigr]\Bigr)\nonumber
\end{eqnarray}
for an appropriate constant $C_{T,\Phi}<\infty$. Note that in order to arrive at the estimate \eqref{compartmental_model_CLT_proof_squared_term_estimate} we have further employed that the random term $\|\nabla U^n\|_{L^2((0,T),L^2)}$ can be estimated by a deterministic bound independent of $n\in\mathbb{N}$ due to Proposition \ref{proposition_in_examepl_proof} (b). Finally, due to the law of large numbers, i.e., Theorem \ref{LLN_compartmental_models}, the sequence of PDMPs converges to the deterministic limit in the mean. Hence the expectation in the right hand side in \eqref{compartmental_model_CLT_proof_squared_term_estimate} converges to zero for $n\to\infty$. Furthermore, $\delta_+(n)$ converges to zero by assumption \eqref{LLN_compartmental_model_part_conds}, as does the term $R(n)$. Thus, overall the right hand side in \eqref{compartmental_model_CLT_proof_squared_term_estimate} converges to zero. The convergence of the quadratic variation is proved which completes the proof of Theorem \ref{compartmental_example_cov_operators}.